\documentclass[10pt, reqno]{amsart}

\usepackage{amsmath, amsthm, amssymb}
\usepackage{setspace}
\usepackage[english]{babel}
\usepackage{mathrsfs}
\usepackage[all]{xy}
\usepackage{paralist}
\usepackage{graphicx}
\usepackage{accents}
\usepackage{enumerate}
\usepackage{mathtools}

\usepackage{csquotes}

\allowdisplaybreaks

\usepackage{hyperref}
\usepackage[dvipsnames]{xcolor}
\hypersetup{
	colorlinks=true,
	citecolor=blue!60!black,
	linkcolor=red!60!black,
	urlcolor=green!40!black,
	filecolor=yellow!50!black,
	breaklinks=true,
	pdfpagemode=UseNone,
	bookmarksopen=false,
}

\usepackage[normalem]{ulem}
\usepackage{soul}
\usepackage{cancel}

\makeatletter
\newtheorem*{rep@theorem}{\rep@title}
\newcommand{\newreptheorem}[2]{%
\newenvironment{rep#1}[1]{%
 \def\rep@title{#2 \ref{##1}}%
 \begin{rep@theorem}}%
 {\end{rep@theorem}}}
 
\makeatother

\theoremstyle{plain}
\newtheorem{theorem}{\normalfont \scshape Theorem}[section]
\newtheorem{lemma}[theorem]{\normalfont \scshape Lemma}
\newtheorem{proposition}[theorem]{\normalfont \scshape Proposition}
\newtheorem{corollary}[theorem]{\normalfont \scshape Corollary}

\newreptheorem{theorem}{Theorem}
\newreptheorem{proposition}{Proposition}

\theoremstyle{definition}
\newtheorem{definition}[theorem]{\normalfont \scshape Definition}

\theoremstyle{remark}
\newtheorem{remark}[theorem]{\normalfont \scshape Remark}

\renewenvironment{proof}[1][\proofname]{\noindent{\scshape #1.\quad}}{\qed}

\numberwithin{equation}{section}

\def\XXint#1#2#3{{\setbox0=\hbox{$#1{#2#3}{\int}$ }
\vcenter{\hbox{$#2#3$ }}\kern-.6\wd0}}

\newcommand{\bfc}{{\bf c}}

\newcommand{\bfj}{{\bf j}}

\newcommand{\bfr}{{\bf r}}

\newcommand{\bfB}{{\bf B}}

\newcommand{\bbI}{\mathbb I}

\newcommand{\bbN}{\mathbb N}

\newcommand{\bbGamma}{\reflectbox{\rotatebox[origin=c]{180}{$\mathbb L$}}}

\newcommand{\calD}{\mathcal D}

\newcommand{\calH}{\mathcal H}

\newcommand{\calL}{\mathcal L}

\newcommand{\calO}{\mathcal O}

\newcommand{\calT}{\mathcal T}

\newcommand{\scrL}{\mathscr L}

\newcommand{\scrR}{\mathscr R}
\newcommand{\scrX}{\mathscr X}

\newcommand{\bp}{\begin{pmatrix}}
\newcommand{\ep}{\end{pmatrix}}
\newcommand{\p}{\partial}

\newcommand{\N}{\mathbb{N}}
\newcommand{\R}{\mathbb{R}}

\newcommand{\T}{\mathbb{T}}

\newcommand{\ve}{\varepsilon}

\newcommand{\dx}{\textnormal{d}x}

\newcommand{\dv}{\textnormal{d}v}
\newcommand{\dw}{\textnormal{d}w}
\newcommand{\dt}{\textnormal{d}t}
\newcommand{\ds}{\textnormal{d}s}

\newcommand{\ddt}{\frac{\textnormal{d}}{\textnormal{d}t}}
\newcommand{\tnd}{\textnormal{d}}

\newcommand{\weakto}{\rightharpoonup}
\newcommand{\weakstarto}{\overset{\star}{\rightharpoonup}}

\newcommand{\lt}{\left}
\newcommand{\rt}{\right}

\newcommand{\inttr}{\int_{\Omega\times\R^d}}

\newcommand{\intr}{\int_{\R^d}}

\newcommand{\renorm}[3]{\lt[#1\rt]_{#2 , #3}}

\date{\today}

\begin{document}

\title[Nonlinear kinetic Fokker--Planck equations in bounded domains]{Global weak solutions to nonlinear kinetic Fokker--Planck equations in bounded domains under physical initial data}

\makeatletter
\newcommand\thankssymb[1]{\textsuperscript{\@fnsymbol{#1}}}
\makeatother

\author[Y.-P. Choi]{Young-Pil Choi\thankssymb{1}}
\thanks{\thankssymb{1} Department of Mathematics, Yonsei University, Seoul 03722, Republic of Korea (ypchoi@yonsei.ac.kr)}

\author[S. Song]{Sihyun Song\thankssymb{2}}
\thanks{\thankssymb{2} Department of Mathematics, Yonsei University, Seoul 03722, Republic of Korea (ssong@yonsei.ac.kr)}

\date{\today}
\thanks{ }

\begin{abstract}
We establish the global existence of weak solutions to a nonlinear kinetic Fokker--Planck equation with degenerate diffusion, under either inflow or partial absorption-reflection boundary conditions. The novelty of our approach lies in constructing solutions under solely the physical assumptions on the initial and boundary data, namely finite mass, kinetic energy, and entropy, with no additional regularity imposed. To overcome the lack of uniform ellipticity, we develop a compactness principle based on weighted Fisher information, which yields strong $L^1$ convergence of approximate solutions. This framework provides a robust existence theory under only the physically relevant conditions, and applies uniformly to both inflow and partial absorption-reflection boundary settings.
\end{abstract}

\maketitle
\tableofcontents

\singlespacing
%
%
%
%
%
%
\section{Introduction}
The main purpose of this work is to investigate nonlinear kinetic Fokker--Planck equations in a bounded spatial domain $\Omega \subset \R^d$ with inflow or partial absorption-reflection boundary conditions. Our focus is on the initial-boundary value problem and the compactness properties of solutions, established under only the physically relevant assumptions of finite mass, energy, and entropy. 

We consider a particle distribution function $f = f(t,x,v) \geq 0$, depending on time $t \geq 0$, position $x \in \Omega$, and velocity $v \in \R^d$. From $f$, one naturally associates the macroscopic quantities: the local density $\rho_f$, mean velocity $u_f$, and temperature $\calT_f$, defined by
\begin{align*}
    \rho_f := \intr f\,\dv, \quad u_f := \begin{cases}
        0 & \rho_f = 0, \\
        \frac{1}{\rho_f}\intr vf\,\dv & \rho_f \ne 0,
    \end{cases}
    \quad \calT_f := \begin{cases}
        0 & \rho_f = 0, \\
        \frac{1}{d\rho_f}\intr |v-u_f|^2 f \,\dv & \rho_f \ne 0.
    \end{cases}
\end{align*}

With these notations, the nonlinear kinetic Fokker--Planck equation under study takes the form
\begin{align}\label{eq: main}
    \begin{cases}
    \p_t f + v\cdot \nabla_x f = \nu_f \nabla_v\cdot \Big( \calT_f \nabla_v f + (v- u _f) f \Big) \quad \text{in} \quad (0,T)\times \Omega\times \R^d,\\
    f|_{t=0} = f^0,
    \end{cases}
\end{align}
where  $\nu_f$ denotes the collision frequency, which we assume throughout this work satisfies
\begin{equation} \label{ASSUMP NU}
    \begin{cases}
        \textnormal{$\nu_f = \nu(\rho_f, j_f, V_f)$ where $j_f := \rho_f u_f$ and $V_f := \rho_f \calT_f$},\\
        \nu\in C^0 \cap L^\infty(\R_+ \times \R^d \times \R_+ ; \R_+),\\
        \textnormal{either $\nu > 0$, or $\nu(\rho,j,V) \ge 0$ with equality if and only if $\rho = 0$}.
        \end{cases}
\end{equation}
This covers the cases of, for instance
\begin{align*}
    \nu(\rho, j, V) = 1, \; \frac{\rho}{1 + \rho}, \; \frac{\rho^{\alpha} \calT^{\beta}}{1 + \rho^{\alpha} \calT^{\beta}},\ \ldots .
\end{align*}

The study of generalized Fokker--Planck type equations can be traced back to the classical works of \cite{kirkwood1946, Green1952}, where they were introduced to describe the evolution of dilute gases. Variants of these models also arise in the context of inelastic granular flows \cite{BDS99}. Despite this long-standing interest, the mathematical analysis of the Cauchy problem \eqref{eq: main} has remained incomplete for decades, with major obstacles stemming from the nonlinear structure of the collision operator.

In the linear Fokker--Planck case, which corresponds to \eqref{eq: main} with $\nu \equiv \calT_f \equiv 1$ and $u_f \equiv 0$, a comprehensive theory has been developed even in the presence of Poisson interactions. Classical and weak solutions in the whole space have been extensively investigated; see, for instance \cite{Bou93, Deg86, HJ13, RW92, VO90} and \cite{CS95, CS97, Cas98, Vic91}, respectively, as well as the smoothing effects described in \cite{Bou95}. For initial-boundary value problems, significant progress has also been achieved in \cite{BCS97, Carrillo98, HK19, MH22, Mischler2010}, covering various types of boundary mechanisms including specular and diffusive reflections. 

Nonlinear Fokker--Planck equations, where the diffusion and drift coefficients depend on the solution itself, remain much less understood. Existing results are mostly confined to perturbative regimes near global Maxwellians, where the nonlinearity can be treated as a small perturbation. In this framework, global strong solutions have been obtained, for instance, in \cite{Bed17, LY21}. Further developments for the ellipsoidal Vlasov--Fokker--Planck model, designed to capture correct Prandtl numbers, were carried out in \cite{MM17, SJ19} and subsequent works. Numerical studies have also addressed the nonlinear Fokker--Planck dynamics, particularly in the design of asymptotic-preserving schemes and the investigation of long-time behavior \cite{AC23, FN22}. We also mention \cite{DFT10}, where a kinetic Cucker--Smale model with a different nonlinear structure, including density-dependent diffusion, is investigated in a perturbative regime near a global Maxwellian, and long-time stability and decay estimates are established. These contributions provide a rich perturbative theory for strong solutions near equilibrium, but they rely essentially on the smallness of perturbations and do not extend to large-data weak solutions.

Indeed, the theory of weak solutions for nonlinear Fokker--Planck operators remains largely unexplored. To the best of our knowledge, previous works concerning \eqref{eq: main} have been restricted to the case of constant collision frequency $\nu \equiv 1$, and to the spatially periodic setting $\T^d \times \R^d$. The first rigorous treatment of the Cauchy problem for \eqref{eq: main} in this framework was given in \cite{CG04} (see also \cite{CG03}). Their approach relied on two fundamental structural features of \eqref{eq: main}. The first is that the nonlinear Fokker--Planck equation can be interpreted as a gradient flow of the relative entropy with respect to the local Maxwellian. This insight becomes clear once one observes that
    \begin{align*}
        \nabla_v\cdot \lt(\calT_f \nabla_v f + (v-u_f) f \rt) = \nabla_v \cdot \lt(f \calT_f \nabla_v \log \frac{f}{M[f]} \rt), \quad M[f] := \frac{\rho_f}{(2\pi \calT_f)^{d/2}} e^{-\frac{|v-u_f|^2}{2\calT_f}}.
    \end{align*}
The second feature is the transport structure in the $(t,x)$ variables. By combining both perspectives into a discrete splitting scheme, Carlen and Gangbo constructed approximate solutions and proved their (weak) convergence to a weak solution of \eqref{eq: main}. A key requirement in their argument was the set of strong a priori assumptions
    \begin{align}\label{eq: CG}
        \int_{\T^d\times\R^d} (1 + |v|^6 + |\log f^0|) f^0 \,\dx\dv < +\infty, \qquad \int_{\T^d\times\R^d} |f^0|^2 \,\dx\dv < +\infty,
    \end{align}
which played a crucial role in displacement convexity estimates controlling entropy dissipation.
    
More recently, in the same periodic setting with $\nu \equiv 1$, the Cauchy problem for \eqref{eq: main} was resolved in \cite{CHY25}. Their analysis was carried out under a different set of assumptions, namely
    \begin{align}\label{eq: CHY}
        \int_{\T^d\times\R^d} (1+|v|^3+|\log f^0|) f^0 \,\dx\dv < +\infty, \qquad \|f^0\|_{L^\infty(\T^d\times\R^d)} < +\infty.
    \end{align}
Compared with the earlier work \cite{CG04}, the requirement on the initial velocity moments is significantly relaxed (from sixth-order moments down to third-order). On the other hand, the condition that $f^0$ be a member of $L^\infty$ is substantially stronger. To be more precise, following \cite{CHY25}, one can observe that the third-moment condition can be replaced by the weaker assumption
    \begin{align*}
        \int_{\T^d\times\R^d} |v|^{2+\ve} f^0 \,\dx\dv < +\infty \quad \textnormal{for some $\ve>0$},
    \end{align*}
but both the boundedness condition on $f^0$ and the requirement of moments beyond the energy level appear essential in their approach, and cannot be entirely removed.

 In this work, we extend the framework of nonlinear Fokker--Planck theory to general bounded domains and physically admissible initial data. Namely, by developing a compactness method, based on a weighted Fisher information, we tackle the existence theory for \eqref{eq: main} under \textbf{solely the physical assumptions} on the initial data:
\begin{align} \label{init phys}
        f^0\ge 0, \quad \int_{\Omega\times\R^d} (1+|v|^2+|\log f^0|) f^0 \,\dx\dv < +\infty.    
\end{align}
The condition \eqref{init phys} ensures that the total mass, kinetic energy, and entropy are finite at $t=0$. Our main result is twofold: we construct solutions to \eqref{eq: main} under no additional hypotheses beyond \eqref{init phys}, and we establish strong $L^1$ compactness of approximating sequences under weighted Fisher-information bounds. Compared with \cite{CG04,CHY25}, this compactness is obtained under weaker assumptions on the initial data.

To clarify the novelty, it is instructive to first recall the linear Fokker--Planck equation,
\begin{align}\label{eq: LINFP}
    \p_t f + v\cdot \nabla_x f  = \Delta_v f + \nabla_v\cdot (vf).
\end{align}
The diffusion in \eqref{eq: LINFP} endows the system with significant regularization; indeed, since the seminal works of H\"ormander \cite{H67} and Bouchut \cite{Bou02}, the hypoelliptic operator $\calL f := \Delta_v f + \nabla_v\cdot (vf)$ is known to provide smoothing effects in all variables $(t,x,v)$ (see also \cite{HN04, LL08}). Consequently, if $\{f_m\}$ is a sequence of smooth solutions to \eqref{eq: LINFP} with uniform bounds of the type \eqref{init phys}, then ${f_m}$ converges \textbf{strongly} in $L^1((0,T)\times\Omega\times\R^d)$ to a limit $f$, which is again a solution to \eqref{eq: LINFP}. This strong compactness mechanism was crucially exploited in \cite{DL88FKB} to construct solutions to the Fokker--Planck--Boltzmann equation, where the nonlinearity stems from the Boltzmann collision operator but the diffusion remains linear.

However, the approach in \cite{DL88FKB} relies in an essential way on the Green's function of the linear operator $-\Delta_v$, and thus cannot be directly applied when the diffusion coefficient depends nonlinearly on the solution. It is therefore natural to expect difficulties when treating nonlinear equations of the type \eqref{eq: main}, or even simpler variants such as
\begin{align}\label{eq: FPT}
    \p_t f + v\cdot \nabla_x f = \calT_f \Delta_v f.
\end{align}
Suppose we attempt to construct solutions to \eqref{eq: FPT} by approximating with a sequence $\{f_m\}$ of smooth solutions to a regularized system. Under the basic assumption \eqref{init phys}, one of the few a priori bounds available (disregarding boundary contributions) is the entropy dissipation estimate
\[
    \int_{\Omega\times\R^d} f\log f \,\dx\dv + 4 \int_{(0,T)\times\Omega\times\R^d} \calT_f |\nabla_v \sqrt{f}|^2 \,\dt\dx\dv \le \int_{\Omega\times\R^d} f^0 \log f^0 \,\dx\dv.
\]
This yields the uniform control
\[
    \sup_m \int_{(0,T)\times\Omega\times\R^d} \calT_{f_m} \lt|\nabla_v \sqrt{f_m} \rt|^2 \dt\dx\dv < C < +\infty,
\]
which may be viewed as a ``nonlinearly weighted Fisher information.''

The central question we address in this work is whether such a bound is sufficient to recover, in the nonlinear setting, a phenomenon analogous to the linear case:
\begin{quote}
Can one deduce that the sequence $\{f_m\}$ is \textbf{strongly compact} in $L^1((0,T)\times\Omega\times\R^d)$?
\end{quote}
 
We show that under mild structural assumptions on the diffusion coefficient (in the model case $\calT_{f_m}$) and on the velocity averages of $f_m$, one can indeed recover strong convergence: $f_m \to f$ almost everywhere and in $L^1((0,T)\times \Omega \times \R^d)$. These assumptions are not artificial; they arise naturally in the construction of approximate solutions to the Cauchy problem for \eqref{eq: main} via regularization schemes. Formulated precisely in Proposition \ref{prop: weighted fisher}, this compactness principle establishes strong $L^1$ compactness for nonlinear kinetic Fokker--Planck operators under weighted Fisher information bounds. It is flexible enough to handle merely measurable diffusion profiles and provides the essential compactness tool for treating bounded domains and rough initial data.

\noindent \textbf{Informal compactness principle.}
Under the assumption of bounded mass, energy, and entropy, we demonstrate that approximate solutions enjoy
\emph{strong $L^1$ compactness}, driven through the weighted Fisher information bound
together with mild hypotheses on the velocity-averages. Roughly speaking: if $\{f_m\}$ is weakly compact in $L^1$, velocity averages of the form $\lt\{\int f_m \psi(v)\,\dv\rt\}$ are compact in $L^1_{t,x}$, and the weighted Fisher information
\[            
 \int_{(0,T)\times\Omega\times\R^d} a_m(t,x) \lt|\nabla_v \sqrt{f_m} \rt|^2 \,\dt\dx\dv  
\]
remains uniformly bounded, then provided that $a_m\to a \ge 0$ pointwise, we can obtain $f_m \to f$ almost everywhere and in $L^1$, under a certain condition on the support of $a$. The precise formulation and proof of this result are given in Proposition \ref{prop: weighted fisher} of Section \ref{sec: comp}.

As a consequence, we are able to solve the initial-boundary value problem associated with \eqref{eq: main} under \textbf{only} the assumption \eqref{init phys}, up to a mild hypothesis on the regularity of the spatial domain $\Omega$ and appropriately prescribed boundary conditions. This compactness principle allows us to extract limits of approximate solutions in the strong topology of $L^1$, thereby overcoming the lack of classical regularity and considerably streamlining the overall argument. A detailed outline of the proof strategy will be presented in Section \ref{sec: strat}.

%
%
%
%
%
%
\subsection{Main results}
We are now in a position to state our main existence theorems. Two natural boundary-value problems are treated: the inflow problem and the partial absorption-reflection problem. Both are established under only the physically relevant assumptions on the initial and boundary data.  

Our first result addresses the inflow problem associated with \eqref{eq: main}, where boundary values are prescribed on the incoming set of characteristics.  We adopt the following nomenclature for $t\in [0,T]$: 
\begin{align*}
    &\Sigma_\pm^t := \{(s,x,v)\in [0,t]\times \p\Omega\times\R^d : \pm (n(x)\cdot v) > 0 \}, \\
    &\Sigma_0^t := \{(s,x,v) \in [0,t]\times \p\Omega \times \R^d : (n(x)\cdot v) = 0 \}, \\
    &\Sigma^t := \Sigma_+^t \cup \Sigma_-^t \cup \Sigma_0^t,
\end{align*}
and denote by $\tnd\sigma = \tnd\sigma(x)$ the Hausdorff measure on $\p\Omega$.

\begin{theorem}\label{thm: inflow}
    We assume that $\Omega\subset \R^d$ is a bounded open set with $C^1$ boundary and outward unit normal $n(x)$. Suppose the collision frequency $\nu$ satisfies the assumptions in \eqref{ASSUMP NU}. Let $f^0$ be given with only the physically relevant assumptions \eqref{init phys}. Suppose $g$ is such that
    \begin{align*}
        g\ge 0, \quad \int_{(0,T)\times \p\Omega \times \R^d} (1 + |v|^2 + |\log g|) g \,|n(x)\cdot v| \,\dt\tnd\sigma\dv < +\infty.
    \end{align*}
    Then for any finite time horizon $T>0$, there exists a weak solution $(f,\gamma f)$ to the following inflow problem associated to \eqref{eq: main},
    \begin{align} \label{eq: inflow}
        \begin{cases}
            \p_t f + v\cdot \nabla_x f = \nu_f \nabla_v\cdot \Big( \calT_f \nabla_v f + (v- u _f) f \Big) \quad \text{in} \quad (0,T)\times \Omega\times\R^d, \\
            f|_{t=0} = f^0, \\
            \gamma f|_{\Sigma_-^T} = g,
        \end{cases}
    \end{align}
    in the sense that the following properties are satisfied:
    \begin{enumerate}[(I)]
    \item The solution satisfies
    \begin{equation*}
    \begin{split}
        &(1+|v|^2)f \in L^\infty(0,T;L^1(\Omega\times\R^d)), \quad f \in L^\infty(0,T;L\log L(\Omega\times\R^d)), \\
        &(1+|v|^2)\gamma f \in L^1(\Sigma^T, |n(x)\cdot v|\dt\tnd\sigma\dv), \quad \gamma f \in L\log L(\Sigma^T, |n(x)\cdot v|\dt\tnd\sigma\dv).
    \end{split}
    \end{equation*}
    \item The following very weak Green's formula holds for all $t\in [0,T]$ and $\varphi\in C_c^\infty([0,t]\times\overline{\Omega}\times\R^d)$
    \begin{equation} \label{very weak}
    \begin{split}
        &\int_{\{t\}\times\Omega\times\R^d} f \varphi \,\dx\dv - \int_{\Omega\times\R^d} f^0 \varphi(0,\cdot,\cdot) \,\dx\dv + \int_{\Sigma_+^t} \gamma f \,\varphi \, (n(x)\cdot v)_+ \,\ds\tnd\sigma\dv \\
        &\quad = -\int_{\Sigma_-^t} g \varphi (n(x)\cdot v)_- \,\ds\tnd\sigma\dv   + \int_{(0,t)\times \Omega\times\R^d} f \nu_f \calT_f \Delta_v \varphi - \nu_f \nabla_v\varphi\cdot (v-u_f) f \; \ds\dx\dv.
    \end{split}
    \end{equation}
    \item The weak gradient $\nabla_v f$ is well-defined as a measurable function on $(0,T)\times\Omega\times\R^d$, and
    \begin{align*}
        \int_{(0,T)\times \Omega\times \R^d} \nu_f \calT_f |\nabla_v \sqrt{f}|^2 \,\dt\dx\dv < +\infty.
    \end{align*}
    Correspondingly, the weak Green's formula also holds, in other words we can replace the term $\int f \nu_f \calT_f \Delta_v \varphi$ of \eqref{very weak} with $\int - \nu_f \calT_f \nabla_v f \cdot \nabla_v\varphi$. 
    \item The energy inequality below holds for all $t\in (0,T]$
    \begin{equation*}
    \begin{split}
        &\int_{\{t\}\times \Omega\times \R^d} (1+|v|^2) f \,\dx\dv - \int_{\Omega \times \R^d} (1+|v|^2)f^0 \,\dx\dv  + \int_{\Sigma_+^t} (1+|v|^2) \gamma f(n(x)\cdot v)_+ \,\ds\tnd\sigma\dv \\
        &\quad \le \int_{\Sigma_-^t} (1+|v|^2) g\,|n(x)\cdot v| \,\ds\tnd\sigma\dv.
    \end{split}
    \end{equation*}
    \item The entropy inequality below holds for all $t\in (0,T]$
    \begin{equation*}
    \begin{split}
        &\int_{\{t\}\times\Omega\times\R^d} f\log f \,\dx\dv + \int_{\Sigma_+^t} \gamma f \log \gamma f \,(n(x)\cdot v)_+ \,\ds\tnd\sigma\dv \\
        &\quad + \int_{(0,t)\times \Omega\times\R^d}  \frac{\nu_f}{f \calT_f}|\calT_f \nabla_v f + (v-u_f)f|^2 \,\ds\dx\dv \\
        &\le  \int_{\Sigma_-^T} g\log g |n(x)\cdot v|\,\dt\tnd\sigma\dv +  \int_{\Omega\times\R^d} f^0 \log f^0 \,\dx\dv.
    \end{split}
    \end{equation*}
    \end{enumerate}
\end{theorem}

\begin{remark}
Several remarks are in order.
\begin{enumerate}
    \item Our result substantially advances the Cauchy theory for \eqref{eq: main}, showing that \eqref{init phys} suffices for weak solutions (compare with \eqref{eq: CG} and \eqref{eq: CHY}).
    \item Assuming additional spatial moment bounds on the initial and boundary datum, the result extends to unbounded domains, including $\Omega = \R^d$. The adaptations are straightforward.
    \item By $L\log L$ we refer to the (equivalence class of) measurable functions for which $\int |f|(1 + |\log f|) < +\infty.$
    \item For a given $f$, the function $\gamma f$ satisfying \eqref{very weak} is unique, and we refer to it as the \textbf{trace} of $f$. However, due to the lack of regularity, it is unclear whether a bounded trace operator exists in the corresponding function spaces (see \cite[Section 1.2.2]{Zhu24}). For this reason, we will always define a solution as a pair $(f,\gamma f)$ which satisfies the Green's formula. Our notion of weak solution formulated through the Green's formula follows the framework of \cite{Mischler2010,Zhu24} and explicitly incorporates both temporal and boundary traces. For the uniformly elliptic linear problems with bounded coefficients arising in our regularization procedure, this formulation agrees with the $L^2$-based weak formulation of \cite{Zhu24}, for which uniqueness is available. However, even in the linear case, the uniqueness result of \cite{Zhu24} does not directly apply to the limiting solution class considered here, which is based only on finite mass, energy, and entropy. Thus, uniqueness in the full physical-data solution class of Theorem \ref{thm: inflow} is not a consequence of the existing $L^2$ theory. 
    \item We clarify the meaning of the weak gradient $\nabla_v f$ appearing in part (III). Set
\[
a_f(t,x):=\nu_f(t,x)\calT_f(t,x).
\]
First, we observe that
\[
a_f(t,x)>0 \quad\Longleftrightarrow\quad \rho_f(t,x)>0
\]
for almost every $(t,x)$. Indeed, if $\rho_f(t,x)>0$, then $\calT_f(t,x)>0$, since otherwise the nonnegative function $f(t,x,\cdot)$ would be supported on the single point ${u_f(t,x)}$ and hence would vanish almost everywhere. The claim then follows from \eqref{ASSUMP NU}.

By part (III) and Fubini's theorem, for almost every $(t,x)$ such that $\rho_f(t,x)>0$, we have
\[
\sqrt{f(t,x,\cdot)}\in H^1(\R^d).
\]
This gives
\[
f(t,x,\cdot)\in W^{1,1}(\R^d), \quad \nabla_v f = 2\sqrt f \nabla_v\sqrt f.
\]
On the vacuum set ${\rho_f=0}$, we have $f(t,x,\cdot)=0$ almost everywhere, and we set $\nabla_v f=0$. This defines $\nabla_v f$ as a measurable, slice-wise weak $v$-gradient on $(0,T)\times\Omega\times\R^d$.

Moreover, using the energy bound and the boundedness of $\nu$, we have
\[
\int_{(0,T)\times\Omega\times\R^d} \nu_f\calT_f f\,\dt\dx\dv = \int_0^T\int_\Omega \nu_f\rho_f\calT_f\,\dx\dt  \le \frac{\|\nu\|_{L^\infty}}{d} \int_0^T\int_{\Omega\times\R^d} |v|^2f\,\dx\dv\dt <+\infty.
\]
Consequently, the Cauchy--Schwarz inequality yields
\[\begin{aligned}
&\int_{(0,T)\times\Omega\times\R^d} \nu_f\calT_f|\nabla_v f|\,\dt\dx\dv \\
&\quad \le 2 \lt( \int_{(0,T)\times\Omega\times\R^d} \nu_f\calT_f f\,\dt\dx\dv \rt)^{1/2} \lt( \int_{(0,T)\times\Omega\times\R^d} \nu_f\calT_f|\nabla_v\sqrt f|^2 \,\dt\dx\dv \rt)^{1/2} <+\infty.
\end{aligned}\]
Thus, the weighted flux $\nu_f\calT_f\nabla_vf$ belongs to $L^1$, and the weak Green's formula in part (III) is well defined. We note that this argument
provides the integrability of the weighted gradient $\nu_f\calT_f\nabla_v f$, rather than that of $\nabla_v f$ itself in all variables.
\end{enumerate}
\end{remark}

It remains an open question whether the weak solutions obtained in Theorem \ref{thm: inflow} are renormalized in the sense of DiPerna--Lions \cite{DipernaLionsODE1989}. In related contexts such as alignment models \cite{shvydkoy2025}, every weak solution was shown to be renormalized, but the argument there relies on a "thickness condition" for the diffusion coefficient, a structural property not present in our model. Establishing renormalization in the present setting is therefore an interesting and nontrivial direction for future research.

Using the same techniques, we can also address the case where both the collision frequency and the temperature are constant, in which case the model reduces to
\begin{align}\label{eq: const temp}
    \p_t g + v\cdot \nabla_x g = \Delta_v g + \nabla_v\cdot((v-u_g)g).
\end{align}
For this equation, global weak solutions were obtained in \cite{KMT2015} assuming initial data of class $L^\infty$. Classical solutions to \eqref{eq: const temp} within a perturbative framework were constructed in \cite{Choi16}, while the incompressible Euler and Navier--Stokes limits were subsequently established in \cite{CJ24, CJ26}. In contrast, our compactness method (see in particular Lemma \ref{lem: key compactness}) provides weak solutions under only the physically relevant assumptions \eqref{init phys}.
 
In addition, if the collision frequency $\nu$ fails to satisfy assumption \eqref{ASSUMP NU}, for example when $\nu(\rho,j,V)=\rho$, the weak formulation of \eqref{eq: main} may not even be well defined. Nevertheless, related models have been investigated in \cite{ImbertMouhot21, AnceschiZhu24}, where global solutions were obtained for
\[
    \p_t h + v\cdot \nabla_x h = \rho_h(\Delta_v h + \nabla_v\cdot(vh)),
\]
under Gaussian-type bounds on the initial data. The crucial observation there is that such Gaussian control propagates in time thanks to a \textbf{maximum principle}, which ensures $\rho_h$ remains uniformly bounded. This effectively reduces the situation to one with bounded collision frequency, bringing it back within the scope of assumptions analogous to \eqref{ASSUMP NU}.

Our second main theorem concerns boundary conditions of absorption-reflection type. In kinetic theory, such laws arise naturally when part of the incoming flux is reflected specularly while the rest is absorbed. This setting generalizes the inflow problem treated in Theorem \ref{thm: inflow}. Under the same physically relevant assumptions on the data, we again construct global weak solutions.

\begin{theorem}\label{thm: reflection}
    We assume that $\Omega$, $\nu$, $f^0$, and $T$ satisfy the assumptions of Theorem \ref{thm: inflow}. Let $\theta \in [0,1)$. Then the partial absorption-reflection problem
    \begin{align*}
        \begin{cases}
            \p_t f + v\cdot \nabla_x f = \nu_f \nabla_v\cdot \Big( \calT_f \nabla_v f + (v- u _f) f \Big) \quad \text{in} \quad (0,T)\times \Omega\times\R^d, \\
            f|_{t=0} = f^0, \\
        \gamma f(t,x,v) = \theta \gamma f (t,x,L_x v), \quad \forall (t,x,v) \in \Sigma_-^T, \\
        L_x (v) := v - 2(n(x)\cdot v) n(x)
        \end{cases}
    \end{align*}
    admits a weak solution $(f,\gamma f)$ in the sense that
    \begin{enumerate}[(I)]
        \item The solution satisfies
        \begin{equation*}
    \begin{split}
        &(1+|v|^2)f \in L^\infty(0,T;L^1(\Omega\times\R^d)), \quad f \in L^\infty(0,T;L\log L(\Omega\times\R^d)), \\
        &(1+|v|^2)\gamma f \in L^1(\Sigma^T_+, |n(x)\cdot v|\dt\tnd\sigma\dv), \quad \gamma f \in L\log L(\Sigma_+^T, |n(x)\cdot v|\dt\tnd\sigma\dv).
    \end{split}
    \end{equation*}
        \item The following very weak Green's formula holds for all $t\in [0,T]$ and $\varphi\in C_c^\infty([0,t]\times\overline{\Omega}\times\R^d)$:
        \begin{align*}
            &\int_{\{t\}\times \Omega\times\R^d} f\varphi\,\dx\dv - \int_{\Omega\times\R^d} f^0 \,\varphi(0,\cdot,\cdot,)\,\dx\dv\\
            &\quad + \int_{\Sigma_+^t} \gamma f(s,x,v)\,[\varphi(s,x,v) - \theta \varphi(s,x,L_x^{-1}v)] \,(n(x)\cdot v)_+ \,\ds\tnd\sigma\dv  \\
            &= \int_{(0,t)\times\Omega\times\R^d} f\nu_f \calT_f \Delta_v \varphi - \nu_f \nabla_v\varphi \cdot (v-u_f)f \,\ds\dx\dv.
        \end{align*}
        \item The weak gradient $\nabla_v f$ is well-defined as a measurable function on $(0,T)\times\Omega\times\R^d$, and
        \begin{align*}
            \int_{(0,T)\times \Omega\times \R^d} \nu_f \calT_f |\nabla_v \sqrt{f}|^2 \,\dt\dx\dv < +\infty.
        \end{align*}
        In particular, in similar manner as Theorem \ref{thm: inflow}, the very weak formulation can be replaced with the weak formulation.
        \item The energy inequality below holds
        \begin{align*}
            \int_{\{t\}\times\Omega\times\R^d} (1+|v|^2) f\,\dx\dv + (1-\theta) \int_{\Sigma_+^t} (1+|v|^2) \gamma f \,(n(x)\cdot v)_+ \,\ds\tnd\sigma\dv  \le \int_{\Omega\times\R^d} (1+|v|^2) f^0\,\dx\dv .
        \end{align*}
        \item The entropy inequality below holds
        \begin{align*}
            &\int_{\{t\}\times\Omega\times\R^d} f\log f \,\dx\dv + (1-\theta) \int_{\Sigma_+^t} \gamma f \log \gamma f \,(n(x)\cdot v)_+ \,\ds\tnd\sigma\dv \\
        &\quad + \int_{(0,t)\times \Omega\times\R^d}  \frac{\nu_f}{f \calT_f}|\calT_f \nabla_v f + (v-u_f)f|^2 \,\ds\dx\dv \\
        &\le \int_{\Omega\times\R^d} f^0 \log f^0 \,\dx\dv.
        \end{align*}
    \end{enumerate}
\end{theorem}

\begin{remark}
Theorem \ref{thm: reflection} does not cover the case $\theta=1$, corresponding to purely specular reflection without absorption. This limitation reflects a well-known obstruction: see, for instance, \cite{GHJO20} for related discussions. In fact, when $\theta=1$ the dominance of grazing collisions is known to cause a loss of regularity in the trace $\gamma f$, and one cannot in general expect $\gamma f \in L^1(\Sigma_+^T,|n(x)\cdot v|\,\dt\tnd\sigma\dv)$. 
A formal way to observe this obstruction is through parts (IV) and (V) of Theorem \ref{thm: reflection}: when $\theta=1$, the boundary contributions vanish, leaving no available a priori control on the outgoing trace $\gamma f|_{\Sigma_+^T}$.
\end{remark}

\begin{remark}
A classical approach to reflection-type problems is to approximate them by a sequence of inflow problems, see for example \cite{Carrillo98, Hamdache92}. In this work, we adopt a different strategy. Since the theory of reflection problems has already been developed to a satisfactory level in the \textbf{linear} Fokker--Planck setting \cite[Section 2]{Zhu24}, we may rely on those results and approximate directly by reflection problems themselves, rather than reducing to the inflow case.
\end{remark}

%
%
%
%
%
%
\subsection{Strategy of proof}\label{sec: strat}

The proof of our main existence results proceeds through a classical approximation procedure, carefully adapted to the nonlinear diffusion structure of \eqref{eq: main} and to the lack of uniform ellipticity. The idea is to construct a family of approximate solutions, establish estimates uniform in the approximation parameter, and then pass to the limit by combining compactness arguments with stability properties of the nonlinear operator. Each step in the argument is designed to overcome a specific difficulty: solvability of the regularized problems, preservation of renormalization properties, derivation of uniform entropy bounds, compactness of velocity averages, and finally the identification of quadratic quantities such as the temperature.

\medskip

\noindent
\textbf{Step 1: Regularization and linearization.}
We begin by considering a regularized version of the inflow problem \eqref{eq: inflow}, where both the collision operator and the boundary data are smoothed. The regularization is designed so that the resulting system admits well-defined solutions for each fixed parameter $\ve>0$. At this stage we further linearize the problem, producing a linear kinetic Fokker--Planck equation with given coefficients. Existence of solutions to such linear problems is well established, and uniqueness was proved in \cite[Theorem 1.1]{Zhu24}. This uniqueness plays a central role in our approach, since it allows us later to view the nonlinear problem as a fixed point of a solution map.

\medskip

\noindent
\textbf{Step 2: Fixed-point construction.}
To recover the nonlinear structure, we rely on a fixed point argument, defining a suitable solution operator on the coefficient space. The uniqueness of the linear problem is instrumental in verifying the continuity of this operator (see \cite{Abdallah1994}). Moreover, one can check directly that the operator is compact. Then through Schaefer's fixed-point theorem, we obtain that this map admits a fixed point $(f_\ve, \gamma f_\ve)$ which is a solution to the nonlinear regularized problem. A key advantage of this construction is that $(f_\ve, \gamma f_\ve)$ inherits properties of the linear solutions, most importantly the renormalization property in the sense of DiPerna--Lions \cite{DipernaLionsODE1989}. This structural inheritance is crucial in the next step, where we prove uniform a priori estimates and apply an averaging lemma.

\medskip

\noindent
\textbf{Step 3: Uniform estimates.}
The renormalization property enables us to rigorously derive energy and entropy inequalities for $(f_\ve,\gamma f_\ve)$. These provide uniform control of the mass, the kinetic energy, and the entropy production, all independent of $\ve$. Such bounds already imply weak compactness of the sequence $\{f_\ve\}$ by the Dunford--Pettis criterion, so that we can pass to a weakly convergent subsequence. However, weak convergence of $(f_\ve,\gamma f_\ve)$ is not sufficient: to handle the nonlinear structure of \eqref{eq: main}, and to ensure that the energy and entropy inequalities persist in the limit (see \textbf{Step 6}), we aim to establish strong convergence of $f_\ve$. 
\medskip

\noindent
\textbf{Step 4: Velocity averages.}
We strive to demonstrate strong convergence of $f_\ve$, by verifying that these approximate solutions satisfy the hypotheses of the compactness principle stated in Proposition \ref{prop: weighted fisher}. A delicate part is showing convergence of the velocity averages $\lt\{\intr f_\ve \varphi(v) \dv\rt\}$ for each $\varphi\in C_c(\R^d)$. Standard velocity averaging lemmas are not directly applicable here: they require $L^p$ control with $p>1$, while in our setting we only have $L^1$ bounds. It is well known that averaging can fail in $L^1$ due to concentration phenomena along characteristics \cite[p. 123]{GLPS88}, \cite[p. 558]{GS02}. To overcome this, we exploit the renormalization property of $(f_\ve, \gamma f_\ve)$. By applying averaging lemmas to nonlinear transforms $\Gamma_\delta(f_\ve)$ with suitable decay, we obtain convergence of the averages $\{\int_{\R^d} \Gamma_\delta(f_\ve)\varphi\dv\}$ for each $\delta>0$. Then, using equiintegrability provided by the entropy bounds, we can pass from convergence of the averages of $\Gamma_\delta(f_\ve)$ back to convergence of the averages of $f_\ve$. Our procedure is greatly motivated by the works \cite{dipernalionsMaxwell1989, GolseSaintRaymond05}.

\medskip

\noindent
\textbf{Step 5: Higher moments and temperature convergence.}
Another obstacle is the identification of the temperatures $\calT_{f_\ve}$ in the limit. Strong convergence of the velocity averages is not enough to control these quadratic quantities. The crucial new ingredient is a \textbf{moment-gain lemma}, which shows that solutions to \eqref{eq: main} acquire additional velocity moments once integrated over time. This phenomenon was first observed for the pure transport equation in \cite{Perthame1992} and has since become a standard tool in kinetic theory. By adapting this technique we show that $\{|v|^3 f_\ve\}$ is uniformly bounded in $L^1$, which in turn ensures almost everywhere convergence of $\calT_{f_\ve}$ on the set where $\rho_f>0$. This suffices to identify the limiting temperature, and in particular, we obtain that all assumptions of Proposition \ref{prop: weighted fisher} hold.

\medskip

\noindent
\textbf{Step 6: Passage to the limit.}
Having verified that the approximate solutions satisfy the necessary compactness criteria, Proposition \ref{prop: weighted fisher} yields strong convergence $f_\ve \to f$ in $L^1((0,T)\times\Omega\times\R^d)$. With this convergence in hand, we can pass to the limit $\ve \downarrow 0$ in the nonlinear terms of the equation and recover a global weak solution to \eqref{eq: main}. It remains to verify that the energy and entropy inequalities also hold for $f$, by passing to the limit in those satisfied by $f_\ve$. Since these inequalities are formulated at fixed times $t\in[0,T]$, one typically needs to identify the weak limit of $f_\ve(t)$. This step is often delicate, as it usually requires establishing some temporal equicontinuity of ${f_\ve}$ (cf. \cite{GolseSaintRaymond05}). Here, however, the strong convergence guaranteed by Proposition \ref{prop: weighted fisher} allows us to bypass this difficulty entirely.  This concludes the proof strategy. The same reasoning applies to both types of boundary conditions considered in this work. For clarity we present the inflow case in full detail in Section \ref{sec: inflow}, while in Section \ref{sec: reflect} we indicate the minor adjustments required for partial absorption--reflection.

%
%
%
%
%
%

\subsection{Organization of the paper}

Section \ref{sec: comp} is devoted to the proof of our main compactness result, Proposition \ref{prop: weighted fisher}. 
In Section \ref{sec: inflow} we address the initial-boundary value problem with inflow data and establish Theorem \ref{thm: inflow}. 
Section \ref{sec: reflect} provides a sketch of the proof of Theorem \ref{thm: reflection}, whose detailed arguments are omitted since they closely parallel those of Section \ref{sec: inflow}. 
Finally, Appendix \ref{app: REG INIT} contains the construction of the regularized initial and boundary data used throughout the paper.

%
%
%
%
%
%
\section{Compactness from the weighted Fisher information} \label{sec: comp}

The compactness theory developed in this section provides the essential input for
our analysis of the initial--boundary value problems in Sections \ref{sec: inflow}
and \ref{sec: reflect}. The key observation is that dissipation in the velocity variable, even when modulated by a possibly degenerate coefficient, can be combined with velocity averaging in the space-time variables to produce strong compactness in all variables. More precisely, provided that we have control of a weighted Fisher information, convergence of velocity averages, and compatibility of the weights with the non-vacuum region of the limiting density, we can guarantee strong convergence of the sequence in $L^1_{t,x,v}$.

This principle is made precise in the following proposition.

\begin{proposition} \label{prop: weighted fisher}
    Let $\Omega\subset\R^d$ be an open subset (possibly unbounded). Assume that $\{f_m(t,x,v)\}$ is a nonnegative sequence such that
    \begin{enumerate}[(P1)]
        \item $f_m \weakto f$ weakly in $L^1((0,T)\times\Omega\times\R^d)$,
        \item For each $\psi \in C_c(\R^d)$, the sequence of averages $\left\{ \int_{\R^d} f_m \psi(v)\,\dv \right\}$ converges in $L^1((0,T)\times\Omega)$.
        \item The following weighted Fisher information is bounded:
        \begin{align*}
            \sup_m \int_{(0,T)\times\Omega\times\R^d} a_m(t,x) \lt|\nabla_v \sqrt{f_m} \rt|^2 \,\dt\dx\dv \le C.
        \end{align*}
        Here we require that each $a_m:(0,T)\times\Omega\to [0,\infty]$ is Lebesgue measurable.
        \item $a_m\to a$ a.e. in $D_+ := \{(t,x)\in (0,T)\times\Omega :a(t,x)>0\}$, and $a$ has the property that $D_+ \supset \{(t,x): \rho_f(t,x) > 0\}$ up to a Lebesgue-negligible set.
    \end{enumerate}
    Then modulo a subsequence it holds that $f_m \to f$ a.e. and strongly in $L^1((0,T)\times\Omega\times\R^d)$.
\end{proposition}

The proof of Proposition \ref{prop: weighted fisher} proceeds in two main steps.
We first establish the constant-coefficient case $a_m\equiv 1$ (Lemma \ref{lem: key compactness}),
which is already nontrivial yet closely related to the recent work \cite{sampaio2024}.
Our argument relaxes several of the structural assumptions used there, showing
that they can be replaced by the more natural entropy and moment bounds available
for kinetic Fokker--Planck models.
We then treat general measurable coefficients $a_m$, which may vanish on subsets of
$(0,T)\times\Omega$. 
The main difficulty is that pointwise convergence $a_m\to a$ does not imply convergence of
level sets $\{a_m \ge c\}$, and naive truncation arguments may fail. 
To overcome this, we implement a careful level-set decomposition which localizes
the Fisher information and reduces the problem to the constant-coefficient setting on each level. This strategy allows us to avoid any uniform ellipticity assumptions, while still
retaining the strong compactness conclusion.

\begin{remark}
The assumptions in Proposition \ref{prop: weighted fisher} are natural and arise in diverse applications.
\begin{enumerate}
    \item (P1) is expected (along some subsequence) once one can verify that the physically relevant bounds in  \eqref{init phys} propagate with time. If $\Omega$ is unbounded, additional spatial moment conditions may be imposed on the initial data.
      \item (P2), the convergence of $\left\{\intr f_m \psi(v)\,\dv\right\}$ for each $\psi\in C_c(\R^d)$, is typically a consequence of velocity averaging lemmas \cite{GLPS88, GPS85, Glassey96}. While these lemmas usually yield convergence only along subsequences, assumption (P1) guarantees uniqueness of the limit and thus removes the need for subsequence extraction.
    \item (P3) naturally arises from entropy dissipation estimates for kinetic Fokker--Planck equations such as \eqref{eq: main}, \eqref{eq: LINFP}, and \eqref{eq: FPT}, which provide precisely the required Fisher-information bounds.
    \item (P4) is satisfied, for instance, when $a_m = \calT_{f_m}$ and $a = \calT_f$, provided that velocity averaging lemmas ensure convergence of the second moments of $f_m$ in $L^1((0,T)\times \Omega)$.
\end{enumerate}
\end{remark}

%
%
%
%
%
%
\subsection{Constant diffusion coefficient case}
 
We first establish the compactness in the baseline case $a_m\equiv1$, which is the
key building block for the general measurable weights.

\begin{lemma}\label{lem: key compactness}
    Let $\Omega$ be an open subset of $\R^d$ (not necessarily bounded), and $\{f_m(t,x,v)\}$ a sequence of nonnegative functions such that
    \begin{enumerate}[(L1)]
        \item The sequence $\{f_m\}$ is weakly convergent in $L^1((0,T)\times\Omega \times\R^d)$ to $f$.
        \item For each $\psi\in C_c(\R^d)$, the sequence of averages $\left\{\int_{\R^d} f_m \psi(v) \,\dv \right\}_m$ is compact in $L^1((0,T)\times \Omega)$.
        \item For each $\alpha,\ve>0$, there exists a measurable set $E\subset (0,T)\times\Omega$ with $\mathscr{L}^{d+1}(E) \le \ve$, such that the Fisher information enjoys the following bound
        \begin{align*}
            \int_{(K_\alpha \cap E^c) \times \R^d} \left|\nabla_v \sqrt{f_m} \right|^2 \dt\dx\dv \le C(\ve,\alpha),
        \end{align*}
        where $\scrL^{d+1}$ is the $(d+1)$-dimensional Lebesgue measure and
        \begin{align*}
            K_\alpha := \{(t,x)\in (0,T)\times\Omega: \rho_f(t,x) > \alpha \}.
        \end{align*}
    \end{enumerate}
    Then modulo a subsequence $f_m$ is convergent a.e. and strongly in $L^1((0,T)\times\Omega\times\R^d)$.
\end{lemma}

\begin{remark} \label{rem: samp}
    The proof of Lemma \ref{lem: key compactness} is an adaptation and generalization of that of \cite[Proposition 2]{sampaio2024}; for completeness, we provide all details. Indeed, in \cite{sampaio2024}, the sequence $f_m$ is assumed to converge weakly$^*$ in $L^\infty$ (due to the quantum structure of the Landau--Fermi--Dirac equation), and assumption (L3) is replaced by an $L^2$-bound on the velocity gradients,
    \begin{align*}
        \int_{(K_\alpha\cap E^c)\times\R^d} |\nabla_v f_m|^2\,\dt\dx\dv \le C(\ve,\alpha).
    \end{align*}
Our contribution is to relax both requirements: 
weak$^*$ convergence in $L^\infty$ is replaced by equiintegrability and tightness in $L^1$, 
and the $L^2$-gradient bound is replaced by boundedness of the Fisher information.
\end{remark}

\begin{proof}[Proof of Lemma \ref{lem: key compactness}]
We first treat the case of bounded $\Omega$. Indeed, in the general case of unbounded $\Omega$, one can set $B_R$ as the standard open ball in $\R^d$ of radius $R>0$, apply the below proof to $\Omega\cap B_R$, and then extract diagonally by taking $R\uparrow \infty$.
    
    First we remark that the weak convergence assumption (L1) and the Dunford--Pettis theorem imply that
    \begin{align*}
        \lim_{R\to\infty} \sup_m \int_{(0,T)\times\Omega\times \{|v|>R\}} f_m(t,x,v)\,\dt\dx\dv = 0 .
    \end{align*}
    In particular, for each $\psi(v)\in C^0\cap L^\infty(\R^d)$, it follows that
    \begin{align} \label{eq: fish tight}
        \lim_{R\to\infty} \sup_m \int_{(0,T)\times\Omega\times \{|v|>R\}} f_m(t,x,v) \psi(v)\,\dt\dx\dv = 0 .
    \end{align}
    By utilizing the above tightness, and standard density arguments, we may therefore assume that assumption (L2) holds for each $\psi \in C^0\cap L^\infty (\R^d)$.
    
    Then, we observe that we can extract a subsequence such that (L2) holds for every $\psi\in C^0\cap L^\infty(\R^d)$ along this subsequence. Indeed, let us fix any $R>0$ and consider those $\psi\in C_c(B_R)$. Now $C_c(B_R) \subset L^1(B_R)$, and by standard approximation theory there is a Schauder basis of $L^1(B_R)$ which consists of smooth, compactly supported functions: thus, by a diagonal extraction we obtain a ($R$-dependent) subsequence for which $\left\{\intr f_m(t,x,v) \psi(v)\,\dv\right\}_m$ is convergent in $L^1((0,T)\times\Omega)$ for every $\psi \in C_c(B_R)$. Then taking $R\uparrow \infty$, we can utilize \eqref{eq: fish tight} to extract another diagonal subsequence, for which
    \begin{align} \label{eq: fish velavg}
        \int_{\R^d} f_m \psi(v) \,\dv \to \int_{\R^d} f \psi(v) \,\dv \text{ a.e. and in $L^1((0,T)\times\Omega)$} \quad \forall \psi\in C^0 \cap L^\infty(\R^d).
    \end{align}
    We remark that in particular, along this subsequence:
    \begin{align}\label{eq: fish: rho}
        \rho_{f_m} \to \rho_f \quad \text{a.e. and in }L^1((0,T)\times\Omega).
    \end{align}
    
    Now fix a mollifier $\eta\in C_c^\infty(\R^d)$ such that $\intr \eta(v)\,\dv = 1$ and $\textnormal{supp}(\eta)\subset B_1$. Denote
    \begin{align*}
        \eta_\delta(v) = \frac{1}{\delta^d} \eta\left(\frac{v}{\delta}\right), \quad f_m^\delta:= f_m * \eta_\delta, \quad f^\delta := f * \eta_\delta. 
    \end{align*}
    Then, assumption (L1) shows that for any $\varphi \in L^\infty((0,T)\times\Omega\times\R^d)$
    \begin{align*}
        \int_0^T \int_{\Omega\times\R^d} f_m^\delta(t,x,v) \varphi(t,x,v) \,\dx\dv\dt &= \int_0^T \int_{\R^d \times \Omega\times\R^d  } f_m(t,x,w) \eta_\delta(v-w) \varphi(t,x,v) \,\dw\dx\dv\dt \\
        &\xrightarrow[m\to\infty]{} \int_0^T \int_{\R^d \times \Omega\times\R^d } f(t,x,w) \eta_\delta(v-w) \varphi(t,x,v) \,\dw\dx\dv\dt \\
        &= \int_0^T \int_{\Omega\times\R^d} f^\delta(t,x,v) \varphi(t,x,v) \,\dx\dv\dt ,
    \end{align*}
    in other words
    \begin{align}\label{eq: fmdelta weak}
        f_m^\delta \xrightarrow[m\to\infty]{} f^\delta \quad \text{weakly in}\quad L^1((0,T)\times\Omega\times\R^d).
    \end{align}
    On the other hand, the velocity averaging \eqref{eq: fish velavg} implies that
    \begin{align*}
        f_m^\delta \xrightarrow[m\to\infty]{} f^\delta \quad \text{a.e. }(t,x,v)\in (0,T)\times\Omega\times\R^d.
    \end{align*}
    Combining this with \eqref{eq: fmdelta weak}, we deduce via the Vitali  convergence theorem (see for instance \cite[IV.8]{DunfordSchwartz}) that
    \begin{align}\label{eq: fmdelta L1}
        f_m^\delta \xrightarrow[m\to\infty]{} f^\delta \quad \text{strongly in} \quad L^1((0,T)\times\Omega\times\R^d).
    \end{align}
    
    Next, we write out using the definition of $f_m^\delta$:
    \begin{align*}
        \int_{\R^d} |f_m - f_m^\delta|\,\dv &= \intr  \left|f_m(v) - \int_{B_\delta} \eta_\delta(v_*) f_m(v-v_*) \,\dv_* \right|\,\dv \\
        &\le \int_{\R^d} \int_{B_\delta} \eta_\delta(v_*) |f_m(v-v_*) - f_m(v)| \,\dv_* \dv \\
        &\le \int_{\R^d} \int_{B_\delta} \int_0^1 \eta_\delta(v_*)  |v_*| |\nabla_v f_m(v- hv_*)| \,\tnd h \dv_* \dv,
    \end{align*}
    the last line following by the mean-value theorem. Utilizing $|v_*|\le \delta$ and the Fubini--Tonelli theorem, we can then estimate further as
    \begin{align*}
        \int_{\R^d} |f_m - f_m^\delta|\,\dv &\le \delta \int_{\R^d} \int_{B_\delta} \int_0^1 \eta_\delta(v_*) |\nabla_v f_m(v-hv_*)| \,\tnd h \dv_* \dv \\
        &\le \delta \int_0^1 \int_{B_\delta} \left(\int_{\R^d} |\nabla_v f_m(v)| \,\dv \right) \eta_\delta(v_*)\,\dv_* \tnd h \\
        &\le \delta \int_{\R^d} |\nabla_v f_m(v)|\,\dv.
    \end{align*}
    Integrating this over $(t,x)\in K_\alpha \cap E^c$ we therefore obtain a Poincar\`e--Wirtinger type inequality
    \begin{align*}
        \int_{K_\alpha \cap E^c} \int_{\R^d} |f_m - f_m^\delta|\,\dt\dx\dv &\le \delta \int_{(K_\alpha\cap E^c)\times \R^d} |\nabla_v f_m(v)|\,\dt\dx\dv. 
    \end{align*}
    To control this with the Fisher information, we observe first that since $\{f_m\}$ is weakly convergent, $\{f_m\}$ is a bounded subset of $L^1((0,T)\times\Omega\times\R^d)$. Thus we can compute
    \begin{equation}\label{eq/ CS}
    \begin{split}
        &\int_{(K_\alpha \cap E^c)\times\R^d} |\nabla_v f_m(v)| \,\dt\dx\dv \\
        &\quad = \int_{(K_\alpha \cap E^c)\times\R^d} 2\sqrt{f_m} \frac{|\nabla_v f_m|}{2\sqrt{f_m}} \mathbf{1}_{f_m>0} \,\dt\dx\dv \\
        &\quad \le 2\left(\int_{(K_\alpha \cap E^c)\times\R^d} f_m\,\dt\dx\dv\right)^{\frac{1}{2}} \left(\int_{(K_\alpha \cap E^c)\times\R^d} \frac{|\nabla_v f_m|^2}{4f_m} \mathbf{1}_{f_m>0} \,\dt\dx\dv\right)^{\frac{1}{2}} \\
        &\quad \le C \left(\int_{(K_\alpha \cap E^c)\times\R^d} \left|\nabla_v \sqrt{f_m}\right|^2 \,\dt\dx\dv\right)^{\frac{1}{2}} \\
        &\quad = C\sqrt{C(\ve,\alpha)}.
    \end{split}
    \end{equation}
    Hence, we have
    \begin{equation} \label{eq: error}
        \int_{(K_\alpha \cap E^c)\times\R^d} |f_m - f_m^\delta|\,\dt\dx\dv \le C\delta \sqrt{C(\ve,\alpha)}.
    \end{equation}

    Now, we can verify that $f_m$ is a Cauchy sequence in $L^1((0,T)\times\Omega\times\R^d)$. Indeed, let us write out
    \begin{align*}
 \int_{(0,T)\times\Omega\times\R^d} |f_m - f_n| \,\dt\dx\dv  
        &= \int_{K_\alpha^c \times \R^d} |f_m - f_n|\,\dt\dx\dv + \int_{E\times \R^d} |f_m - f_n|\,\dt\dx\dv \\
        &\quad + \int_{(K_\alpha \cap E^c)\times \R^d} |f_m - f_n|\,\dt\dx\dv \\
        &=: I_1 + I_2 + I_3.
    \end{align*}
    
    For $I_1$, the triangle inequality and \eqref{eq: fish: rho} give
    \begin{align*}
        \limsup_{m,n\to\infty} I_1 &\le \limsup_{m,n\to\infty} \int_{K_\alpha^c} (\rho_{f_m} + \rho_{f_n}) \,\dt\dx  \xrightarrow[m,n\to\infty]{} \int_{K_\alpha^c} 2\rho_f \,\dt\dx  \le 2\alpha \mathscr{L}^{d+1}((0,T)\times\Omega),
    \end{align*}
    the last line following by definition of $K_\alpha$.

    For $I_2$, we find 
    \begin{align*}
        I_2 \le \int_E (\rho_{f_m} + \rho_{f_n}) \,\dt\dx.
    \end{align*}
    Since $\{\rho_{f_m}\}$ is strongly convergent in $L^1((0,T)\times\Omega)$ (see \eqref{eq: fish: rho}), it is also weakly convergent; in particular $\{\rho_{f_m}\}$ is equiintegrable in $(0,T)\times\Omega$. This implies the existence of a modulus of continuity $\omega:[0,\infty]\to [0,\infty]$ (independent of $m$) such that
    \begin{align*}
        \sup_m \int_E \rho_{f_m} \,\dt\dx \le \omega(\ve) , \quad \text{where} \quad  \omega(\ve)\xrightarrow[\ve\downarrow 0]{} 0.
    \end{align*}
    We deduce that
    \begin{align*}
     I_2 \le 2 \omega(\ve).   
    \end{align*}

    For $I_3$, we split it as
    \begin{align*}
        \limsup_{m,n\to\infty} I_3 &\le \limsup_{m,n\to\infty}\Bigg( \int_{(K_\alpha \cap E^c)\times \R^d} |f_m - f_m^\delta| \,\dt\dx\dv + \int_{(K_\alpha \cap E^c)\times\R^d} |f_n - f_n^\delta| \,\dt\dx\dv \\
        &\qquad \qquad \qquad + \int_{(K_\alpha \cap E^c)\times\R^d} |f_m^\delta - f_n^\delta| \,\dt\dx\dv \Bigg)\\
        &\le C\delta \sqrt{C(\ve,\alpha)},
    \end{align*}
    the last line following by \eqref{eq: error} and the fact that $\{f_m^\delta\}_m$ is Cauchy in $L^1((0,T)\times\Omega\times\R^d)$ (see \eqref{eq: fmdelta L1}). Collecting the estimates, we conclude that for any $\alpha,\delta,\ve>0$
    \begin{align*}
        \int_{(0,T)\times\Omega\times\R^d} |f_m - f_n|\,\dt\dx\dv \le 2\alpha \mathscr{L}^{d+1}((0,T)\times\Omega) + 2\omega(\ve) + C\delta \sqrt{C(\ve,\alpha)},
    \end{align*}
    where $C>0$ is independent of $\ve,\delta,\alpha$. We fix $\ve,\alpha>0$ and take $\delta\downarrow 0$ first to deduce that the left-hand side is bounded by $2\alpha \mathscr{L}^{d+1}((0,T)\times\Omega) + 2\omega(\ve)$. Then we take $\alpha,\ve\downarrow 0$ to complete the proof.
\end{proof}
%
%
%
%
%
%
\subsection{General diffusion coefficient case} 
We now prove Proposition \ref{prop: weighted fisher} for general measurable weights,
following the level-set localization strategy outlined above.

\begin{proof}[Proof of Proposition \ref{prop: weighted fisher}]
    We assume that $\Omega$ is bounded, since the case of unbounded domains follows
    by applying the argument to $\Omega\cap B_R$ and then resorting to a diagonal process.

    Our starting observation is that we can extract a strictly decreasing sequence $c_k\downarrow 0$
    with
    \begin{align*}
        \mathscr{L}^{d+1}(\{(t,x):a(t,x) = c_k\}) = 0 \quad \forall k\in \bbN .
    \end{align*}
    This is based upon the simple fact that a measurable function on a bounded domain cannot have ``too many level sets.'' Indeed, since
    \begin{align*}
        (0,T)\times\Omega \supset \bigcup_{c\in (0,1]} \{(t,x): a(t,x) = c\},
    \end{align*}
    we deduce that only at most countably many of the level sets can have non-zero Lebesgue measure. Without loss of generality, we set
    \begin{align*}
        c_k = \frac{1}{k},\quad k\in \bbN.
    \end{align*}

    The advantage of distinguishing level sets in this way is that, now, the pointwise convergence assumption (P4) implies convergence of the following indicator functions:
    \begin{align}\label{eq: ind ptw}
        \mathbf{1}_{\{a_m \ge 1/k\}}(t,x) \xrightarrow[m\to\infty]{} \mathbf{1}_{\{a \ge 1/k\}}(t,x) \quad \text{for almost every }(t,x)\in D_+.
    \end{align}
    Let us view $\mathbf{1}_{\{a_m\ge 1/k\}}$ as a sequence bounded in $L^\infty((0,T)\times\Omega \times \R^d)$. We will show that the sequence $\{\mathbf{1}_{D_+} \mathbf{1}_{\{a_m \ge 1/k\}} f_m\}$ satisfies the assumptions of Lemma \ref{lem: key compactness}, for each fixed $k$. 

    Toward this end, we first verify that
    \begin{align}\label{eq: prod weak}
        \mathbf{1}_{D_+} \mathbf{1}_{\{a_m\ge 1/k\}} f_m \xrightarrow[m\to\infty]{} \mathbf{1}_{D_+} \mathbf{1}_{\{a\ge 1/k\}} f \quad \text{weakly in }L^1((0,T)\times\Omega\times\R^d).
    \end{align}
    The starting step here is localization in the velocity variable. Owing to assumption (P1) and the Dunford--Pettis criterion, given $\ve>0$, we fix $R_\ve>0$ for which
    \begin{align*}
        \sup_m \int_{(0,T)\times\Omega\times (\R^d\setminus B_{R_\ve})} f_m \,\dt\dx\dv \le \ve.
    \end{align*}
    Similarly, since $f\in L^1((0,T)\times\Omega\times\R^d)$, we may assume without loss of generality that the same $R_\ve>0$ satisfies
    \begin{align*}
        \int_{(0,T)\times\Omega\times (\R^d\setminus B_{R_\ve})} f \,\dt\dx\dv \le \ve.
    \end{align*}
    Now the pointwise convergence in \eqref{eq: ind ptw} and Egorov's theorem show we can find a measurable subset $E_\ve \subset (0,T)\times\Omega\times B_{R_\ve}$ with $\mathscr{L}^{2d+1}(E_\ve)\le \ve$ such that
    \begin{align}\label{eq: ind unif}
        \mathbf{1}_{D_+} \mathbf{1}_{\{a_m\ge 1/k\}} \xrightarrow[m\to\infty]{} \mathbf{1}_{D_+} \mathbf{1}_{\{a \ge 1/k\}} \quad \text{uniformly in }[(0,T)\times\Omega\times B_{R_\ve}]\setminus E_\ve.
    \end{align}
    Since the family $\{f_m\}$ is equiintegrable and $f\in L^1((0,T)\times\Omega\times\R^d)$, we remark that there is a modulus of continuity $\omega:[0,\infty]\to [0,\infty]$, independent of $m$, for which
    \begin{align*}
        \sup_m \int_{E_\ve} f_m \,\dt\dx\dv , \; \int_{E_\ve} f\,\dt\dx\dv \le \omega(\ve) \xrightarrow[\ve\downarrow 0]{} 0 .
    \end{align*}
    Hence, for an arbitrary function $\varphi\in L^\infty((0,T)\times\Omega\times\R^d)$, we obtain that
    \begin{align*}
        &\int_{(0,T)\times\Omega\times\R^d} \mathbf{1}_{D_+} (\mathbf{1}_{\{a_m\ge 1/k\}} f_m - \mathbf{1}_{\{a \ge 1/k \} } f ) \varphi(t,x,v) \,\dt\dx\dv\\
        &\quad \le \|\varphi\|_{L^\infty} \int_{E_\ve} (f_m + f) \,\dt\dx\dv   + \int_{[(0,T)\times\Omega\times B_{R_\ve}]\setminus E_\ve} \mathbf{1}_{D_+} \, (\mathbf{1}_{\{a_m\ge 1/k\}} - \mathbf{1}_{\{a\ge 1/k\}} ) f \varphi \,\dt\dx\dv \\
        &\qquad + \int_{[(0,T)\times\Omega\times B_{R_\ve}]\setminus E_\ve} \mathbf{1}_{D_+} \mathbf{1}_{\{a_m\ge 1/k \} } (f_m - f) \varphi \,\dt\dx\dv  + \int_{(0,T)\times\Omega\times (\R^d\setminus B_{R_\ve}) } \mathbf{1}_{D_+} (f_m + f) \varphi \,\dt\dx\dv \\
        &\quad\le 2\|\varphi\|_{L^\infty} \omega(\ve) + \|\mathbf{1}_{D_+}\mathbf{1}_{\{a_m\ge 1/k\}} - \mathbf{1}_{D_+} \mathbf{1}_{\{a\ge 1/k\}}\|_{L^\infty([(0,T)\times\Omega\times B_{R_\ve}]\setminus E_\ve )} \|\varphi\|_{L^\infty} \|f_m\|_{L^1}  \\
        &\qquad + \int_{[(0,T)\times\Omega\times B_{R_\ve}]\setminus E_\ve} \mathbf{1}_{D_+} \mathbf{1}_{\{a\ge 1/k \} } (f_m - f) \varphi \,\dt\dx\dv + 2\|\varphi\|_{L^\infty} \ve.
    \end{align*}
    We now take $m\to\infty$. The uniform convergence \eqref{eq: ind unif} and the weak convergence $f_m\weakto f$ (which implies $\|f_m\|_{L^1} \le C$) show
    \begin{align*}
        \limsup_{m\to\infty} \int_0^T \inttr (\mathbf{1}_{\{a_m\ge 1/k\}} f_m - \mathbf{1}_{\{a \ge 1/k \} } f ) \varphi(t,x,v) \,\dt\dx\dv \le 2\|\varphi\|_{L^\infty} (\omega(\ve) + \ve).
    \end{align*}
    As $\ve>0$ is arbitrary we then take it to zero in order to conclude that \eqref{eq: prod weak} holds true.

    Next, we find that (L2) of Lemma \ref{lem: key compactness} holds for $\{\mathbf{1}_{D_+} \mathbf{1}_{\{a_m\ge 1/k\}} f_m\}$. Indeed, for each $\psi \in C_c(\R^d)$, writing $\psi$ into its positive and negative parts $\psi = \psi^+ - \psi^-$, we see that the following dominations hold:
    \begin{align*}
        \left|\intr \mathbf{1}_{D_+} \mathbf{1}_{\{a_m\ge 1/k\}}  f_m \psi^+(v)\,\dv \right| &= \mathbf{1}_{D_+} \mathbf{1}_{\{a_m\ge 1/k\}} \left|\intr f_m \psi^+(v)\,\dv\right| \le \intr f_m \psi^+(v) \,\dv ,\\
        \left|\intr \mathbf{1}_{D_+} \mathbf{1}_{\{a_m\ge 1/k\}} f_m \, \psi^-(v)\,\dv \right| &\le \intr f_m \psi^-(v)\,\dv.
    \end{align*}
    Since (P2) gives that $\intr f_m\psi^{\pm}(v)\,\dv$ is strongly convergent in $L^1((0,T)\times\Omega)$, it follows by the generalized Lebesgue dominated convergence theorem (or one can also use the Vitali convergence theorem) to deduce that
    \begin{align*}
        \mathbf{1}_{D_+} \mathbf{1}_{\{a_m\ge 1/k\}} \intr f_m \psi(v)\,\dv \xrightarrow[m\to\infty]{} \mathbf{1}_{D_+} \mathbf{1}_{\{a\ge 1/k\}} \intr f \psi(v) \,\dv \quad \text{in} \quad L^1((0,T)\times\Omega).
    \end{align*}

    Finally, (L3) holds for $\{\mathbf{1}_{D_+}\mathbf{1}_{\{a_m\ge 1/k\}} f_m\}_m$, since only $f_m$ is $v$-dependent:
    \begin{align*}
        \int_{(0,T)\times\Omega\times\R^d} \left|\nabla_v \sqrt{\mathbf{1}_{D_+} \mathbf{1}_{\{a_m\ge 1/k\}} f_m}\right|^2 \,\dt\dx\dv &\le \int_{\{a_m \ge 1/k\}\times\R^d} \left|\nabla_v \sqrt{f_m}\right|^2 \,\dt\dx\dv \\
        &\le k \int_{(0,T)\times\Omega\times\R^d} a_m(t,x) \left|\nabla_v \sqrt{f_m}\right|^2 \,\dt\dx\dv \\
        &\le Ck \quad \forall k\in \bbN.
    \end{align*}

    From Lemma \ref{lem: key compactness} we deduce that
    \begin{align} \label{eq: prod ptw}
        f_m \mathbf{1}_{\{a_m\ge 1/k\}}  \xrightarrow[m\to\infty]{}  f \mathbf{1}_{\{a\ge 1/k\}} \quad \text{a.e. in }D_+\times\R^d, \quad \forall k\in \bbN.
    \end{align}
    We can now assume, after a diagonal extraction, that we are along a subsequence for which \eqref{eq: prod ptw} holds for every $k$. It is then straightforward to check that $f_m \to f$ almost everywhere in $D_+\times\R^d$. Indeed, fix $(t,x,v) \in D_+\times \R^d$ for which \eqref{eq: prod ptw} holds. Then by definition, $a(t,x)>0$. Due to the Archimedean property, we can locate $k_* = k_*(t,x)\in \bbN$ for which
    \begin{align*}
        a(t,x) > 1/k_*,
    \end{align*}
    and by assumption (P4) we can suppose $a_m(t,x) > 1/k_*$ for all $m$ large enough. Since this implies $\mathbf{1}_{\{a_m \ge 1/k_*\}}(t,x) \equiv 1$ for all large enough $m$, we use \eqref{eq: prod ptw} to deduce
    \begin{align*}
        f_m(t,x,v) = f_m(t,x,v) \mathbf{1}_{\{a_m\ge 1/k_*\}}(t,x) \xrightarrow[m\to\infty]{} f(t,x,v) \mathbf{1}_{\{a\ge 1/k_*\}}(t,x) = f(t,x,v).
    \end{align*}
    Hence $f_m\to f$ a.e. in $D_+\times \R^d$.

    Finally, we need to discuss the convergence $f_m\to f$ on $\{a = 0\}$. Notice that thanks to the weak compactness of $\{f_m\}$, the arguments preceding \eqref{eq: fish velavg}--\eqref{eq: fish: rho} show $\varphi\equiv 1$ is admissible in assumption (P2). Hence
    \begin{align*}
        \rho_{f_m}\to \rho_f \quad \text{a.e. and in }L^1((0,T)\times\Omega).
    \end{align*}
    We remark $\rho_f(t,x) = 0$ is equivalent to $f(t,x,v) = 0$ a.e. $v\in \R^d$, and from (P4) we know $D_0 := \{(t,x):a(t,x)=0 \} \subset \{(t,x): \rho_f(t,x) = 0\}$. In this way,
    \begin{align*}
        \int_{D_0} \intr |f_m - f| \,\dv\dt\dx &= \int_{D_0} \intr f_m \,\dv\dt\dx  = \int_{D_0} \rho_{f_m} \,\dt\dx  \xrightarrow[m\to\infty]{} \int_{D_0} \rho_f \,\dt\dx  = 0,
    \end{align*}
    and therefore $f_m\to f$ in $L^1(D_0\times \R^d)$.     Altogether, modulo a subsequence, $f_m\to f$ a.e. on $(0,T)\times\Omega\times\R^d$,  and the strong convergence in $L^1((0,T)\times\Omega\times\R^d)$ follows from
    the Vitali convergence theorem.
    \end{proof}

The following corollary is a direct consequence, and discusses the case where $a(t,x)$ does not vanish anywhere.

\begin{corollary} 
    Let $\Omega\subset\R^d$ be an open subset. Assume that $\{f_m(t,x,v)\}$ satisfies
    \begin{enumerate}
        \item[(P1)] $f_m \weakto f$ weakly in $L^1((0,T)\times\Omega\times\R^d)$,
        \item[(P2)] For each $\psi \in C_c(\R^d)$, the sequence of averages $\left\{ \int_{\R^d} f_m \psi(v)\,\dv \right\}$ is compact in $L^1((0,T)\times\Omega)$.
        \item[(P3)] The following weighted Fisher information is bounded:
        \begin{align*}
            \int_{(0,T)\times\Omega\times\R^d} a_m(t,x) |\nabla_v \sqrt{f_m}|^2 \,\dt\dx\dv \le C,
        \end{align*}
        where $a_m:(0,T)\times\Omega\to [0,\infty]$ is Lebesgue measurable.
        \item[(P4')] $a_m(t,x) \to a(t,x) > 0$ almost every $(t,x)\in (0,T)\times \Omega$.
    \end{enumerate}
    Then modulo a subsequence it holds that $f_k \to f$ a.e. and strongly in $L^1((0,T)\times\Omega\times\R^d)$.
\end{corollary}

\begin{proof}
    Following the proof of Proposition \ref{prop: weighted fisher}, we have that $f_m\to f$ a.e. in $D_+\times\R^d$, where $D_+ := \{(t,x):a(t,x)>0\}$. Since (P4') implies $D_+ = (0,T)\times\Omega$ (up to a zero measure set), the result follows.
\end{proof}

%
%
%
%
%
%

\section{The inflow problem} \label{sec: inflow}

In this section we construct solutions to \eqref{eq: main} regarding the inflow problem, and namely prove Theorem \ref{thm: inflow}. Recall that $f^0$ is assumed only to satisfy \eqref{init phys}, whereas the inflow data $g$ satisfies
\begin{align}\label{ASSUMP g}
    g\ge 0, \qquad \int_{\Sigma_-^T} (1+|v|^2+|\log g|) g \,|n(x)\cdot v|\,\dt\tnd\sigma\dv < +\infty.
\end{align}

\noindent
\textbf{Setup of the problem.} We let $\Omega$ denote a bounded, open subset of $\R^d$ with $C^1$-boundary. Specifically, let us assume there exists $\mathfrak{b}\in C^1\cap W^{1,\infty}(\R^d;\R)$ with nonvanishing gradient such that
\begin{align*}
    \Omega = \{x\in \R^d : \mathfrak{b}(x) > 0 \}, \qquad 
    n(x) = - \frac{\nabla \mathfrak{b}(x)}{|\nabla \mathfrak{b}(x)|}.
\end{align*}
The effective boundary portion, \textit{i.e.} the set on which we impose the initial and boundary conditions, is
\begin{align*}
    \bbGamma := \lt( \{0\}\times \Omega\times \R^d \rt)  \cup \lt( (0,T] \times \Gamma_- \rt),
\end{align*}
where
\begin{align*}
    \Gamma_- := \{(x,v)\in \p\Omega\times\R^d : n(x)\cdot v < 0\}.
\end{align*}
Naturally, we also denote
\begin{align*}
    &\Gamma_0 := \{(x,v)\in \p\Omega\times\R^d : n(x)\cdot v = 0\},\\
    &\Gamma_+ := \{(x,v)\in \p\Omega\times\R^d : n(x)\cdot v > 0\}, \\
    &\Gamma = \Gamma_- \cup \Gamma_0 \cup \Gamma_+,
\end{align*}
and also, for $t\in [0,T]$, recall the notations
\begin{align*}
    \Sigma_{\pm}^t := [0,t]\times \Gamma_{\pm}.
\end{align*}
Regarding the phase domain, we use the notation
\begin{align*}
    &\calO := \Omega\times\R^d.
\end{align*}

\noindent
\textbf{A priori estimates.}  Recall that the nonlinear operator can be rewritten as
\begin{align*}
    \nu_f \nabla_v\cdot \lt(\calT_f \nabla_v f + (v-u_f)f \rt) = \nu_f \nabla_v\cdot \lt( f \calT_f \nabla_v \log \frac{f}{M[f]} \rt),
\end{align*}
where $M[f]$ is the local Maxwellian
\begin{align*}
M[f] = \frac{\rho_f}{(2\pi \calT_f)^{d/2}} e^{-\frac{|v-u_f|^2}{2\calT_f}}.
\end{align*}
Taking $1,v_i, |v|^2$ as test functions, we find the formal conservation of mass, momentum, and energy:
\begin{align*}
    \ddt \int_{\calO} (1,v, |v|^2) f \,\dx\dv + \int_{\Gamma} (1,v,|v|^2)\gamma f \,v\cdot n(x) \,\dx\dv  = 0.
\end{align*}

Next, regarding the $\calH$-functional, we take $(1+\log f)$ as a test function in \eqref{eq: main} to find that
\begin{align*}
 \ddt \int_{\calO} f\log f \,\dx\dv + \int_{\Gamma} (\gamma f \log \gamma f) (v\cdot n(x)) \,\tnd\sigma\dv + \int_{\calO} 4 \nu_f \calT_f |\nabla_v \sqrt{f}|^2 \,\dx\dv  = d\int_{\calO} \nu_f f\,\dx\dv.
\end{align*}
Clearly, the mass-conserving nature of $f$ shows that we can then obtain natural bounds for the entropy and also of the weighted Fisher information. Formal manipulations also show that
\begin{align*}
    \ddt \int_{\calO} f\log f \,\dx\dv+ \int_{\Gamma }(\gamma f \log \gamma f) (v\cdot n(x)) \,\tnd\sigma\dv + \int_{\calO} \frac{1}{f\calT_f}|\calT_f \nabla_v f + (v-u_f)f|^2 \,\dx\dv = 0,
\end{align*}
and thus the (negative of the) entropy is decreasing provided there is no net flux through the boundary.

%
%
%
%
%
%

\subsection{The regularized system}\label{subsec: the regularized system}
Given $r\in \R_+$ and $\ve_1,\ve_2>0$, we introduce the $\ve_1,\ve_2$-renormalization of $r$ as
\begin{align*}
    \renorm{r}{\ve_1}{\ve_2} := \frac{r}{\ve_1 + \ve_2(1 + r)}. 
\end{align*}
Similarly, for a vector $\bfr\in \R^d$, we define
\begin{align*}
    \renorm{\bfr}{\ve_1}{\ve_2} := \frac{\bfr}{\ve_1 + \ve_2(1 + |\bfr|)}.
\end{align*}
Using the above notations, for fixed $\ve>0$, we define the regularized temperature as
\begin{align*}
    \calT^{(\ve)}_f := \renorm{V_f}{\rho_f}{\ve} + \ve = \frac{V_f}{\rho_f + \ve(1 + V_f)} + \ve,
\end{align*}
and the regularized bulk velocity as
\begin{align*}
    u_f^{(\ve)} := \renorm{j_f}{\rho_f}{\ve} = \frac{j_f}{\rho_f + \ve (1 + |j_f|)}, \quad j_f := \rho_f u_f.
\end{align*}
It can be easily seen that in this way
\begin{equation}\label{eq: reg major}
\begin{split}
    &|\renorm{\bfr}{1}{\ve}| \le \min\lt\{\frac{1}{\ve}, |\bfr| \rt\} \quad \forall \bfr\in \R^d,  \quad \calT_{f}^{(\ve)} \le \min\lt\{\calT_f + \ve\; , \; \frac{1}{\ve} + \ve \rt\}, \quad |u_f^{(\ve)}| \le \min\lt\{|u_f|, \frac{1}{\ve}\rt\}.
    \end{split}
\end{equation}
We then consider the following system of equations
\begin{equation}\label{eq: reg}
    \begin{cases}
        \p_t f_\ve + v\cdot \nabla_x f_\ve = \left(\nu_{f_\ve} +  \ve\right) \nabla_v\cdot \left( \calT_{f_\ve}^{(\ve)} \nabla_v f_\ve +  (\renorm{v}{1}{\ve} - u_{f_\ve}^{(\ve)}) f_\ve \right), \\
        f_\ve|_{t=0} = f_\ve^0 \in C_c(\overline{\Omega}\times\R^d;\R_+), \\
        \gamma f_{\ve}|_{\Sigma_-^T} = g_\ve \in C_c([0,T]\times \p\Omega \times \R^d; \R_+),
    \end{cases}
\end{equation}
where the regularized initial and boundary datum are assumed to satisfy the conditions \eqref{init phys} and \eqref{ASSUMP g} uniformly in $\ve$, and also
\begin{align}\label{ASSUMP f0 g}
\begin{cases}
    f_\ve^0 \to f^0 \quad \text{in} \quad L^1(\calO),\\
    \displaystyle \lim_{\ve\downarrow 0}\int_{\calO} |v|^2 f_\ve^0 \,\dx\dv  = \int_{\calO} |v|^2 f^0 \,\dx\dv,\\[6pt]
    \displaystyle \lim_{\ve\downarrow 0}\int_{\calO} f_\ve^0 \log f_\ve^0 \,\dx\dv = \int_{\calO} f^0 \log f^0 \,\dx\dv, \\
    g_\ve \to g \quad \text{in} \quad L^1(\Sigma_-^T, |n(x)\cdot v|\,\dt\tnd\sigma\dv), \\
    \displaystyle \lim_{\ve\downarrow 0} \int_{\Sigma_-^T} |v|^2 g_\ve |n(x)\cdot v| \,\dt\tnd\sigma\dv = \int_{\Sigma_-^T} |v|^2 g |n(x)\cdot v| \,\dt\tnd\sigma\dv ,\\
    \displaystyle\lim_{\ve\downarrow 0} \int_{\Sigma_-^T} g_\ve \log g_\ve \, |n(x)\cdot v| \,\dt\tnd\sigma\dv = \int_{\Sigma_-^T} g \log g \, |n(x)\cdot v| \,\dt\tnd\sigma\dv.
    \end{cases}
\end{align}
Since the construction of such regularizations is rather delicate, we provide a proof in Appendix \ref{app: REG INIT}.

The solution $(f_\ve, \gamma f_\ve)$ to \eqref{eq: reg} will be found by means of a fixed point argument. Toward that end it is necessary for us to also discuss the linear counterpart of \eqref{eq: reg}. Consider the convex space
\begin{align*}
    \mathscr{X} := L_+^1((0,T)\times\Omega) \times [L^1((0,T)\times\Omega)]^d \times L_+^1((0,T)\times\Omega),
\end{align*}
and let $(\varrho, \bfj, V) \in \mathscr{X}$ be given. The linear system corresponding to \eqref{eq: reg} then reads (for fixed $\ve>0$)
\begin{equation}\label{eq: lin}
    \begin{cases}
    \p_t F + v\cdot \nabla_x F = (\nu(\varrho, \bfj, V) + \ve)  \left( \renorm{V}{\varrho}{\ve} + \ve \right) \Delta_v F + (\nu(\varrho, \bfj, V) + \ve) \nabla_v\cdot \Big( ( \renorm{v}{1}{\ve} - \renorm{\bfj}{\varrho}{\ve}) F \Big), \\
    F|_{t=0} = f_\ve^0 , \\
    \gamma F|_{\Sigma_-^T} = g_\ve.
    \end{cases}
\end{equation}
We refer to the following preliminary result for the well-posedness of \eqref{eq: lin}.

\begin{lemma}\label{lem: zhu}
    \cite[Theorem 1.1, Corollary 2.9]{Zhu24} Let the domain $\Omega\in C^{1,1}$ be bounded, and consider the inflow problem
    \begin{equation}\label{eq: zhu}
    \begin{cases}
        \p_t h + v\cdot \nabla_x h = \nabla_v\cdot (A \nabla_v h) + \bfB \cdot \nabla_v h + ch,\\
        \gamma h|_{\Sigma_-^T} = g.
    \end{cases}
    \end{equation}
    Assume that for some $\Lambda>0$ we have
    \begin{equation} \label{eq: Ellipticity and Bounds}
    \begin{split}
        &\textnormal{$A$ is a symmetric $d\times d$ matrix with bounded measurable coefficients,}\\
        &\Lambda^{-1} |\xi|^2 \le A\xi\cdot \xi \le \Lambda |\xi|^2 \quad \forall \xi\in \R^d, \quad |\bfB| + |c| \le \Lambda.
    \end{split}
    \end{equation}
    Then for any $k\ge 0$, there exists $l>0$ dependent only on $d,k$ such that the following conditions are fulfilled:
    \begin{enumerate}[(I)]
    \item Whenever $h^0:=h|_{t=0}$ and $g$ satisfy $\langle v \rangle^l h^0 \in L^\infty(\Omega\times\R^d)$ and $\langle v\rangle^l g \in L^\infty((0,T)\times\Omega\times\R^d)$, then there exists a unique bounded weak solution $(h, \gamma h)$ to \eqref{eq: zhu} such that $h|_{t=0} = h^0$ and $\gamma h|_{\Sigma_-^T} = g$.
    \item In particular, we have that
    \begin{align*}
    &h\in C([0,T];L^2(\calO)) \cap L^2((0,T)\times \Omega; H^1_v(\R^d)),\\
    &\gamma h \in L^2_{\rm loc}(\Sigma^T, |n(x)\cdot v|\,\dt\tnd\sigma\dv),
    \end{align*}
    and the following Green's formula holds for all $\varphi\in C_c^1([0,T]\times\overline{\Omega}\times\R^d)$ and $t\in [0,T]$:
    \begin{equation*}
    \begin{split}
        &\int_{\calO} h(t)\varphi(t) - \int_{\calO} h^0 \varphi(0) + \int_{\Sigma^t} \gamma h\, \varphi\, (n(x)\cdot v) \\
        &\quad = \int_{[0,t]\times\calO} h \Big(\p_s\varphi + v\cdot \nabla_x \varphi \Big) - A\nabla_v f\cdot \nabla_v \varphi + \varphi \bfB \cdot \nabla_v h + c h \varphi.
        \end{split}
    \end{equation*}
    \item There is a constant $C>0$ dependent only on $d,T,\Lambda,k,\Omega$ such that
    \begin{align*}
        &\|\langle v \rangle^k h\|_{L^\infty((0,T)\times\Omega\times\R^d)} \le C\Big(\|\langle v\rangle^l h^0 \|_{L^\infty(\Omega\times\R^d)} + \|\langle v\rangle^l g\|_{L^\infty(\Sigma_-^T)} \Big).
    \end{align*}
    \end{enumerate}
\end{lemma}

To the case of \eqref{eq: lin}, we remark that in the notations of \eqref{eq: zhu}, we have
\begin{equation}\label{eq: abc}
\begin{split}
    &A = (\nu(\varrho,\bfj,V) + \ve) \left(\renorm{V}{\varrho}{\ve} + \ve \right) \bbI_{d\times d}, \\
    &\bfB = (\nu(\varrho,\bfj,V) + \ve) \Big(\renorm{v}{1}{\ve} - \renorm{\bfj}{\varrho}{\ve} \Big),\\
    &c = \lt(\nu(\varrho,\bfj,V) + \ve \rt) \nabla_v\cdot \renorm{v}{1}{\ve} = \lt(\nu(\varrho,\bfj,V) + \ve \rt) \frac{d + \ve d + \ve(d-1)|v|}{(1+\ve(1+|v|))^2} ,
\end{split}
\end{equation}
and that \eqref{eq: Ellipticity and Bounds} are satisfied as
\begin{equation}\label{eq: abc bounds}
\begin{split}
    &\ve^{2} \bbI_{d\times d} \le A \le \|\nu\|_{L^\infty} \left(\frac{1}{\ve} + \ve \right) \bbI_{d\times d},  \quad |\bfB| + |c| \le \|\nu\|_{L^\infty} \lt(d + \frac{10d}{\ve}\rt) .
    \end{split}
\end{equation}
We also remark that the bounds in \eqref{eq: abc bounds} are dependent on $\ve$ but independent of choice of $(\varrho, \bfj, V)\in \mathscr{X}$.

Prior to commencing the fixed point argument, let us recall and investigate some basic properties of the general equation \eqref{eq: zhu}. First of all, it is clear that thanks to Lemma \ref{lem: zhu} (II), \eqref{eq: abc bounds}, and the compact supports of $f^0$, $g$, the solution $h$ to \eqref{eq: lin} has arbitrary algebraic decay. Written explicitly, we have that for all $k\in \bbN$ there is an appropriate constant $C_{T,d,k,\ve}>0$ (again independent of $(\varrho, \bfj, V)$) for which
\begin{equation}\label{eq: lin: decay}
    F(t,x,v) \le \frac{C_{T,d,k,\ve}}{1 + |v|^k}.
\end{equation}
The decay \eqref{eq: lin: decay}, although $\ve$-dependent, helps justify many of the integration by parts that are to follow. Clearly, the above also implies that
\begin{equation} \label{eq: lin: Lp unif in X}
\textnormal{$ F \in L^\infty(0,T; L^1 \cap L^\infty(\Omega\times\R^d))$, \; independently of choice of $(\varrho, \bfj, V) \in \mathscr{X}$}.
\end{equation}

Further, we recall from \cite[Lemma 2.5]{Zhu24} that the solution $(h, \gamma h)$ to \eqref{eq: zhu} is renormalized in the sense of DiPerna--Lions \cite{DipernaLionsODE1989}. For convenience of the reader we recall the general definition here.

\begin{definition}
A pair $(h,\gamma h)$ is said to be a renormalized solution to \eqref{eq: zhu} if, for any pair $(\chi,\varphi)$ such that $\chi \in C^{1,1}(\R)$, $\chi(h) = O(h^2)$, $\chi''(h) = O(1)$ as $|h|\to\infty$, and $\varphi\in C_c^1([0,T]\times\overline{\Omega}\times\R^3)$, it holds for every $t\in (0,T]$ that
    \begin{align*}
        &\int_{\calO} \chi(h(t)) \varphi(t) \,\dx\dv - \int_{\calO} \chi(h^0) \varphi(0) \,\dx\dv + \int_{\Sigma^t} (n(x)\cdot v) \chi(\gamma h) \varphi \,\ds\tnd\sigma\dv\\
        &\quad= \int_{(0,t)\times\calO} \chi(h) (\p_t+v\cdot \nabla_x)\varphi \,\ds\dx\dv \\
        &\qquad + \int_{(0,t)\times\calO} \Big(- A \chi'(h) \nabla_v h \cdot \nabla_v \varphi - \varphi \chi''(h) A\nabla_v h \cdot \nabla_v h + \varphi \bfB \cdot \nabla_v (\chi(h)) + c h \chi'(h) \varphi \Big) \,\ds\dx\dv .
    \end{align*}
\end{definition}

In our equation \eqref{eq: lin}, this formula provides the nonnegativity of $F$ and $\gamma F$. Let us provide a proof in a slightly more general case.

\begin{lemma}[Maximum principle] \label{lem: nonnegative}
    Suppose there is a constant $\Lambda>0$ for which $A,\bfB$ satisfy \eqref{eq: Ellipticity and Bounds}. Let $(h,\gamma h)$ be a renormalized solution to
    \begin{align*}
    \begin{cases}
        \p_t h + v\cdot \nabla_x h = \nabla_v \cdot (A\nabla_v h) + \bfB \cdot \nabla_v h + ch \quad \text{in}\quad [0,T]\times\Omega\times\R^3,\\
        h|_{t=0} \ge 0,\\
        \gamma h|_{\Sigma_-^T} \ge 0.
    \end{cases}
    \end{align*}
    We assume that $c$ is possibly unbounded, but $\bfB$ admits a $v$-divergence $\nabla_v\cdot \bfB$ in the sense of distributions for which
    \begin{align*}
        |-\nabla_v\cdot \bfB + 2c| \le C < \infty.
    \end{align*}
    If $h \in L^1\cap L^\infty((0,T)\times\calO)$ and has sufficient decay at infinity, \textit{e.g.} $h=o(|v|^{-(d-1)/2})$, then
    \begin{align*}
        h\ge 0, \qquad \gamma h \ge 0.
    \end{align*}
\end{lemma}

\begin{proof}
    We take $\chi(h) := \mathbf{1}_{h<0} |h|^2$ in the renormalization formula. Since $h\in L^2((0,T)\times \calO)$, it is straightforward to check (through approximation) that $\varphi\equiv 1$ is admissible in the weak formulation of the renormalized equation as a test function. Let us notice that $\chi'' = 2 \cdot \mathbf{1}_{h<0} \ge 0$, and $\chi\lt(\gamma h|_{\Sigma_-^T} \rt) = \chi(h^0) \equiv 0$. Moreover, recalling the ellipticity condition \eqref{eq: Ellipticity and Bounds} on $A$, we find
    \begin{equation} \label{eq: nonnegativity}
    \begin{split}
        &\int_{\calO} \chi(h(t)) \,\dx\dv + \int_{\Sigma_+^t} (n(x)\cdot v)_+ \chi(\gamma h) \,\tnd\sigma\dv \\
        &\quad= \int_{(0,t)\times \calO} - 2 \cdot \mathbf{1}_{h<0} A\nabla_v h \cdot \nabla_v h + \bfB\cdot \nabla_v (\chi(h)) + c h \chi'(h)  \; \ds\dx\dv \\
        &\quad\le \int_{(0,t)\times \calO} \chi(h) (-\nabla_v\cdot \bfB) + c h \chi'(h) \; \ds\dx\dv \\
        &\quad= \int_{(0,t)\times\calO} \chi(h) (-\nabla_v \cdot \bfB + 2c) \; \ds\dx\dv \\
        &\quad\le C \int_{(0,t)\times \calO} \chi(h) \; \ds\dx\dv.
    \end{split}
    \end{equation}
    In the above, we integrated by parts with respect to $v$, which is justified due to the fact that $\bfB$ is bounded and $h$ decays sufficiently. Since $\chi\ge 0$, we can discard the boundary term on the left-hand side in order to apply Gr\"onwall's lemma, deducing that for each $t\in [0,T]$ it holds $\chi(h(t)) = 0$ a.e. $(x,v)\in\Omega\times\R^3$. Fubini's theorem then shows $\chi(h) = 0$ a.e. in $(0,T)\times\Omega\times\R^3$.
    
    Eq. \eqref{eq: nonnegativity} then simply reduces to
    \begin{align*}
        \int_{\Sigma_+^T} (n(x)\cdot v)_+ \chi(\gamma h) \;\ds\dx\dv \le 0,
    \end{align*}
    and this implies $\gamma h \ge 0$ in $\Sigma_+^T$, as desired.
\end{proof}

The application of Lemma \ref{lem: nonnegative} to \eqref{eq: lin} is clear. Indeed, in the notations of \eqref{eq: abc}, we simply have
\begin{align*}
    |-\nabla_v\cdot \bfB + 2c| = |c|,
\end{align*}
the latter of which is bounded (see \eqref{eq: abc bounds}). Hence $F\ge 0$ and $\gamma F \ge 0$.

Furthermore, in regard of the solution $(F,\gamma F)$ to \eqref{eq: lin}, by utilizing the fact that $F$ is bounded, we can use the renormalization formula to deduce $L^1\cap L^\infty$-bounds for $\gamma F$ as well. Namely, we can obtain that
\begin{equation} \label{lin: boundary Lp is uniform in X}
\begin{split}
&\textnormal{$\gamma F\in L^1 \cap L^\infty(\Sigma^T, |n(x)\cdot v| \dt\tnd\sigma\dv )$,} \textnormal{ independently of choice of $(\varrho,  \bfj, V) \in \mathscr{X}$}.
\end{split}
\end{equation}

As the arguments which lead to this are well-organized in \cite{Mischler2010, Mischler2000}, here we provide only the ideas. One considers the renormalization formula for \eqref{eq: lin} with nonlinearity $\chi(F) = F^p$ and $\varphi\equiv 1$ (done by approximation, recalling \eqref{eq: lin: Lp unif in X}). Splitting the boundary as $\Sigma^T = \Sigma_+^T \cup \Sigma_-^T$, and using that $g\in L^1\cap L^\infty(\Sigma_-^T,(n(x)\cdot v)_- \dt\tnd\sigma\dv)$, we can deduce an appropriate bound for $\chi(\gamma F) = |\gamma F|^p$. This can also be done in the following way. By methods such as convolution-translation regularizations \cite{BlouzaLeDret2001}, or by further regularizing the initial and boundary datum, we can assume $F$ is sufficiently regular, say H\"older continuous; then the trace of $F(t)$ on $\p\Omega\times\R^3$ is naturally well-defined, and it clearly coincides with $\gamma F(t)$. The usual trace theorem for Sobolev regular functions then shows \eqref{lin: boundary Lp is uniform in X}. To make this logic rigorous, one then exploits the lower semicontinuity of $L^p$-norms.

Let us now enter the details of the fixed point theorem.
%
%
%
%
%
%
\subsection{Fixed point operator}\label{sec: fixed}

Given $(\varrho, \bfj, V)\in \mathscr{X}$ let $(F,\gamma F)$ denote the solution to \eqref{eq: lin}. We define the operator $\mathscr{R}:\mathscr{X}\to \mathscr{X}$ with 
\begin{align*}
    \scrR (\varrho, \bfj, V) = (\rho_F, j_F, V_F) = \lt(\intr F\,\dv, \; \intr vF \,\dv, \; \frac{1}{d}\intr |v-u_F|^2 F\,\dv \rt).
\end{align*}
It is straightforward to observe that a solution to \eqref{eq: reg} is realized by a fixed point of $\mathscr{R}$.

Expanding the variance, a direct computation shows that
\begin{align}\label{eq: variance and second moment}
    V_F = \frac{1}{d}\lt( \intr |v|^2 F\,\dv - \frac{|j_F|^2}{\rho_F} \rt) \le \frac{1}{d} \intr |v|^2 F \,\dv,
\end{align}
and therefore in view of \eqref{eq: lin: decay} it is clear that $\rho_F, j_F,$ and $V_F$ are all integrable over $(0,T)\times\Omega$. Furthermore, since the solution $(F,\gamma F)$ to \eqref{eq: lin} is unique, we readily check that the operator $\scrR$ is well-defined.

The next lemmas demonstrate that all conditions of Schaefer's fixed-point theorem are verified.

\begin{lemma} \label{lem: R continuous}
    The map $\mathscr{R}: \scrX\to \scrX$ is continuous.
\end{lemma}

\begin{proof}
    It is enough to show sequential continuity. We consider a sequence for which
    \begin{equation} \label{lin: continuity assumption}
(\varrho_n, \bfj_n, V_n) \to (\varrho , \bfj, V) \quad \text{in }\scrX.
    \end{equation}
    Denote by $(F_n, \gamma F_n)$ the unique solution to \eqref{eq: lin} which corresponds to $(\varrho_n, \bfj_n, V_n)$. Analogously, denote by $(F, \gamma F)$ the solution corresponding to $(\varrho , \bfj, V)$.

    Let us pass to a subsequence for which the convergences in \eqref{lin: continuity assumption} hold almost everywhere; later we will show that this passage to the subsequence is unnecessary. By the dominated convergence theorem, we can then first remark that
    \begin{equation}\label{lin: continuity renorm}
    \begin{split}
        &\renorm{V}{\varrho_n}{\ve} \to \renorm{V}{\varrho}{\ve} \quad \text{in }L^p((0,T)\times\Omega),\\
        &\renorm{\bfj_n}{\varrho_n}{\ve} \to \renorm{\bfj}{\varrho}{\ve} \quad \text{in }L^p((0,T)\times\Omega) \quad \forall p\in [1,\infty). 
    \end{split}
    \end{equation}
    Similarly, by the continuity and boundedness assumption of $\nu$, we second note that
    \begin{equation}\label{lin: continuity nu}
        \nu(\varrho_n, \bfj_n, V_n) \to \nu(\varrho, \bfj, V) \quad \textnormal{a.e. and in }L^p((0,T)\times\Omega), \quad \forall p\in [1,\infty).
    \end{equation}
    
    Next, in virtue of \eqref{eq: lin: decay} and \eqref{lin: boundary Lp is uniform in X} it is clear that $F_n$ (respectively $\gamma F_n)$) is weakly$^*$ compact in $L^\infty(0,T;L^1\cap L^\infty(\calO))$ (respectively $L^1\cap L^\infty(\Sigma^T, |n(x)\cdot v|\dt\tnd\sigma\dv)$). Passing to the weakly$^*$ convergent subsequence, in view of the strong convergences in \eqref{lin: continuity renorm} and \eqref{lin: continuity nu}, we can straightforwardly pass to the limit in \eqref{eq: lin} (that is satisfied by the $(F_n, \gamma F_n)$) to deduce that the weak$^*$ limit is a solution to \eqref{eq: lin} with $( \varrho, \bfj, V)$. Owing to the uniqueness property in Lemma \ref{lem: zhu}, we deduce that the weak$^*$ limit is necessarily $(F, \gamma F)$. Moreover, since the limit is uniquely identified along any weakly$^*$ convergent subsequence, we deduce that $(F_n, \gamma F_n) \to (F, \gamma F)$ weakly$^*$ without having to pass to subsequences.
    
Then, we notice from \eqref{eq: lin: decay} that the application of classical averaging lemmas (see for example \cite{dipernalionsMaxwell1989, DipernaLionsMeyer, perthamesouganidis}) show that for any $\varphi \in C^0(\R^3)$ with at most polynomial growth, the sequence of averages
    \begin{align}\label{eq: poly vel avg}
        \left\{\int_{\R^3} \varphi(v) F_n(t,x,v)\,\dv \right\}
    \end{align}
    is strongly compact in $L^1((0,T)\times\Omega)$. Choosing $\varphi(v) = 1,v_i$ in \eqref{eq: poly vel avg} we deduce that $\rho_{F_n}$ and $j_{F_n}$ are compact in $L^1((0,T)\times\Omega)$. To obtain the compactness of $V_{F_n}$ we argue as follows. Taking $\varphi(v) = |v|^2$ we have that $\left\{\int_{\R^3} |v|^2 F_n\,\dv\right\}_n$ is strongly compact in $L^1((0,T)\times\Omega)$. Then, from \eqref{eq: variance and second moment} we recall
    \begin{align*}
        V_{F_n} &:= \frac{1}{d} \int_{\R^3} |v - u_{F_n}|^2 F_n\,\dv  = \frac{1}{d}\lt( \int_{\R^3} |v|^2 F_n\,\dv - \frac{|j_{F_n}|^2}{\rho_{F_n}} \rt),
    \end{align*}
    and that all terms on the right-hand side are pointwise convergent (up to some further subsequence). It follows that $V_{F_n}$ itself is a.e. convergent. Also, since
    \begin{align*}
        0 \le V_{F_n} \le \int_{\R^3} |v|^2 F_n\,\dv,
    \end{align*}
    the right-hand side of which is strongly convergent in $L^1((0,T)\times\Omega)$, we deduce via the Vitali convergence theorem that $V_{F_n}$ is strongly convergent in $L^1((0,T)\times\Omega)$.
    
    To conclude the proof, we recall that $F_n\weakstarto F$ \textbf{without} passing to subsequences. Thus any convergent subsequences of $\{\rho_{F_n}\}, \{j_{F_n}\},$ and $\{V_{F_n}\}$ necessarily converge to $\rho_F, j_F$, and $V_F$. Again we deduce that, from the unique identification of the limit, it must be that the full sequences $\rho_{F_n}, j_{F_n}$, and $V_{F_n}$ are strongly convergent in $L^1((0,T)\times\Omega)$ to $\rho_F, j_F$, and $V_F$. The proof is complete.
\end{proof}

\begin{remark}
    Strictly speaking, velocity averaging lemmas such as those stated in \cite{DipernaLionsMeyer} are stated for kinetic equations which are imposed to hold in the whole domain $\calD'((0,T)\times \R^d \times \R^d)$. These lemmas have been subsequently employed to general open bounded domains, but the authors could not locate a suitable reference providing an explicit statement. One possible method is to use cutoff functions near the spatial boundary and extend the localized equation to the whole space, as in the related argument of \cite[Section 5]{TT07}. Another argument can proceed as follows. 
    
    One represents $\Omega$ appropriately as the union of countably many open balls $\{\Omega_i\}_{i=1}^\infty$ compactly contained in $\Omega$. Then we take a partition of unity $\{\eta_i\}_{i=1}^\infty$ subordinate to the $\Omega_i$, satisfying
    \begin{align}\label{eq: partition of unity}
        \eta_i \in C_c^\infty(\Omega_i), \qquad \sum_{i=1}^\infty \eta_i(x) = 1 \quad \forall x\in \Omega.
    \end{align}
    Let us suppose we have a sequence $F_m = F_m(t,x,v)$ satisfying an appropriate system of kinetic equations in $\calD'((0,T)\times\Omega\times\R^d)$. Since $\textnormal{supp}(\eta_i) \Subset \Omega$, the function $\eta_i(x)F_m(t,x,v)$ satisfies an appropriate kinetic equation in $\calD'((0,T)\times\R^d\times\R^d)$. Hence, for each fixed $i$, the whole-space velocity averaging lemma implies, for every $\varphi\in C_c^\infty(\R^d_v)$, the compactness in $L^1((0,T)\times\Omega)$ of 
\[ 
\lt\{ \int_{\R^d}\varphi(v)\eta_i(x)F_m(t,x,v)\,\dv \rt\}_m. 
\] 
To conclude, for $\delta>0$, set 
\[ 
\Omega_\delta := \{x\in\Omega:\textnormal{dist}(x,\partial\Omega)>\delta\}. 
\] 
Since $\overline{\Omega_\delta}\Subset\Omega$ and the partition of unity is locally finite, there exists a finite set $I_\delta\subset \N$ such that 
\[ 
\sum_{i\in I_\delta}\eta_i(x)=1 \quad\text{for all }x\in\Omega_\delta. 
\] 
Thus, 
\[ 
\int_{\R^d}\varphi(v)F_m(t,x,v)\,\dv = \sum_{i\in I_\delta} \int_{\R^d}\varphi(v)\eta_i(x)F_m(t,x,v)\,\dv 
\] 
on $(0,T)\times\Omega_\delta$. Since the sum is finite, the velocity averages are strongly compact in $L^1((0,T)\times\Omega_\delta)$. On the other hand, the contribution from $\Omega\setminus\Omega_\delta$ tends to zero uniformly in $m$ as $\delta\downarrow0$, owing to the uniform $L^\infty$ bound and velocity decay of $F_m$. This proves the compactness of 
\[ 
\lt\{ \int_{\R^d}\varphi(v)F_m(t,x,v)\,\dv \rt\}_m 
\] 
in $L^1((0,T)\times\Omega)$. By a standard truncation argument using the uniform velocity decay, the same conclusion holds for continuous functions $\varphi$ with the polynomial growth considered above.
\end{remark}

\begin{lemma}\label{lem: R comp}
    The map $\mathscr{R}:\scrX\to \scrX$ is compact.
\end{lemma}

\begin{proof}
    Since the proof is much easier than showing continuity of $\mathscr{R}$, we provide only a rough sketch of the ideas. Let $(\varrho_n, \bfj_n, V_n)$ be bounded in $\scrX$. Denote by $(F_n, \gamma F_n)$ the corresponding solution to \eqref{eq: lin}. We want to show that
    \begin{align*}
        \mathscr{R}(\varrho_n, \bfj_n, V_n) := (\rho_{F_n}, j_{F_n}, V_{F_n} )
    \end{align*}
    is convergent in $\scrX$ modulo some subsequence. As before, it is enough to show that $\left\{\int_{\R^3}\varphi(v) F_n\,\dv\right\}_n$ is strongly compact in $L^1((0,T)\times\Omega)$ for $\varphi(v) = 1,v,|v|^2$, which follows easily owing to \eqref{eq: lin: decay} and averaging lemmas.
\end{proof}

Finally, it is straightforward to check that

\begin{lemma} \label{lem: R eigen}
    The set of eigenvectors
    \begin{align*}
        \scrX_\scrR := \lt\{(\varrho, \bfj, V)\in \scrX \; \middle| \; (\varrho, \bfj, V) = \lambda \mathscr{R}(\varrho, \bfj, V) \text{ for some }\lambda\in [0,1] \rt\}
    \end{align*}
    is bounded in $\scrX$.
\end{lemma}

We can now state the existence of solutions to \eqref{eq: reg}.

\begin{lemma} \label{lem: ve weak sol}
    There exists a weak solution $(f_\ve, \gamma f_\ve)$ to the equation \eqref{eq: reg}, in the sense that
    \begin{align*}
        (f_\ve, \gamma f_\ve) \in L^2((0,T)\times \Omega; H^1_v(\R^d)) \times L^2_{\rm loc}(\Sigma^T, |n(x)\cdot v| \dt\tnd\sigma\dv)
    \end{align*}
    satisfies for each $t\in (0,T]$ and $\varphi\in C_c^\infty([0,t]\times\overline\Omega\times\R^d)$:
    \begin{equation*}
    \begin{split}
    &\int_{\calO} f_\ve(t) \varphi(t,x,v) \,\dx\dv - \int_{\calO} f_\ve^0 \varphi(0,x,v)\,\dx\dv + \int_{\Sigma_+^t} \gamma f_\ve \, \varphi \, (n(x)\cdot v)_+ \ds\tnd\sigma\dv\\
    &\quad = - \int_{\Sigma_-^t} g_\ve \, \varphi \, (n(x)\cdot v)_- \ds\tnd\sigma\dv +\int_{(0,t)\times \calO} f_\ve (\p_s \varphi + v\cdot \nabla_x \varphi) \,\ds\dx\dv\\
    &\qquad + \int_{(0,t)\times \calO} (\nu_{f_\ve} + \ve) \lt(\Delta_v \varphi \calT_{f_\ve}^{(\ve)} f_\ve - \nabla_v\varphi \cdot \lt(\renorm{v}{1}{\ve}-u_{f_\ve}^{(\ve)} \rt) f_\ve \rt) \,\ds\dx\dv.
    \end{split}
    \end{equation*}
    Furthermore, the solution is renormalized, and for each $k\in \bbN$ it satisfies for some constant $C_{T,d,k,\ve}>0$
    \begin{align}\label{eq: ve decay}
        f_\ve(t,x,v) \le \frac{C_{T,d,k,\ve}}{(1+|v|)^k}.
    \end{align}
    In particular,
    \begin{equation}\label{eq: ve memberships}
    \begin{split}
        f_\ve \in L^\infty(0,T; L^1\cap L^\infty(\calO)),  \quad \gamma f_\ve \in L^1 \cap L^\infty (\Sigma^T, |n(x)\cdot v| \dt \tnd\sigma \dv).
        \end{split}
    \end{equation}
\end{lemma}

\begin{proof}[Proof of Lemma \ref{lem: ve weak sol}]
    Collecting Lemmas \ref{lem: R continuous}, \ref{lem: R comp}, \ref{lem: R eigen} we find that Schaefer's fixed-point theorem is applicable. The properties regarding renormalization, and the regularity imposed above, follow directly from \eqref{eq: lin: decay}, \eqref{eq: lin: Lp unif in X}, \eqref{lin: boundary Lp is uniform in X}, and the fact that the solution $(f_\ve, \gamma f_\ve)$ is realized through a fixed point of the linear problem \eqref{eq: lin}.
\end{proof}

%
%
%
%
%
%

\subsection{Uniform estimates} \label{subsec: uniform estimates}
We begin this subsection by providing energy estimates which are uniform in the regularization parameter $\ve$. 

\begin{lemma}\label{lem: energy}
    Let $(f_\ve, \gamma f_\ve)$ denote the solution to the regularized problem \eqref{eq: reg} as constructed in Lemma \ref{lem: ve weak sol}. Then
    \begin{align*}
        &\textnormal{$\{(1+|v|^2)f_\ve\}_{\ve\in (0,1]}$ is uniformly bounded in $L^\infty(0,T;L^1(\calO))$}, \\
        &\textnormal{$\{(1+|v|^2)\gamma f_\ve\}_{\ve\in (0,1]}$ is uniformly bounded in $L^1(\Sigma^T, |n(x)\cdot v|\dt\tnd\sigma \dv)$}.
    \end{align*}
\end{lemma}

\begin{proof}
Owing to \eqref{eq: lin: decay}, we readily notice that $\varphi(v)=1+|v|^2$ is admissible as a test function into the (very) weak formulation of \eqref{eq: reg}. The rigorous derivation can be done by taking the standard method of using smooth cutoffs, which we omit here and refer to for instance \cite{Abdallah1994}. The identity that is obtained is
\begin{equation}\label{eq: energy est}
\begin{split}
    &\int_{\calO} (1+|v|^2) f_\ve(t) \,\dx\dv - \int_{\calO} (1+|v|^2) f_\ve^0 \,\dx\dv + \int_{\Sigma_+^t} (1+|v|^2) \gamma f_\ve (n(x)\cdot v)_+ \ds\tnd\sigma\dv \\
    &\quad= - \int_{\Sigma_-^t} (1+|v|^2) g_\ve (n(x)\cdot v)_- \ds\tnd\sigma\dv\\
    &\qquad + \int_{(0,t)\times \calO} (\nu_{f_\ve} + \ve) \lt(2 d \calT_{f_\ve}^{(\ve)} f_\ve - 2 v \cdot (\renorm{v}{1}{\ve} - u_{f_\ve}^{(\ve)}) f_\ve \rt) \; \ds\dx\dv.
\end{split}
\end{equation}
It is enough to estimate the last integral of the right-hand side. We note that $\calT_{f_\ve}^{(\ve)} \le \calT_{f_\ve} + 1$ holds for all $\ve\in (0,1]$, which gives:
\begin{align*}
    \int_{(0,t)\times \calO} (\nu_{f_\ve} + \ve) \calT_{f_\ve}^{(\ve)} f_\ve  \;\ds\dx\dv &\le C(\|\nu\|_{L^\infty}) \int_{(0,t)\times \calO} \calT_{f_\ve}^{(\ve)} f_\ve \,\ds\dx\dv \quad (\because \nabla_v\cdot \renorm{v}{1}{\ve} = O(1) ) \\
    &\le C(\|\nu\|_{L^\infty}) \int_{(0,t)\times \calO} (1 + \calT_{f_\ve}) f_\ve \,\ds\dx\dv \quad (\because \eqref{eq: reg major}) \\
    &\le C \int_{(0,t)\times \calO} (1 + |v|^2) f_\ve \,\ds\dx\dv.
\end{align*}
Next, we estimate
\begin{align*}
    &\int_{(0,t)\times \calO} (\nu_{f_\ve} + \ve) (-2v\cdot (\renorm{v}{1}{\ve} -u_{f_\ve}^{(\ve)}) f_\ve) \,\ds\dx\dv\\
    &\quad \le C(\|\nu\|_{L^\infty}) \int_{(0,t)\times \calO} \lt( |v|^2 f_\ve + |v| |u_{f_\ve}^{(\ve)}| f_\ve \rt) \,\ds\dx\dv\\
    &\quad \le C(\|\nu\|_{L^\infty}) \int_{(0,t)\times \calO} \lt(\frac{3}{2}|v|^2 f_\ve + |u_{f_\ve}^{(\ve)}|^2 f_\ve \rt) \,\ds\dx\dv\\
    &\quad \le C(\|\nu\|_{L^\infty}) \int_{(0,t)\times \calO} \lt(|v|^2 f_\ve + |u_{f_\ve}|^2 f_\ve \rt) \,\ds\dx\dv \quad (\because \eqref{eq: reg major}) \\
    &\quad \le C(\|\nu\|_{L^\infty}) \int_{(0,t)\times \calO} |v|^2 f_\ve \,\ds\dx\dv,
\end{align*}
the last line following by Jensen's inequality $\rho_{f_\ve}|u_{f_\ve}|^2 \le \intr |v|^2 f_\ve\dv$. Collecting these estimates, we deduce via \eqref{eq: energy est} and Gr\"onwall's lemma that
\begin{align*}
    \sup_{t\in [0,T]} \int_{\calO} (1+|v|^2) f_\ve \,\dx\dv \le C_T,  \quad \int_{\Sigma_+^T} (1+|v|^2) \gamma f_\ve (n(x)\cdot v)_+ \ds\tnd\sigma\dv \le C_T,
\end{align*}
with the constant $C_T>0$ dependent on $g$ and $f^0$ but independent of $\ve$ (recall the uniform bounds that are due to \eqref{ASSUMP f0 g}).
\end{proof}

However, the lemma above is not enough for the passage to $\ve\to 0$, as it does not provide a method for us to obtain compactness for the temperature in any way. For that reason we make the observation that multiplier methods \cite{Perthame1992, perthame1996, Perthame2004}, originally stated for the transport equation and utilized thorougly in the study of the BGK equation \cite{ChenZhang2016, ChoiYun20}, can be adapted to kinetic Fokker--Planck equations under fairly mild assumptions. We provide the statement in rather general form for the sake of future applications.

\begin{lemma}[Higher moments] \label{lem: higher moments}
    Let $(h_m,\gamma h_m)$ denote a sequence of nonnegative very weak solutions to 
    \begin{align*}
    \begin{cases}
    \p_t h_m + v\cdot \nabla_x h_m = a_m(t,x) \Delta_v h_m + b_m(t,x) \nabla_v\cdot (\bfc_m(t,x,v) h_m) + d_m(t,x,v) h_m,\\
    h_m|_{t=0} = h_m^0
    \end{cases}
    \end{align*}
    in $(0,T)\times \calO$. Suppose that for some $k \ge 1$:
    \begin{align*}
        &\textnormal{$\{h_m^0\}$ is bounded in $L^1 \lt( \calO, \, \lt<v\rt>^k\dx\dv \rt)$},  \\
        &\textnormal{$\{h_m\}$ is bounded in $L^\infty\left( 0,T ;L^1 \lt(\calO, \lt<v\rt>^k \dx\dv \rt) \right)$}, \\
        &\textnormal{$\{a_m h_m\}$ is bounded in $L^1\left((0,T)\times\calO, \lt<v\rt>^{k-2}\dt\dx\dv\right)$}, \\
        &\textnormal{$\{b_m |\bfc_m| h_m\}$ is bounded in $L^1\left((0,T)\times\calO, \,\lt<v\rt>^{k-1}\dt\dx\dv\right)$}, \\
        &\textnormal{$\{d_m h_m\}$ is bounded in $L^1\lt((0,T)\times\calO, \, \lt<v\rt>^k\dt\dx\dv \rt)$}, \\
        &\textnormal{$\{\gamma h_m\}$ is bounded in $L^1 \lt( \Sigma^T, \lt<v\rt>^k |n(x)\cdot v|\dt\tnd\sigma\dv \rt)$}. 
    \end{align*}
    Then we have
    \begin{align*}
        \left\{\lt<v\rt>^{k+1} h_m \right\}_m \textnormal{ is bounded in $L^1((0,T)\times\Omega\times\R^d)$}.
    \end{align*}
\end{lemma}
\begin{proof}
The proof is an adaptation of the results in \cite{Perthame1992}. For a fixed $x_0\in \Omega$, we admit
\begin{align*}
    \varphi(x,v) = \lt<v\rt>^{k-1} \frac{v\cdot (x-x_0)}{\lt<x-x_0\rt>}
\end{align*}
as a test function in the weak formulation for $h_m$. The overcoming of the non-compact support of $\varphi$ can be done through the usual cutoff procedure, which is well-illustrated in say \cite[Lemma 2.2]{ChenZhang2016} and thus omitted here. By using the computation
\begin{align*}
    &v\cdot \nabla_x\varphi = \frac{\lt<v\rt>^{k+1}}{\lt<x-x_0\rt>} \left(1 - \frac{(v\cdot (x-x_0))^2}{|v|^2\lt<x-x_0\rt>^2}\right),
\end{align*}
we obtain that
\begin{equation}\label{eq: higher moments estimate}
\begin{split}
    &\int_{(0,T)\times \calO} h_m \frac{\lt<v\rt>^{k+1}}{\lt<x-x_0\rt>} \left(1 - \frac{(v\cdot (x-x_0))^2}{|v|^2 \lt<x-x_0\rt>^2} \right) \;\dt\dx\dv \\
    &\quad= \int_{(0,T)\times \calO} \bfc_m(t,x,v) b_m(t,x) h_m \cdot \nabla_v \varphi \;\dt\dx\dv - \int_{(0,T)\times\calO} a_m(t,x) h_m \Delta_v\varphi \;\dt\dx\dv \\
    &\qquad + \int_{(0,T)\times \calO} d_m h_m \varphi \;\dt\dx\dv + \int_{\calO} (h_m(T) - h_m(0)) \,\varphi \;\dx\dv + \int_{\Sigma^T} \gamma h_m \, \varphi \, v\cdot n(x) \;\dt\tnd\sigma\dv \cr
    &\quad=: \sum_{i=1}^5 R_i.
\end{split}
\end{equation}
We find a lower-bound for the left-hand side first. Since
\begin{align*}
    1 - \frac{(v\cdot (x-x_0))^2}{|v|^2 \lt<x-x_0\rt>^2} \ge 1 - \frac{|x-x_0|^2}{\lt<x-x_0\rt>^2} = \frac{1}{1+|x-x_0|^2} \ge \frac{1}{1 + [\textnormal{diam}(\Omega)]^2 },
\end{align*}
we discover that the left-hand side of \eqref{eq: higher moments estimate} satisfies
\begin{align*}
    &\int_{(0,T)\times\calO} h_m \frac{\lt<v\rt>^{k+1}}{\lt<x-x_0\rt>} \left(1 - \frac{(v\cdot (x-x_0))^2}{|v|^2 \lt<x-x_0\rt>^2} \right) \dt\dx\dv  \ge \frac{1}{1+[\textnormal{diam}(\Omega)]^3} \int_{(0,T)\times\calO} \lt<v\rt>^{k+1} h_m \;\dt\dx\dv.
\end{align*}
The remaining terms are estimated by using:
\begin{align*}
        \varphi &= O(|v|^k), \quad  \nabla_v\varphi  = \left( (k-1) \lt<v \rt>^{k-3} \frac{v\cdot (x-x_0)}{\lt<x-x_0\rt>} v + \lt<v\rt>^{k-1} \frac{x-x_0}{\lt<x-x_0\rt>} \right) ,\\
        \Delta_v \varphi &= \left((k-1)(k-3)\left<v\right>^{k-5}|v|^2 \frac{v\cdot (x-x_0)}{\left<x-x_0\right>} + 4(k-1) \left<v\right>^{k-3} \frac{v\cdot (x-x_0)}{\left<x-x_0\right>} \right) ,
    \end{align*}
which lead immediately to
\begin{align*}
    &R_1 \lesssim_k \int_{(0,T)\times\calO} |b_m \bfc_m| \lt<v\rt>^{k-1} h_m \;\dt\dx\dv \le C, \quad R_2 \lesssim_k \int_{(0,T)\times \calO} |a_m h_m| \lt<v\rt>^{k-2} \;\dt\dx\dv \le C,\\
    &R_3 \le \int_{(0,T)\times \calO} \lt<v\rt>^k |d_m h_m| \;\dt\dx\dv \le C,  \quad R_4 \le \int_{\calO} \lt<v\rt>^k (h_m(T) - h_m(0)) \; \dx\dv \le C, \\
    &R_5 \le \int_{\Sigma^T} |v|^k \gamma h_m |v\cdot n(x)| \; \dt\tnd\sigma\dv \le C,
\end{align*}
by assumption on the coefficients. Collecting the estimates we deduce the result.
\end{proof}

Particularly, it is thus deduced that

\begin{lemma}\label{lem: ve higher moments}
    Let $(f_\ve,\gamma f_\ve)$ denote the solution to \eqref{eq: reg} as constructed in Lemma \ref{lem: ve weak sol}. Then
    \begin{align*}
        \sup_{\ve \in (0,1]} \int_{(0,T)\times \calO} |v|^3 f_\ve(t) \;\dt\dx\dv \le C_T
    \end{align*}
    for some constant $C_T>0$.
\end{lemma}

\begin{proof}
We set (in the notations of Lemma \ref{lem: higher moments})
\begin{align*}
    \begin{cases}
        a_\ve(t,x) = (\nu_{f_\ve} + \ve) \calT_{f_\ve}^{(\ve)} , \\
        b_\ve(t,x) = (\nu_{f_\ve} + \ve), \\
        \bfc_\ve(t,x,v) = \renorm{v}{1}{\ve} - u_{f_\ve}^{(\ve)}, \\
        d_\ve(t,x,v) \equiv 0.
    \end{cases}
\end{align*}
Owing to Lemma \ref{lem: energy}, to show that all conditions of Lemma \ref{lem: higher moments} are satisfied (with $k=2$), it is enough to obtain bounds for $a_\ve f_\ve$ and $b_\ve |\bfc_\ve| f_\ve$. Now \eqref{eq: reg major} shows
\begin{align*}
    &|a_\ve f_\ve| \le  (\|\nu\|_{L^\infty} + 1) (1 + \calT_{f_\ve}) f_\ve, \\
    &|b_\ve |\bfc_\ve| f_\ve| \le (\|\nu\|_{L^\infty} + 1) (| \renorm{v}{1}{\ve}| + |u_{f_\ve}^{(\ve)}|) f_\ve \le (\|\nu\|_{L^\infty}+1) (|v| + |u_{f_\ve}|) f_\ve
\end{align*}
Then clearly for any $\ve\in (0,1]$
\begin{align*}
    \int_{(0,T)\times \calO} |a_\ve f_\ve| \;\dt\dx\dv  &\le C(\|\nu\|_{L^\infty}) \lt(\int_{(0,T)\times \calO} f_\ve + \int_{(0,T)\times \Omega} V_{f_\ve} \rt)\dt\dx\dv \\
    &\le C(\|\nu\|_{L^\infty}) \int_{(0,T)\times \calO} (1+|v|^2) f_\ve  \;\dt\dx\dv \\
    &\le C(\|\nu\|_{L^\infty}, f^0).
\end{align*}
Next, we estimate
\begin{align*}
    &\int_{(0,T)\times \calO} b_\ve |\bfc_\ve| f_\ve \lt<v\rt> \;\dt\dx\dv\\
    &\quad\le (\|\nu\|_{L^\infty}+1) \lt(\int_{(0,T)\times \calO} (1+|v|^2) f_\ve + \int_{(0,T)\times \calO} |v| |u_{f_\ve}| f_\ve \rt) \dt\dx\dv \\
    &\quad\le C(\|\nu\|_{L^\infty}, f^0) \lt(1 + \lt(\int_{(0,T)\times \calO} |v|^2 f_\ve \,\dt\dx\dv \rt)^{1/2} \lt(\int_{(0,T)\times \Omega} \rho_{f_\ve} |u_{f_\ve}|^2 \,\dt\dx \rt)^{1/2} \rt) \\
    &\quad\le C(\|\nu\|_{L^\infty},f^0) \int_{(0,T)\times\calO} (1+|v|^2) f_\ve \,\dt\dx\dv \\
    &\quad\le C(\|\nu\|_{L^\infty}, f^0) .
\end{align*}
From the second line, we used Lemma \ref{lem: energy} and the Cauchy--Schwarz inequality to get the third line, then $\rho_{f_\ve}|u_{f_\ve}|^2 \le \intr |v|^2 f_\ve$ (Jensen's inequality) to obtain the fourth, and finally Lemma \ref{lem: energy} again to deduce the last bound.
\end{proof}

As one final step before passage to the limit, we now prove the entropy identity.

\begin{lemma} \label{lem: entropy}
    The solution $(f_\ve,\gamma f_\ve)$ to the regularized problem \eqref{eq: reg} satisfies the entropy evolution identity
    \begin{align*}
        &\ddt \int_{\calO} f_\ve(t) \log f_\ve(t) \,\dx\dv + \int_{\Gamma_+} \gamma f_\ve(t) \log \gamma f_\ve(t) (n(x)\cdot v)_+ \,\tnd\sigma\dv + 4 \int_{\calO} (\nu_{f_\ve} + \ve) \calT_{f_\ve}^{(\ve)} \lt|\nabla_v\sqrt{f_\ve}\rt|^2 \,\dx\dv \\
        &\quad = - \int_{\Gamma_-} g_\ve(t) \log g_\ve(t) (n(x)\cdot v)_- \,\tnd\sigma\dv  +  \int_{\calO} (\nu_{f_\ve} + \ve) f_\ve(t) \lt(\nabla_v\cdot \renorm{v}{1}{\ve} \rt) \,\dx\dv.
    \end{align*}
    In particular, we have
    \begin{align*}
        &\sup_{t\in [0,T]} \int_{\calO} f_\ve(t) \log f_\ve(t) \,\dx\dv \le C_T, \\
        &\int_{\Sigma_+^T} \gamma f_\ve \log \gamma f_\ve \, (n(x)\cdot v)_+ \,\ds\tnd\sigma\dv \le C_T , \\
        &\int_{(0,T)\times \calO} (\nu_{f_\ve} + \ve) |\nabla_v \sqrt{f_\ve}|^2 \,\ds\dx\dv \le C_T
    \end{align*}
    for some $C_T>0$ independent of $\ve$.
\end{lemma}

\begin{proof}
    Let us first remark that the decay \eqref{eq: ve decay}, and the assumption that $\Omega$ is bounded, imply that $f_\ve \log f_\ve \in L^\infty(0,T;L^1(\Omega\times\R^d))$ and $\gamma f_\ve \log \gamma f_\ve \in L^1(\Sigma^T,|n(x)\cdot v|\dt\tnd\sigma\dv)$ (although not necessarily uniformly in $\ve$). This can be seen, for instance, by following the classical estimate performed in \cite[pp.7--8]{JKO98}. 

    Now, although the solution $f_\ve$ is renormalized, we notice that we cannot take $z\mapsto z\log z$ as a renormalizing function since the Hessian of $z\log z$ blows up near $z=0$. From this perspective we take the ``regularized'' renormalizing function
    \begin{align*}
        \chi_\delta(z) := z \log (\delta + z) + \delta \log\left(1 + \frac{z}{\delta}\right), \quad z\ge 0, \quad \delta\in (0,1].
    \end{align*}
    In this way, the classical inequality $\log(1+z)\le z$ shows
    \begin{align}\label{eq: chidelta ineq}
        \chi_\delta(f_\ve) \le \frac{(f_\ve)^2}{\delta} + f_\ve (\log \delta + 1),
    \end{align}
    and therefore $\chi_\delta(f_\ve)$ enjoys the same tail decay as in \eqref{eq: ve decay}. Next, by direct computations:
    \begin{align*}
        &\chi_\delta'(z) = \log(\delta+z) + 1, \quad \chi_\delta''(z) = \frac{1}{\delta+z}.
    \end{align*}
    Hence, we readily find that $(\chi_\delta(z),1)$ is an admissible pair in the renormalization formula for $f_\ve$. The identity that is obtained is
    \begin{equation} \label{eq: ent prod}
    \begin{split}
        &\int_{\calO} \chi_\delta(f_\ve)(t)\,\dx\dv - \int_{\calO} \chi_\delta(f_\ve^0) \,\dx\dv + \int_{\Sigma_+^t} \chi_\delta(\gamma f_\ve) (n(x)\cdot v)_+ \,\ds\tnd\sigma\dv\\
        &\quad + \int_{(0,t)\times\calO} (\nu_{f_\ve} + \ve) \calT_{f_\ve}^{(\ve)} \frac{4f_\ve}{\delta+f_\ve} |\nabla_v \sqrt{f_\ve}|^2 \,\ds\dx\dv \\
        &= -\int_{\Sigma_-^t} \chi_\delta(g_\ve) (n(x)\cdot v)_- \ds\tnd\sigma\dv - \int_{(0,t)\times \calO} (\nu_{f_\ve} + \ve) \frac{1}{\delta+f_\ve} \nabla_v f_\ve \cdot ( \renorm{v}{1}{\ve} -u_{f_\ve}^{(\ve)} ) f_\ve \;\ds\dx\dv.
    \end{split}
    \end{equation}
    We now investigate the limit of each term in \eqref{eq: ent prod} as $\delta\downarrow 0$. In regard of $\chi_\delta$, we first observe that pointwise:
\begin{equation*}
\begin{split}
    |\chi_\delta(z) - z \log z| &= \lt| z \log \lt(1 + \frac{\delta}{z}\rt) + \delta \log \lt(1+\frac{z}{\delta}\rt) \rt|  = \lt|z \log \lt(1 + \frac{\delta}{z}\rt) + \delta \log (\delta + z) - \delta \log \delta \rt| 
    \end{split}
\end{equation*}
tends to zero as $\delta\downarrow 0$. In fact, from
\begin{align*}
    \frac{\p}{\p \delta} \chi_\delta(z) = \log \lt(1 + \frac{z}{\delta}\rt) \ge 0,
\end{align*}
we find that the function $\chi_\delta(z)$ decreases to $z\log z$ as $\delta\downarrow 0$. Therefore, to apply the dominated convergence theorem to the integrals $\int \chi_\delta(\cdot)$, it is sufficient to prove that the integral is finite when $\delta=1$ (the maximal value). It is clearly from \eqref{eq: chidelta ineq} and \eqref{eq: ve decay} that
\begin{align*}
    \int_{\calO} \chi_1(f_\ve)(t) \,\dx\dv \le \int_{\calO} \Big( |f_\ve(t)|^2 + f_\ve(t) \Big) \,\dx\dv \le C_\ve,
\end{align*}
and thus we conclude
\begin{align*}
    \int_{\calO} \chi_\delta(f_\ve)(t) \,\dx\dv \xrightarrow[\delta\downarrow 0]{} \int_{\calO} f_\ve(t) \log f_\ve(t) \,\dx\dv.
\end{align*}
In precisely the same manner, using \eqref{eq: ve memberships}, we obtain that
\begin{align*}
    \int_{\Sigma_+^t} \chi_\delta(\gamma f_\ve) (n(x)\cdot v)_+ \,\ds\tnd\sigma\dv \xrightarrow[\delta\downarrow 0]{} \int_{\Sigma_+^t} \gamma f_\ve \log \gamma f_\ve (n(x)\cdot v)_+ \,\ds\tnd\sigma\dv.
\end{align*}
The convergences of $\int_{\calO} \chi_\delta(f_\ve^0)$ and $\int_{\Sigma_-^t} \chi_\delta(g_\ve)(n(x)\cdot v)_-$ follow analogously.

Next we discuss the weighted Fisher information, in other words the second line of \eqref{eq: ent prod}. Since $\frac{4z}{\delta + z} \uparrow 4$ as $\delta\downarrow 0$, it is immediately deduced via the monotone convergence theorem that
    \begin{align*}
         \lim_{\delta\downarrow 0} \int_{(0,t)\times \calO} (\nu_{f_\ve} + \ve) \calT_{f_\ve}^{(\ve)} \frac{4f_\ve}{\delta+f_\ve} |\nabla_v \sqrt{f_\ve}|^2 \,\ds\dx\dv  = 4\int_{(0,t)\times \calO} (\nu_{f_\ve} + \ve) \calT_{f_\ve}^{(\ve)} |\nabla_v \sqrt{f_\ve}|^2 \,\ds\dx\dv.
    \end{align*}

Finally, let us focus on the last integral of the right-hand side of \eqref{eq: ent prod}. We use the simple algebraic manipulation
    \begin{align*}
        \frac{1}{\delta + f_\ve} f_\ve \nabla_v f_\ve  &= \nabla_v \Big( f_\ve - \delta \log (\delta + f_\ve) \Big) = \nabla_v \Big( f_\ve - \delta \log (\delta + f_\ve) + \delta \log \delta \Big)
    \end{align*}
    in order to rewrite the integral as
    \begin{align}
        &-\int_{(0,t)\times \calO} (\nu_{f_\ve} + \ve) \frac{1}{\delta+f_\ve} \nabla_v f_\ve \cdot ( \renorm{v}{1}{\ve} -u_{f_\ve}^{(\ve)} ) f_\ve \;\ds\dx\dv \nonumber \\
        &\quad= -\int_{(0,t)\times\calO} (\nu_{f_\ve} + \ve) \nabla_v \Big( f_\ve - \delta \log (\delta + f_\ve) + \delta \log \delta \Big) \cdot (\renorm{v}{1}{\ve} - u_{f_\ve}^{(\ve)}) \;\ds\dx\dv \nonumber\\
        &\quad= \int_{(0,t)\times \calO} (\nu_{f_\ve} + \ve) \Big( f_\ve - \delta\log(\delta + f_\ve) + \delta \log \delta \Big) \,\lt(\nabla_v\cdot \renorm{v}{1}{\ve}\rt) \;\ds\dx\dv. \label{eq: ent: last}
    \end{align}
    Using the relations
    \begin{align*}
        &0 \le z - \delta \log (\delta + z) + \delta \log \delta = z + \delta \log \lt(1 - \frac{z}{\delta+z}\rt) \le z - \frac{\delta z}{\delta + z} \le z, \\
        &\nabla_v\cdot \renorm{v}{1}{\ve} = O(1),
    \end{align*}
    we deduce from the dominated convergence theorem that \eqref{eq: ent: last} converges as $\delta\downarrow 0$ to
    \begin{align*}
        \int_{(0,t)\times \calO} (\nu_{f_\ve} + \ve) f_\ve \, \lt(\nabla_v\cdot \renorm{v}{1}{\ve}\rt) \;\ds\dx\dv.    
    \end{align*}
    Collecting all, we find that \eqref{eq: ent prod} reads in the limit
    \begin{equation}\label{eq: entropy identity}
    \begin{split}
        &\int_{\calO} f_\ve(t) \log f_\ve(t) \,\dx\dv - \int_{\calO} f_\ve^0 \log f_\ve^0 \,\dx\dv + \int_{\Sigma_+^t} \gamma f_\ve \log \gamma f_\ve (n(x)\cdot v)_+ \,\ds\tnd\sigma\dv\\
        &\quad + 4\int_{(0,t)\times\calO} (\nu_{f_\ve} + \ve) \calT_{f_\ve}^{(\ve)} |\nabla_v \sqrt{f_\ve}|^2 \,\dx\dv \\
        &= -\int_{\Sigma_-^t} g_\ve \log g_\ve (n(x)\cdot v)_- \ds\tnd\sigma\dv + \int_{(0,t)\times \calO} (\nu_{f_\ve} + \ve) f_\ve \, \lt(\nabla_v\cdot \renorm{v}{1}{\ve}\rt)\,\dx\dv,
    \end{split}
    \end{equation}
    which is precisely the desired entropy identity.
    
    To obtain the final assertion, we simply note that \eqref{eq: entropy identity} and the assumptions on the initial and boundary datum give us
    \begin{align*}
        &\sup_{t\in [0,T]} \int_{\calO} f_\ve(t) \log f_\ve(t) \,\dx\dv + \int_{\Sigma_+^T} \gamma f_\ve \log \gamma f_\ve (n(x)\cdot v)_+ \,\ds\tnd\sigma\dv  + \int_{(0,T)\times\calO} 4(\nu_{f_\ve} + \ve) |\nabla_v \sqrt{f_\ve}|^2 \,\ds\dx\dv \\
        &\quad\le C_T + C(d,\|\nu\|_{L^\infty}) \int_{(0,T)\times \Omega} f_\ve \;\ds\dx\dv  \le C(d, T, \|\nu\|_{L^\infty}, f^0),
    \end{align*}
    where the last inequality is due to Lemma \ref{lem: energy}. 
\end{proof}

Lemmas \ref{lem: energy} and \ref{lem: entropy} imply by virtue of the Dunford--Pettis criteria that modulo a subsequence there exists a weak limit $(f,\gamma f)$ such that
\begin{align*}
    &f_\ve \weakto f \quad \text{weakly in }L^1((0,T)\times \calO), \quad \gamma f_\ve \weakto \gamma f \quad \text{weakly in }L^1(\Sigma_+^T, (n(x)\cdot v)_+ \dt\tnd\sigma\dv).
\end{align*}
From hereon we will always assume that we are along this weakly convergent subsequence.
%
%
%
%
%
%
\subsection{Renormalization and compactness}

We begin this section by showing that we can pass to the limit in the macroscopic variables. Let us first record the specific averaging lemma we will use, see \cite[Theorem 6]{DipernaLionsMeyer}.

\begin{lemma}\label{lem: L2 vel avg}
    Assume that $h_m \in L^2((0,T)\times\calO)$ is a sequence of solutions to
    \begin{align*}
        \p_t h_m + v\cdot \nabla_x h_m = H_m^1 + \nabla_v\cdot H_m^2 \quad \text{in } \calD'((0,T)\times \calO),
    \end{align*}
    where the $H_m^i$ are bounded in $L^1((0,T)\times\calO)$ for $i=1,2$. Then for each $\varphi\in C_c^\infty(\R^d)$ the sequence of averages $\lt\{\intr h_m \varphi(v)\,\dv\rt\}_m$ is compact in $L^1((0,T)\times\Omega)$.
\end{lemma}

However, this lemma is not applicable in our situation since the $L^2((0,T)\times \calO)$-norm of $f_\ve$ is not uniformly bounded (it depends on the regularized initial datum $f_\ve^0$). Toward this end, inspired by the works \cite{dipernalionsMaxwell1989, GolseSaintRaymond05, DipernaLions89Annals} (see also \cite{Masmoudi10}), we resort to the advantage that $f_\ve$ is renormalized.

\begin{lemma} \label{lem: subcubic}
    Let $(f_\ve,\gamma f_\ve)$ denote the solution to \eqref{eq: reg} as constructed through Lemma \ref{lem: ve weak sol}. Then for any $\varphi\in C^0(\R^d)$ with subcubic growth, the averages $\lt\{\intr f_\ve \varphi(v)\,\dv\rt\}_\ve$ are compact in $L^1((0,T)\times\Omega)$.
\end{lemma}
\begin{proof}
Consider the renormalizing function
\begin{align*}
    \Gamma_\delta(f_\ve) = \frac{f_\ve}{1 + \delta f_\ve},
\end{align*}
which is a bounded weak solution to
\begin{equation}\label{eq: renorm vel avg}
\begin{split}
    \p_t \Gamma_\delta(f_\ve) + v\cdot \nabla_x \Gamma_\delta(f_\ve) &= \Gamma_\delta'(f_\ve) (\nu_{f_\ve} + \ve) \nabla_v \cdot \lt( \calT_{f_\ve}^{(\ve)} \nabla_v f_\ve \rt)   + \Gamma_\delta'(f_\ve) (\nu_{f_\ve} + \ve) \lt( \nabla_v\cdot \renorm{v}{1}{\ve} \rt) f_\ve \\
    &\quad + (\nu_{f_\ve} + \ve) (\renorm{v}{1}{\ve} - u_{f_\ve}^{(\ve)}) \cdot \nabla_v \Gamma_\delta(f_\ve) .
\end{split}
\end{equation}
We observe that the following identities
\begin{align*}
    &\Gamma_\delta'(f_\ve) \nabla_v\cdot \lt( \calT_{f_\ve}^{(\ve)} \nabla_v f_\ve \rt)  = \nabla_v \cdot \lt( \calT_{f_\ve}^{(\ve)} \Gamma_\delta'(f_\ve) 2\sqrt{f_\ve} \nabla_v \sqrt{f_\ve} \rt) - \Gamma_\delta''(f_\ve) \calT_{f_\ve}^{(\ve)} |\nabla_v \sqrt{f_\ve}|^2 4f_\ve ,\\
    &(\renorm{v}{1}{\ve}-u_{f_\ve}^{(\ve)}) \cdot \nabla_v \Gamma_\delta(f_\ve) = \nabla_v\cdot \lt( (\renorm{v}{1}{\ve}-u_{f_\ve}^{(\ve)}) \Gamma_\delta(f_\ve) \rt) - \Gamma_\delta(f_\ve) \nabla_v\cdot \renorm{v}{1}{\ve}, 
\end{align*}
allow us to rewrite the right-hand side of \eqref{eq: renorm vel avg} as
\begin{align*}
    \eqref{eq: renorm vel avg} &= (\nu_{f_\ve} + \ve) \nabla_v\cdot \lt( \calT_{f_\ve}^{(\ve)} \Gamma_\delta'(f_\ve) 2\sqrt{f_\ve} \nabla_v \sqrt{f_\ve} + \lt(\renorm{v}{1}{\ve} - u_{f_\ve}^{(\ve)}\rt) \Gamma_\delta(f_\ve) \rt) \\
    &\quad + (\nu_{f_\ve} + \ve) \lt( -\Gamma_\delta''(f_\ve) \calT_{f_\ve}^{(\ve)} |\nabla_v \sqrt{f_\ve}|^2 4f_\ve + \lt(\Gamma_\delta'(f_\ve) f_\ve - \Gamma_\delta(f_\ve) \rt) \lt(\nabla_v\cdot \renorm{v}{1}{\ve}\rt) \rt).
\end{align*}
That is, we can write
\begin{equation} \label{eq: h1 h2}
    \begin{split}
    &\p_t \Gamma_\delta(f_\ve) + v\cdot \nabla_x \Gamma_\delta(f_\ve) = H_\ve^1 + \nabla_v\cdot H_\ve^2, \\
    &H_{\ve}^1 := (\nu_{f_\ve} + \ve) \lt( -\Gamma_\delta''(f_\ve) \calT_{f_\ve}^{(\ve)} |\nabla_v \sqrt{f_\ve}|^2 4f_\ve + \lt(\Gamma_\delta'(f_\ve) f_\ve - \Gamma_\delta(f_\ve) \rt) \lt(\nabla_v\cdot \renorm{v}{1}{\ve}\rt) \rt), \\
    &H_\ve^2 := (\nu_{f_\ve} + \ve) \lt( 2\calT_{f_\ve}^{(\ve)} \Gamma_\delta'(f_\ve) \sqrt{f_\ve} \nabla_v \sqrt{f_\ve} + \lt(\renorm{v}{1}{\ve} - u_{f_\ve}^{(\ve)}\rt) \Gamma_\delta(f_\ve) \rt).
    \end{split}
\end{equation}
Then, the relations
\begin{align*}
    \Gamma_\delta(f_\ve)\le \min\lt\{\frac{1}{\delta}, f_\ve \rt\}, \quad \Gamma_\delta'(f_\ve) = \frac{1}{(1+\delta f_\ve)^2}, \quad \Gamma_\delta''(f_\ve) = -\frac{2\delta}{(1+\delta f_\ve)^3},
\end{align*}
along with the bounds
\begin{align*}
    &\int_{(0,T)\times \calO} (\nu_{f_\ve} + \ve) \calT_{f_\ve}^{(\ve)} |\nabla_v \sqrt{f_\ve}|^2 \; \dt\dx\dv \le C_T \ (\because \textnormal{Lemma \ref{lem: entropy}}), \quad \nabla_v\cdot \renorm{v}{1}{\ve} = O(1), \quad |u_{f_\ve}^{(\ve)}| \le |u_{f_\ve}|,
\end{align*}
allow us to observe that
\begin{align*}
    \textnormal{$\{H_\ve^1\}, \{H_\ve^2\}$ are bounded in $L^1((0,T)\times \calO)$}.
\end{align*}
Since $\Gamma_\delta(f_\ve) \in L^1\cap L^\infty((0,T)\times \calO)$, we deduce from Lemma \ref{lem: L2 vel avg} that for each $\delta>0$ and $\varphi\in C_c(\R^d)$, the sequence of averages $\lt\{ \intr \Gamma_\delta(f_\ve) \varphi(v)\,\dv\rt\}_\ve$ is compact in $L^1((0,T)\times \Omega)$.

To see that the compactness propagates to $f_\ve$, we proceed as in \cite{GolseSaintRaymond05}. Recalling $\{f_\ve\}$ is weakly convergent in $L^1((0,T)\times \calO)$, we write, for any $M>0$
\begin{align*}
    |\Gamma_\delta(f_\ve) - f_\ve| = \frac{\delta (f_\ve)^2}{1+\delta f_\ve}  = \frac{\delta (f_\ve)^2}{1+\delta f_\ve}\mathbf{1}_{f_\ve \le M} + \frac{\delta (f_\ve)^2}{1+\delta f_\ve} \mathbf{1}_{f_\ve > M}  \le \delta M f_\ve + f_\ve \mathbf{1}_{f_\ve > M}.
\end{align*}
The weak convergence of $f_\ve$ implies that $f_\ve \mathbf{1}_{f_\ve > M} = o(1)_{L^1((0,T)\times \calO)}$, uniformly in $\ve$, as $M\uparrow \infty$. We fix $M$ large enough, next take $\delta\downarrow 0$ to deduce that
\begin{align*}
    \|\Gamma_\delta (f_\ve) - f_\ve \|_{L^1((0,T)\times \calO)} \xrightarrow[\delta\downarrow 0]{} 0.
\end{align*}
This then implies that for each $\varphi\in C_c(\R^d)$, the averages $\lt\{\intr f_\ve \varphi(v)\,\dv\rt\}_\ve$ are also compact in $L^1((0,T)\times\Omega)$. Finally, the compactness can be extended to those $\varphi\in C(\R^d)$ with subcubic growth by using the tightness that is provided by the third moment bounds in Lemma \ref{lem: ve higher moments}.
\end{proof}

We can now show the strong convergence of the distributions $\{f_\ve\}$ themselves.

\begin{lemma}\label{lem: ptw}
    Let $f$ denote the weak $L^1((0,T)\times \calO)$ limit of $\{f_\ve\}$. Then up to some subsequence,
\[
        f_\ve \to f \quad \text{a.e. and strongly in }L^1((0,T)\times \calO).
\]
\end{lemma}

\begin{proof}
    We check that the conditions of Proposition \ref{prop: weighted fisher} are valid for $f_\ve$ and $a_\ve := (\nu_{f_\ve} + \ve) \calT_{f_\ve}^{(\ve)}$. We already know that $f_\ve\weakto f$ weakly, and that the averages $\lt\{\intr f_\ve \varphi(v)\dv\rt\}$ are compact in $L^1((0,T)\times\Omega)$. This gives (P1) and (P2). On the other hand, we recall from Lemma \ref{lem: entropy} that
\[
        \int_{(0,T)\times\calO} (\nu_{f_\ve} + \ve) \calT_{f_\ve}^{(\ve)} |\nabla_v \sqrt{f_\ve}|^2 \;\dt\dx\dv \le C_T,
\]
    and thus (P3) is true. Note that Lemma \ref{lem: subcubic} provides, up to some subsequence, that the following convergences hold almost everywhere and in $L^1((0,T)\times \Omega)$:
    \begin{equation} \label{eq: MAC convergence}
        \rho_{f_\ve} \to \rho_f, \quad j_{f_\ve} \to j_f, \quad  \intr |v|^2 f_\ve \,\dv \to \intr |v|^2 f \,\dv.
    \end{equation}
    The last convergence of the second moments, and the relation in \eqref{eq: variance and second moment}, allows us to deduce again that (recall the arguments after \eqref{eq: poly vel avg})
    \begin{align}\label{eq: V convergence}
        V_{f_\ve} \to V_f \quad \text{a.e. and in}\quad L^1((0,T)\times\Omega).
    \end{align}
    The convergences \eqref{eq: MAC convergence}, \eqref{eq: V convergence} then imply also that
    \begin{equation} \label{eq: U T convergence}
    \begin{split}
        u_{f_\ve} \to u_f \quad \text{a.e. in } D_+ := \{(t,x)\in (0,T)\times \Omega : \rho_f(t,x) > 0\}, \\
        \calT_{f_\ve} \to \calT_f \quad \text{a.e. in } D_+ := \{(t,x)\in (0,T)\times \Omega : \rho_f(t,x) > 0\}.
    \end{split}
    \end{equation}
    Moreover,
    \begin{equation} \label{eq: Tve convergence}
        \calT_{f_{\ve}}^{(\ve)} := \renorm{V_{f_\ve}}{\rho_{f_\ve}}{\ve} + \ve = \frac{V_{f_\ve}}{\rho_{f_\ve} + \ve} \to \frac{V_f}{\rho_f} = \calT_f \quad \text{a.e. in }D_+.
    \end{equation}
    By the assumption that $\nu_{f_\ve}$ is continuous, the above convergences also imply
    \begin{align}\label{eq: NU convergence}
        \nu_{f_\ve} + \ve = \nu(\rho_{f_\ve}, j_{f_\ve}, V_{f_\ve}) \to \nu(\rho_f, j_f, V_f) = \nu_f  \quad \text{a.e. in }(0,T)\times \Omega.
    \end{align}
    Finally, by definition of the temperature $\calT_f$ (and the assumption \eqref{ASSUMP NU} on $\nu$), we have that
    \begin{align*}
        \nu_f \calT_f = 0 \quad \text{if and only if}\quad \rho_f = 0,
    \end{align*}
    therefore this and \eqref{eq: Tve convergence}, \eqref{eq: NU convergence} show (P4) holds with $a = \nu_f \calT_f$. Applying Proposition \ref{prop: weighted fisher} we deduce that up to some subsequence $f_\ve$ converges almost everywhere and strongly in $L^1((0,T)\times \calO)$.
\end{proof}



\subsection{Passing to the limit: the very weak formulation} \label{sec: very weak}
Using the convergences implied by Lemma \ref{lem: ptw}, \eqref{eq: MAC convergence}, \eqref{eq: U T convergence}, \eqref{eq: Tve convergence}, \eqref{eq: NU convergence}, and the weak convergence $\gamma f_\ve \weakto \gamma f$ in $L^1(\Sigma_+^T, |n(x)\cdot v|\dt\tnd\sigma\dv)$, we aim to pass to the limit $\ve\downarrow 0$ in the very weak formulation of \eqref{eq: reg}. We recall that $f_\ve$ satisfies, for any $t\in (0,T]$ and $\varphi\in C_c^\infty([0,t]\times \overline{\Omega} \times \R^d)$
\begin{equation} \label{eq: ve weak form}
\begin{split}
    &\int_{\calO} f_\ve(t) \varphi(t,x,v) \,\dx\dv - \int_{\calO} f_\ve^0 \varphi(0,x,v)\,\dx\dv + \int_{\Sigma_+^t} \gamma f_\ve \, \varphi \, (n(x)\cdot v)_+ \,\ds\tnd\sigma\dv \\
    &\quad= - \int_{\Sigma_-^t} g_\ve \, \varphi \, (n(x)\cdot v)_- \,\ds\tnd\sigma\dv +\int_{(0,t)\times \calO} f_\ve (\p_s \varphi + v\cdot \nabla_x \varphi) \,\ds\dx\dv\\
    &\qquad + \int_{(0,t)\times \calO} (\nu_{f_\ve} + \ve) \lt(\Delta_v \varphi \calT_{f_\ve}^{(\ve)} f_\ve - \nabla_v\varphi \cdot \lt(\renorm{v}{1}{\ve}-u_{f_\ve}^{(\ve)} \rt) f_\ve \rt) \, \ds\dx\dv.
\end{split}\end{equation}
%
%
%
%
%
%

\medskip
\noindent \textbf{Step 1: Convergence of the left-hand side.}
Since Lemma \ref{lem: ptw} implies
\[
    \lt\| \|f_\ve(t,\cdot,\cdot) - f(t,\cdot,\cdot)\|_{L^1(\calO)} \rt\|_{L^1(0,T)} \xrightarrow[\ve\downarrow 0]{} 0, 
\]
there exists a subsequence such that for almost every $t\in (0,T]$:
\[
    \|f_\ve(t,\cdot,\cdot) - f(t,\cdot,\cdot)\|_{L^1(\calO)} \xrightarrow[\ve\downarrow 0]{} 0.
\]
Hence, we have
\[
    \int_{\calO} |f_\ve(t) - f(t)| |\varphi(t,x,v)| \,\dx\dv \le \|\varphi\|_{L^\infty} \int_{\calO} |f_\ve(t) - f(t)| \,\dx\dv \to 0 \quad \text{a.e. }t\in (0,T],
\]
from which we deduce that the first term of the left-hand side converges as
\[
    \int_{\calO} f_\ve(t) \varphi(t,x,v) \,\dx\dv \to \int_{\calO} f(t) \varphi(t,x,v) \,\dx\dv,
\]
at least a.e. $t$.

The term containing the initial data converges as $\int_{\calO} f_\ve^0 \varphi(0,x,v) \to \int_{\calO} f^0 \varphi(0,x,v)$ since we assume $f_\ve^0 \to f^0$ at least weakly in $L^1(\calO)$.

For the term containing the trace, the weak convergence $\gamma f_\ve \weakto \gamma f$ in $L^1(\Sigma_+^T, |n(x)\cdot v|\dt\tnd\sigma\dv)$ implies
\[
    \int_{\Sigma_+^t} \gamma f_\ve \, \varphi \, (n(x)\cdot v)_+ \,\ds\tnd\sigma\dv \to \int_{\Sigma_+^t} \gamma f \, \varphi \, (n(x)\cdot v)_+ \,\ds\tnd\sigma\dv \quad \forall t\in (0,T].
\]

We conclude that the left-hand side of \eqref{eq: ve weak form} is convergent for a.e. $t$ in the limit $\ve\downarrow 0$.
%
%
%
%
%
%

\medskip
\noindent \textbf{Step 2: Convergence of the right-hand side.}  Since $g_\ve \to g$ in $L^1(\Sigma_-^T, |n(x)\cdot v| \dt\tnd\sigma\dv)$,
\begin{align*}
    \int_{\Sigma_-^t} g_\ve \, \varphi \, (n(x)\cdot v)_- \,\ds\tnd\sigma\dv \to \int_{\Sigma_-^t} g \varphi \, (n(x)\cdot v)_- \,\ds\tnd\sigma\dv.
\end{align*}
The convergence of the terms that are linear in $f_\ve$ is clear. The only nontrivial remaining convergence is that of the last line of \eqref{eq: ve weak form}. Let us begin by demonstrating

\begin{lemma}
For each $\varphi\in C_c^\infty([0,t]\times \overline{\Omega}\times \R^d)$ we have
\begin{align}\label{eq: PASS weak form 1}
    \int_{(0,t)\times \calO} (\nu_{f_\ve} + \ve) \Delta_v \varphi \calT_{f_\ve}^{(\ve)} f_\ve \,\ds\dx\dv  \to \int_{(0,t)\times \calO} \nu_{f} \Delta_v \varphi \calT_f f \,\ds\dx\dv.
\end{align}
\end{lemma}
\begin{proof}
Resorting to a density argument, we may assume without loss of generality that the test function $\varphi\in C_c^\infty([0,t]\times\overline\Omega\times\R^d)$ takes the form
\[
    \varphi(s,x,v) = \psi(s,x) \eta(v), \quad \psi\in C^\infty([0,t]\times \overline{\Omega}), \quad \eta\in C_c^\infty(\R^d).
\]
Then, with the notations
\[
 \rho_{\Delta\eta}^\ve(s,x) := \intr f_\ve(s,x,v) \Delta\eta(v) \,\dv, \quad \rho_{\Delta\eta}(s,x) := \intr f(s,x,v) \Delta\eta(v)\,\dv.
\]
Note that \eqref{eq: PASS weak form 1} is equivalent to showing that
\begin{align} \label{eq: PASS weak form 1, EQUIV}
    \int_{(0,t)\times \Omega} \psi(s,x) \lt[ (\nu_{f_\ve} + \ve) \calT_{f_\ve}^{(\ve)} \rho_{\Delta\eta}^\ve - \nu_f \calT_f \rho_{\Delta\eta} \rt] \,\ds\dx\dv \to 0.
\end{align}
Hence we only need show that $(\nu_{f_\ve} + \ve) \calT_{f_\ve}^{(\ve)} \rho_{\Delta\eta}^\ve \weakto \nu_f \calT_f \rho_{\Delta\eta}$ weakly in $L^1([0,T]\times \Omega)$. We will prove a stronger statement and show that
\begin{align}\label{eq: stronger}
     \lt[ (\nu_{f_\ve} + \ve) \calT_{f_\ve}^{(\ve)} \rho_{\Delta\eta}^\ve - \nu_f \calT_f \rho_{\Delta\eta} \rt] \to 0 \quad \text{in} \quad L^1((0,T)\times\Omega).
\end{align}
By the averaging effect of Lemma \ref{lem: subcubic}, up to some subsequence it holds that
\begin{align} \label{eq: RHO ETA}
    \rho_{\Delta\eta}^\ve \xrightarrow[\ve\downarrow 0]{} \rho_{\Delta\eta} := \intr f\,\Delta\eta(v)\,\dv \quad \text{a.e. and in} \quad L^1((0,T)\times\Omega).
\end{align}
Taking the product with the (regularized) temperature, we can then observe the following bound, due to Young's inequality
\begin{align}\label{eq: eta YOUNG}
    \calT_{f_\ve}^{(\ve)} |\rho_{\Delta\eta}^\ve| \le C \|\Delta\eta\|_{L^\infty} \lt(1 + \calT_{f_\ve}^{(\ve)}\rt)  \rho_{f_\ve} = C(\|\Delta\eta\|_{L^\infty}) (\rho_{f_\ve} + V_{f_\ve}),
\end{align}
the right-hand side of which is strongly convergent in $L^1((0,T)\times\Omega)$ by \eqref{eq: MAC convergence}, \eqref{eq: V convergence}. As a result, the family $\lt\{\calT_{f_\ve}^{(\ve)} \rho_{\Delta\eta}^\ve \rt\}_\ve$ is equiintegrable in $(0,T)\times\Omega$. Since \eqref{eq: RHO ETA} implies
\begin{align*}
    \calT_{f_\ve}^{(\ve)} \rho_{\Delta\eta}^\ve \xrightarrow[\ve\downarrow 0]{} \calT_f \rho_{\Delta\eta} \quad \text{a.e. in}\quad D_+ := \{(t,x)\in (0,T)\times\Omega: \rho_f(t,x)>0\},
\end{align*}
we deduce from the Vitali convergence theorem that
\begin{align}\label{eq: T ETA D+}
    \calT_{f_\ve}^{(\ve)} \rho_{\Delta\eta}^\ve \to \calT_{f} \rho_{\Delta\eta} \quad \text{in} \quad L^1(D_+).
\end{align}
On the other hand, on the set $D_0 := [(0,T)\times \Omega]\setminus D_+$, we have $f(t,x,v) = 0$ a.e. $v\in \R^d$, which gives that $\rho_{\Delta\eta} = \calT_f = 0$ in $D_0$. Therefore, 
\begin{equation} \label{eq: T ETA D0}
\begin{split}
    \lim_{\ve\downarrow 0} \int_{D_0} \lt| \calT_{f_\ve}^{(\ve)} \rho_{\Delta\eta}^\ve - \calT_f \rho_{\Delta\eta} \rt| \,\dt\dx &= \lim_{\ve\downarrow 0}\int_{D_0} \calT_{f_\ve}^{(\ve)} \lt|\rho_{\Delta\eta}^\ve \rt| \,\dt\dx \\
    &\le C(\|\Delta\eta\|_{L^\infty}) \lim_{\ve\downarrow 0} \int_{D_0} (\rho_{f_\ve} + V_{f_\ve}) \,\dt\dx\\
    &= C \int_{D_0} (\rho_f + V_f) \,\dt\dx\\
    &= 0,
\end{split}
\end{equation}
the inequality due to \eqref{eq: eta YOUNG}. In other words this proves $\calT_{f_\ve}^{(\ve)} \rho_{\Delta\eta}^\ve \to \calT_f \rho_{\Delta\eta}$ in $L^1(D_0)$. Combining this with \eqref{eq: T ETA D+}, we find
\begin{align}\label{eq: PASS prod T eta}
    \calT_{f_\ve}^{(\ve)} \rho_{\Delta\eta}^\ve \to \calT_f \rho_{\Delta\eta} \quad \text{in} \quad L^1((0,T)\times\Omega).
\end{align}
To complete the proof, we recall that $(\nu_{f_\ve} + \ve) \to \nu_f$ pointwise a.e. in $(0,T)\times \Omega$ (see \eqref{eq: NU convergence}), and $\|\nu_{f_\ve} + \ve\|_{L^\infty} \le \|\nu\|_{L^\infty} + 1$. By Egorov's theorem we can extract a set $E\subset (0,T)\times\Omega$ with arbitrarily small Lebesgue measure for which $\nu_{f_\ve} \to \nu_f$ uniformly outside $E$. Then, writing the difference in \eqref{eq: PASS weak form 1, EQUIV} as
\begin{equation} \label{eq: j1 j2}
\begin{split}
     \lt|(\nu_{f_\ve} + \ve) \calT_{f_\ve}^{(\ve)} \rho_{\Delta\eta}^\ve - \nu_f \calT_f \rho_{\Delta\eta} \rt|  
    &\le \lt| (\nu_{f_\ve} + \ve - \nu_f)  \calT_{f_\ve}^{(\ve)} \rho_{\Delta\eta}^\ve \rt| + \lt|\nu_f \lt( \calT_{f_\ve}^{(\ve)} \rho_{\Delta\eta}^\ve - \calT_f \rho_{\Delta\eta} \rt) \rt|\\
    &=: J_1 + J_2,
\end{split}
\end{equation}
we see that $J_1\to 0$ in $L^1([(0,T)\times\Omega]\setminus E)$ and $J_2\to 0$ in $L^1((0,T)\times\Omega)$, thanks to \eqref{eq: PASS prod T eta}. Taking the Lebesgue measure of $E$ to zero, we then see that $\|J_1\|_{L^1(E)}\to 0$ since $\nu$ is bounded and the family $\lt\{\calT_{f_\ve}^{(\ve)} \rho_{\Delta\eta}^\ve\rt\}$ is equiintegrable. It follows that
\begin{align*}
    (\nu_{f_\ve} + \ve) \calT_{f_\ve}^{(\ve)} \rho_{\Delta\eta}^\ve \to \nu_f \calT_f \rho_{\Delta\eta} \quad \text{in}\quad L^1((0,T)\times \Omega),
\end{align*}
and this proves \eqref{eq: stronger}. Clearly, the latter then implies \eqref{eq: PASS weak form 1, EQUIV}.
\end{proof}

Finally, we check that:
\begin{lemma}
In the limit as $\ve\downarrow 0$, we have the following convergences regarding the right-hand side of \eqref{eq: ve weak form}:
\begin{align}
    &\int_{(0,t)\times \calO} (\nu_{f_\ve} + \ve) \nabla_v \varphi \cdot \renorm{v}{1}{\ve}  f_\ve \;\ds\dx\dv \to \int_{(0,t)\times \calO} \nu_f \nabla_v\varphi \cdot v f \;\ds\dx\dv,  \label{eq: PROD renorm v} \\
    &\int_{(0,t)\times \calO} (\nu_{f_\ve} + \ve) \nabla_v\varphi \cdot u_{f_\ve}^{(\ve)} f_\ve \;\ds\dx\dv \to \int_{(0,t)\times \calO} \nu_f \nabla_v\varphi \cdot u_f f \;\ds\dx\dv. \label{eq: PROD u}
\end{align}
\end{lemma}
\begin{proof}
Since $\|\nu_{f_\ve} + \ve \|_{L^\infty} \le \|\nu\|_{L^\infty} + 1$, we have that
\begin{align*}
    (\nu_{f_\ve} + \ve) f_\ve \le C(\|\nu\|_{L^\infty}) f_\ve,
\end{align*}
which implies that the sequence $\{(\nu_{f_\ve} + \ve) f_\ve\}$ is weakly compact in $L^1((0,T)\times \calO)$; by \eqref{eq: NU convergence} and Lemma \ref{lem: ptw}, the sequence is almost everywhere convergent. By the Vitali convergence theorem
\begin{align*}
    (\nu_{f_\ve} + \ve) f_\ve \to \nu_f f \quad \text{in}\quad L^1((0,T)\times \calO).
\end{align*}
It follows that the difference of \eqref{eq: PROD renorm v} can be estimated as
\begin{align*}
    &\lt|\int_{(0,t)\times \calO} (\nu_{f_\ve} + \ve) \nabla_v \varphi \cdot \renorm{v}{1}{\ve}  f_\ve \;\ds\dx\dv - \int_{(0,t)\times \calO} \nu_f \nabla_v\varphi \cdot v f \;\ds\dx\dv \rt| \\
    &\quad\le \int_{(0,t)\times \calO} \lt| (\nu_{f_\ve} + \ve) f_\ve  - \nu_f f \rt|  |\nabla_v \varphi| \lt|\renorm{v}{1}{\ve} \rt| \;\ds\dx\dv  + \int_{(0,t)\times \calO} |\nabla_v \varphi| \lt|\renorm{v}{1}{\ve} - v \rt| \nu_f f \;\ds\dx\dv \\
    &\quad\le \textnormal{diam}(\textnormal{supp}\varphi) \|\nabla_v\varphi\|_{L^\infty} \int_{(0,t)\times \calO} |(\nu_{f_\ve} + \ve) f_\ve - \nu_f f| \;\ds\dx\dv \\
    &\qquad + \ve ( 1 + \textnormal{diam}(\textnormal{supp}(\varphi))) \|\nu\|_{L^\infty} \int_{(0,t)\times \calO} f \,\ds\dx\dv \quad \lt(\because \lt|\renorm{v}{1}{\ve} - v\rt| = \frac{\ve ( 1 + |v|)}{1 + \ve (1 + |v|)} \rt) \\
    &\quad\xrightarrow[\ve\downarrow 0]{} 0.
\end{align*}
Toward the end of \eqref{eq: PROD u} now, we proceed similarly as in the proof of \eqref{eq: PASS weak form 1, EQUIV}. Taking $\varphi$ to be of the form $\varphi(s,x,v) = \psi(s,x) \eta(v)$, it is enough to show that
\begin{align*}
    (\nu_{f_\ve} + \ve) u_{f_\ve}^{(\ve)} \rho_{\Delta\eta}^\ve \to \nu_f u_f \rho_{\Delta\eta} \quad \text{weakly in}\quad L^1((0,T)\times\Omega).
\end{align*}
Again we will prove the stronger statement that the convergence holds strongly. As aforementioned the ideas are almost the same and thus we provide only a sketch of proof. Using the relation
\begin{align*}
    |u_{f_\ve}^{(\ve)}\rho_{\Delta\eta}^\ve | \le C(\|\Delta_\eta\|_{L^\infty}) \rho_{f_\ve} |u_{f_\ve}|,
\end{align*}
we deduce similarly as in the arguments in \eqref{eq: eta YOUNG} and \eqref{eq: T ETA D0} that $u_{f_\ve}^{(\ve)} \rho_{\Delta\eta}^\ve$ is strongly convergent in $L^1((0,T)\times \Omega)$. In particular, $\lt\{u_{f^{\ve}}^{(\ve)} \rho_{\Delta\eta}^\ve\rt\}$ is equiintegrable over $(0,T)\times \Omega$. Then borrowing Egorov's theorem again, we have that $(\nu_{f_\ve}+\ve) \to \nu_f$ uniformly outside of a small Lebesgue-measure set. Combining the two given information we deduce through telescoping as in \eqref{eq: j1 j2} that
\begin{align*}
    (\nu_{f_\ve} + \ve) u_{f_\ve}^{(\ve)} \rho_{\Delta\eta}^\ve \to \nu_f u_f \rho_{\Delta\eta} \quad \text{strongly in}\quad L^1((0,T)\times\Omega).
\end{align*}
This completes the proof of \eqref{eq: PROD u}.
\end{proof}

Collecting all, we find that for a.e. $t\in (0,T]$ it holds that
\begin{align*}
    &\int_{\calO} f(t) \varphi(t,x,v) \,\dx\dv - \int_{\calO} f^0 \varphi(0,x,v)\,\dx\dv + \int_{\Sigma_+^t} \gamma f \, \varphi \, (n(x)\cdot v)_+ \,\ds\tnd\sigma\dv\\
    &\quad= - \int_{\Sigma_-^t} g \, \varphi \, (n(x)\cdot v)_- \,\ds\tnd\sigma\dv  + \int_{(0,t)\times \calO} \nu_f \lt( \Delta_v \varphi \, \calT_f f - \nabla_v\varphi \cdot (v - u_f) f \rt) \; \ds\dx\dv.
\end{align*}
That is, up to a modification of $f$ in a $\scrL^1([0,T])$-zero set, the map
\begin{align*}
    (0,T] \ni t \mapsto \int_{\calO} f(t) \varphi(t,x,v) \,\dx\dv
\end{align*}
is absolutely continuous, and up to this version, the very weak formulation holds for all $t\in (0,T]$.
%
%
%
%
%
%
\subsection{Passing to the limit: the weak formulation} \label{sec: weak}
If we are to use the weak formulation (instead of the very weak formulation), the only difference is that we must show 
\begin{align} \label{eq: PASS: WEAK}
    \int_{(0,t)\times \calO} (\nu_{f_\ve} + \ve) \nabla_v\varphi \cdot \calT_{f_\ve}^{(\ve)} \nabla_v f_\ve \,\ds\dx\dv  \to \int_{(0,t)\times \calO} \nu_f \nabla_v\varphi \cdot \calT_f \nabla_v f \,\ds\dx\dv. 
\end{align}
To give a sense to this identity, in virtue of the entropy dissipation in Lemma \ref{lem: entropy}, we rewrite
\begin{align}\label{eq: WEAK rewrite}
    (\nu_{f_\ve} + \ve) \cdot \calT_{f_\ve}^{(\ve)} \nabla_v f_\ve = 2  \lt[\sqrt{\nu_{f_\ve} + \ve} \lt(\calT_{f_\ve}^{(\ve)}\rt)^{1/2}\sqrt{f_\ve}\rt] \lt[ \sqrt{\nu_{f_\ve} + \ve}\lt(\calT_{f_\ve}^{(\ve)}\rt)^{1/2} \nabla_v \sqrt{f_\ve} \rt].
\end{align}
Let us investigate the convergence of each term in the brackets.
\begin{lemma} \label{lem: prod sqrt STR0}
    The sequence $\lt\{ \sqrt{\nu_{f_\ve} + \ve} \lt(\calT_{f_\ve}^{(\ve)}\rt)^{1/2}\sqrt{f_\ve}\rt\}$ converges strongly in $L^2((0,T)\times \calO)$ to $\sqrt{\nu_f}\calT_f^{1/2} \sqrt{f}$.
\end{lemma}
\begin{proof}
    First, we prove weak convergence. Using the inequality:
    \begin{align*}
        \calT_{f_\ve}^{(\ve)} f_\ve \le C (1 + \calT_{f_\ve}) f_\ve,
    \end{align*}
    and also that $\nu$ is universally bounded, we can straightforwardly observe that $\lt\{ \sqrt{\nu_{f_\ve} + \ve}\lt(\calT_{f_\ve}^{(\ve)}\rt)^{1/2}\sqrt{f_\ve}\rt\}$ is bounded in $L^2((0,T)\times \calO)$. To see that the weak $L^2$ limit is $\sqrt{\nu_f}\calT_f^{1/2} \sqrt{f}$, we first fix a compactly supported test function $\psi \in C_c^\infty((0,T)\times \calO)$. Then, using the convergences provided by Lemma \ref{lem: ptw}, \eqref{eq: U T convergence}, \eqref{eq: NU convergence}, we resort to Egorov's theorem, extracting a set $E\subset \textnormal{supp}(\psi) \cap (D_+ \times \R^d)$ with small Lebesgue measure, for which 
    \begin{align*}
    \textnormal{$\sqrt{\nu_{f_\ve} + \ve}\lt(\calT_{f_\ve}^{(\ve)}\rt)^{1/2}\sqrt{f_\ve} \to \sqrt{\nu_f}\calT_f^{1/2} \sqrt{f}$\quad uniformly in $(\textnormal{supp}(\psi) \cap (D_+ \times \R^d) )\setminus E$.}
    \end{align*}
    Then, using that $\sqrt{\nu_{f_\ve}+\ve}\lt(\calT_{f_\ve}^{(\ve)}\rt)^{1/2}\sqrt{f_\ve}$ is equiintegrable over $(0,T)\times \calO$ (owing to it being bounded in $L^2$), we let the measure of $E$ tend to zero to deduce that
    \begin{align} \label{eq: sqrt dist}
        \int_{D_+ \times \R^d} \sqrt{\nu_{f_\ve} + \ve} \lt(\calT_{f_\ve}^{(\ve)}\rt)^{1/2}\sqrt{f_\ve} \, \psi(t,x,v) \,\dt\dx\dv \to \int_{D_+\times \R^d} \sqrt{\nu_f} \calT_f^{1/2} \sqrt{f} \,\psi(t,x,v) \,\dt\dx\dv.
    \end{align}
    In regard of the complement $D_0 = ((0,T)\times \Omega) \setminus D_+$, we can easily check that
    \begin{equation} \label{eq: sqrt D0}
    \begin{split}
        &\int_{D_0\times \R^d} \lt|\sqrt{\nu_{f_\ve} + \ve} \lt(\calT_{f_\ve}^{(\ve)}\rt)^{1/2} \sqrt{f_\ve} - \sqrt{\nu_f} \calT_f^{1/2} \sqrt{f} \rt|^2 \,\dt\dx\dv  \\
        &\quad= \int_{D_0 \times \R^d} (\nu_{f_\ve} + \ve) \calT_{f_\ve}^{(\ve)} f_\ve \,\dt\dx\dv\\
        &\quad= \int_{D_0} (\nu_{f_\ve} + \ve) \calT_{f_\ve}^{(\ve)} \rho_{f_\ve} \,\dt\dx \\
        &\quad\le (\|\nu\|_{L^\infty}+1) \int_{D_0} (1 + \calT_{f_\ve}) \rho_{f_\ve} \,\dt\dx \\
        &\quad= C \int_{D_0} (\rho_{f_\ve} + V_{f_\ve}) \,\dt\dx \xrightarrow[\ve\downarrow 0]{} 0.
    \end{split}
    \end{equation}
    Combining this with \eqref{eq: sqrt dist}, it is therefore clear that for all $\psi\in C_c^\infty((0,T)\times\calO)$
    \begin{align*}
        \int_{(0,T)\times \calO} \sqrt{\nu_{f_\ve} + \ve} \lt(\calT_{f_\ve}^{(\ve)}\rt)^{1/2}\sqrt{f_\ve} \, \psi(t,x,v) \,\dt\dx\dv \to \int_{(0,T)\times \calO} \sqrt{\nu_f} \calT_f^{1/2} \sqrt{f} \,\psi(t,x,v) \,\dt\dx\dv.
    \end{align*}
    To relax the condition on the test function to $\psi\in L^2((0,T)\times \calO)$, we simply use that $C_c^\infty((0,T)\times \calO)$ is dense in $L^2((0,T)\times \calO)$ to conclude that \eqref{eq: sqrt dist} holds for all $\psi\in L^2$. Hence the weak $L^2$ limit of $\sqrt{\nu_{f_\ve} + \ve}\lt(\calT_{f_\ve}^{(\ve)}\rt)^{1/2}\sqrt{f_\ve}$ is necessarily $\sqrt{\nu_f}\calT_f^{1/2} \sqrt{f}$.

    In virtue of the arguments so far, to show that the $L^2$ convergence is in fact strong, it is enough to prove the norm convergence
    \begin{align}\label{eq: L2 str}
        \lim_{\ve\downarrow 0} \int_{(0,T)\times \calO} (\nu_{f_\ve} + \ve) \calT_{f_\ve}^{(\ve)} f_\ve \,\dt\dx\dv = \int_{(0,T)\times \calO} \nu_f \calT_f f \,\dt\dx\dv.
    \end{align}
    Notice that in \eqref{eq: sqrt D0}, we already showed
    \begin{align*}
        \sqrt{\nu_{f_\ve} + \ve} \lt(\calT_{f_\ve}^{(\ve)}\rt)^{1/2} \sqrt{f_\ve} \to \sqrt{\nu_f} \calT_f^{1/2} \sqrt{f} \quad \text{in} \quad L^2(D_0\times \R^d), 
    \end{align*}
    and therefore, it is enough to prove that the integral of the left-hand side of \eqref{eq: L2 str} converges in the domain $D_+\times\R^d$.
    Towards this goal, we first observe that \eqref{eq: L2 str} is equivalent to showing
    \begin{align}\label{eq: L2 NORM}
        \lim_{\ve\downarrow 0} \int_{D_+} (\nu_{f_\ve} + \ve) \calT_{f_\ve}^{(\ve)} \rho_{f_\ve} \,\dt\dx = \int_{D_+} \nu_f \calT_f \rho_f \,\dt\dx.
    \end{align}
    Using \eqref{eq: MAC convergence}, \eqref{eq: U T convergence}, \eqref{eq: NU convergence}, and Egorov's theorem, we can again extract a set $E\subset D_+$ with small Lebesgue measure for which 
    \begin{align}\label{eq: rho T alpha UNIF}
        (\nu_{f_\ve} + \ve) \calT_{f_\ve}^{(\ve)} \rho_{f_\ve} \to \nu_f \calT_f \rho_f \quad \text{uniformly in $D_+\setminus E$}.
    \end{align}
    Hence, it is enough to prove that $\lt\{(\nu_{f_\ve} + \ve)\calT_{f_\ve}^{(\ve)} \rho_{f_\ve} \rt\}$ cannot concentrate on the small set $E$. By using the definition of the (regularized) temperature, we note  the estimate
    \begin{equation} \label{eq: RHO T alpha}
    \begin{split}
        \int_E (\nu_{f_\ve} + 1) \calT_{f_\ve}^{(\ve)} \rho_{f_\ve} \,\dt\dx &\le (\|\nu\|_{L^\infty} + 1) \int_E \lt(1 + \calT_{f_\ve}\rt) \rho_{f_\ve} \,\dt\dx \\
        &= (\|\nu\|_{L^\infty}+1) \lt(\int_E \rho_{f_\ve} \,\dt\dx +  \int_E V_{f_\ve} \,\dt\dx \rt).
    \end{split}
    \end{equation}
    Since $\{\rho_{f_\ve}\}$ and $\{V_{f_\ve}\}$ are both equiintegrable over $(0,T)\times \Omega$ (they are strongly convergent in $L^1$ by \eqref{eq: MAC convergence} and \eqref{eq: V convergence}), we observe that the right-hand side of \eqref{eq: RHO T alpha} vanishes as $\scrL^{d+1}(E) \downarrow 0$. This combined with \eqref{eq: rho T alpha UNIF} proves \eqref{eq: L2 NORM} and hence the assertion of the lemma.
\end{proof}

Finally, we discuss the Fisher information arising in \eqref{eq: WEAK rewrite}.

\begin{lemma} \label{lem: nabla SQRT WEAK}
    The sequence $\lt\{ \sqrt{\nu_{f_\ve}+\ve}\lt(\calT_{f_\ve}^{(\ve)}\rt)^{1/2} \nabla_v \sqrt{f_\ve}\rt\}$ converges weakly in $L^2((0,T)\times \calO)$ to $\sqrt{\nu_f}\calT_f^{1/2} \nabla_v \sqrt{f}$.
\end{lemma}

\begin{proof}
    From Lemma \ref{lem: entropy}, we observe that $\lt\{\sqrt{\nu_{f_\ve}+\ve} \lt(\calT_{f_\ve}^{(\ve)}\rt)^{1/2} \nabla_v \sqrt{f_\ve}\rt\}$ is bounded in $L^2((0,T)\times \calO)$. It is thus enough to identify the weak limit. Clearly, the arguments surrounding \eqref{eq: sqrt dist} and \eqref{eq: sqrt D0} imply that for all $\psi\in C_c^\infty((0,T)\times\calO)$:
    \begin{align*}
        &\int_{(0,T)\times \calO}  \sqrt{\nu_{f_\ve} + \ve}\lt(\calT_{f_\ve}^{(\ve)}\rt)^{1/2} \sqrt{f_\ve} \, \p_{v_i} \psi \,\dt\dx\dv \to \int_{(0,T)\times \calO} \sqrt{\nu_f} \calT_f^{1/2} \sqrt{f} \, \p_{v_i} \psi \,\dt\dx\dv, \quad i=1,2,\ldots, d.
    \end{align*}
    By definition of the weak derivative, we find
    \begin{align*}
        \textnormal{$\sqrt{\nu_{f_\ve}+\ve}\lt(\calT_{f_\ve}^{(\ve)}\rt)^{1/2} \nabla_v \sqrt{f_\ve} \to \sqrt{\nu_f}\calT_f^{1/2} \nabla_v \sqrt{f}$ \quad in \quad $\calD'((0,T)\times \calO)$}.
    \end{align*}
    Then, through a density argument, we obtain the assertion of the lemma.
\end{proof}

Combining Lemma \ref{lem: prod sqrt STR0} and Lemma \ref{lem: nabla SQRT WEAK}, we conclude that the term in \eqref{eq: WEAK rewrite} converges weakly in $L^1((0,T)\times\calO)$ as
\begin{align*}
    2(\nu_{f_\ve} + \ve)  \lt[ \lt(\calT_{f_\ve}^{(\ve)}\rt)^{1/2}\sqrt{f_\ve}\rt] \lt[ \lt(\calT_{f_\ve}^{(\ve)}\rt)^{1/2} \nabla_v \sqrt{f_\ve} \rt] \weakto 2 \nu_f \lt[ \calT_f^{1/2} \sqrt{f} \rt] \lt[ \calT_f^{1/2} \nabla_v \sqrt{f} \rt].
\end{align*}
This proves \eqref{eq: PASS: WEAK}.

%
%
%
%
%
%

\subsection{The energy inequality} \label{sec: energy}
In this section we revisit Lemma \ref{lem: energy} and prove an energy inequality which holds for the solution $f$.

\begin{proposition} \label{prop: energy}
    The limit $f$ satisfies the following energy inequality:
    \begin{align*}
        &\int_{\calO} (1+|v|^2) f(t) \,\dx\dv - \int_{\calO} (1+|v|^2) f^0 \,\dx\dv + \int_{\Sigma_+^t} (1+|v|^2) \gamma f (n(x)\cdot v)_+ \,\ds\tnd\sigma\dv \\
    &\quad\le -\int_{\Sigma_-^t} (1+|v|^2) g (n(x)\cdot v)_-.
    \end{align*}
\end{proposition}
The proof of the proposition consists in showing that, in the energy inequality satisfied by the regularized solution $f_\ve$, the left-hand side is lower semicontinuous whereas the right-hand side converges in the limit $\ve\downarrow 0$. Indeed, for the convenience of the reader, let us recall \eqref{eq: energy est} here:
\begin{equation} \label{eq: recall EGY}
\begin{split}
    &\int_{\calO} (1+|v|^2) f_\ve(t) \,\dx\dv - \int_{\calO} (1+|v|^2) f_\ve^0 \,\dx\dv + \int_{\Sigma_+^t} (1+|v|^2) \gamma f_\ve (n(x)\cdot v)_+ \,\ds\tnd\sigma\dv \\
    &\quad + \int_{(0,t)\times \calO} (\nu_{f_\ve} + \ve) 2v\cdot \renorm{v}{1}{\ve} f_\ve  \,\ds\dx\dv \\
    &= - \int_{\Sigma_-^t} (1+|v|^2) g_\ve (n(x)\cdot v)_- \,\ds\tnd\sigma\dv+ \int_{(0,t)\times \calO} (\nu_{f_\ve} + \ve) \lt(2 d \calT_{f_\ve}^{(\ve)} f_\ve + 2 v \cdot u_{f_\ve}^{(\ve)} f_\ve \rt) \,\ds\dx\dv.
\end{split}
\end{equation}
%
%
%
%
%
%
\subsubsection{Lower semicontinuity of the left-hand side}
By the lower semicontinuity of the moments under weak convergence:
\begin{align*}
    &\int_{\calO}(1+|v|^2) f(t) \,\dx\dv + \int_{\Sigma_+^t} (1+|v|^2) \gamma f (n(x)\cdot v)_+ \,\ds\tnd\sigma\dv\\
    &\quad \le \liminf_{\ve\downarrow 0} \int_{\calO} (1+|v|^2) f_\ve(t) \,\dx\dv + \int_{\Sigma_+^t} (1+|v|^2) \gamma f_\ve (n(x)\cdot v)_+ \,\ds\tnd\sigma\dv.
\end{align*}
On the other hand, the assumption \eqref{ASSUMP f0 g} on the regularized initial data gives
\begin{align*}
    \lim_{\ve\downarrow 0} \int_{\calO} (1+|v|^2) f_\ve^0 \,\dx\dv = \int_{\calO} (1+|v|^2) f^0 \,\dx\dv.
\end{align*}
Lastly, by Lemma \ref{lem: ptw}, \eqref{eq: NU convergence}, $\renorm{v}{1}{\ve}\to v$, and Fatou's lemma, the final integral of the left-hand side is lower semicontinuous as
\begin{align*}
    \int_{(0,t)\times \calO} 2\nu_f |v|^2 f_\ve \,\ds\dx\dv \le \liminf_{\ve\downarrow 0} \int_{(0,t)\times\calO} (\nu_{f_\ve} + \ve) 2v\cdot \renorm{v}{1}{\ve} f_\ve \;\ds\dx\dv.
\end{align*}
%
%
%
%
%
%
\subsubsection{Convergence of the right-hand side}
We now investigate the right-hand side of \eqref{eq: recall EGY}. First, again by the conditions \eqref{ASSUMP f0 g} for the regularized boundary data:
\begin{align*}
    \lim_{\ve\downarrow 0} \int_{\Sigma_-^t} (1+|v|^2) g_\ve (n(x)\cdot v)_- \,\ds\tnd\sigma\dv = \int_{\Sigma_-^t} (1+|v|^2) g (n(x)\cdot v)_- \,\ds\tnd\sigma\dv.
\end{align*}
Only the convergence of the last integral in \eqref{eq: recall EGY} remains in question. Let us consider the convergence of the following integral first:
\begin{lemma} 
In the limit $\ve\downarrow 0$, we have that
\begin{align}\label{eq: egy mass}
    \int_{(0,t)\times \calO} (\nu_{f_\ve} + \ve) d \calT_{f_\ve}^{(\ve)} f_\ve \,\ds\dx\dv \to \int_{(0,t)\times \calO} \nu_f d\calT_f f \,\ds\dx\dv.
\end{align}
\end{lemma}
\begin{proof}
Notice that proving \eqref{eq: egy mass} is equivalent to showing
\begin{align}\label{eq: egy mass 2}
    \int_{(0,t)\times\Omega} (\nu_{f_\ve}+\ve) \calT_{f_\ve}^{(\ve)}\rho_{f_\ve} \,\ds\dx \to \int_{(0,t)\times\Omega} \nu_f \calT_f \rho_f \,\ds\dx.
\end{align}
By \eqref{eq: MAC convergence}, \eqref{eq: Tve convergence}, and \eqref{eq: NU convergence}, the integrand of the left-hand side of \eqref{eq: egy mass 2} converges pointwise to $\nu_f \calT_f  \rho_f$ in $D_+:= \{(t,x): \rho_f(t,x)>0\}$. We claim now that the convergence holds strongly in $L^1(D_+)$. Similarly as before, we can resort to Egorov's theorem thus it is enough to prove that $\{(\nu_{f_\ve}+\ve)\calT_{f_\ve}^{(\ve)} \rho_{f_\ve}\}$ does not concentrate on Lebesgue-null sets of $(0,T)\times \Omega$. Since
\begin{align*}
    (\nu_{f_\ve} + \ve) \calT_{f_\ve}^{(\ve)} \rho_{f_\ve} &\le  (\|\nu\|_{L^\infty} + 1) (1 + \calT_{f_\ve}) \rho_\ve \\
    &\le C (\rho_\ve + V_{f_\ve}),
\end{align*}
the right-hand side of which is convergent in $L^1((0,T)\times\Omega)$, we deduce that $\{(\nu_{f_\ve}+\ve)\calT_{f_\ve}^{(\ve)} \rho_{f_\ve}\}$ is indeed equiintegrable over $(0,T)\times \Omega$.

On the other hand, on $D_0:= ((0,T)\times\Omega)\setminus D_+$, we have
\begin{align*}
    \lim_{\ve\downarrow 0} \int_{D_0} |(\nu_{f_\ve} + \ve) \calT_{f_\ve}^{(\ve)}\rho_{f_\ve} - \nu_f \calT_f d \rho_f | \,\ds\dx &= \lim_{\ve\downarrow 0} \int_{D_0} (\nu_{f_\ve} + \ve) \calT_{f_\ve}^{(\ve)} \rho_{f_\ve} \,\ds\dx \\
    &\le (\|\nu\|_{L^\infty} + 1) \lim_{\ve\downarrow 0} \int_{D_0} (1 + \calT_{f_\ve}) \rho_{f_\ve}  \,\ds\dx \\
    &\le C \lim_{\ve\downarrow 0} \int_{D_0} (\rho_{f_\ve} + V_{f_\ve}) \,\ds\dx \\
    &= C \int_{D_0} (\rho_f + V_f) \,\ds\dx \quad (\because \eqref{eq: MAC convergence}, \eqref{eq: V convergence}) \\
    &= 0.
\end{align*}
Altogether, we deduce
\begin{align*}
    (\nu_{f_\ve} + \ve) \calT_{f_\ve}^{(\ve)} \rho_{f_\ve} \to \nu_f \calT_f \rho_\ve \quad \text{in} \quad L^1((0,T)\times\Omega),
\end{align*}
and therefore the integral in \eqref{eq: egy mass 2} converges as $\ve\downarrow 0$. This proves the convergence of \eqref{eq: egy mass}.
\end{proof}

Finally, we check the convergence of the last term remaining in \eqref{eq: recall EGY}:

\begin{lemma}
In the limit $\ve\downarrow 0$, we have that
\begin{align*}
    \int_{(0,t)\times \calO} (\nu_{f_\ve} + \ve) v\cdot u_{f_\ve}^{(\ve)}  f_\ve \,\ds\dx\dv \to \int_{(0,t)\times \Omega} \nu_f \rho_f |u_f|^2 \,\ds\dx.
\end{align*}
\end{lemma}
\begin{proof}
Notice first that the integral in concern can be rewritten as
\begin{align}\label{eq: mass last int}
    \int_{(0,t)\times \calO} (\nu_{f_\ve} + \ve) v\cdot u_{f_\ve}^{(\ve)}  f_\ve \,\ds\dx\dv = \int_{(0,t)\times \Omega} (\nu_{f_\ve} + \ve) \rho_{f_\ve} u_{f_\ve} \cdot u_{f_\ve}^{(\ve)} \,\ds\dx.
\end{align}
On one hand, it is clear from \eqref{eq: MAC convergence}, \eqref{eq: U T convergence} that the integrand of the right-hand side converges pointwise to
\begin{align}\label{eq: rho u^2 ptw}
    (\nu_{f_\ve} + \ve)(\rho_{f_\ve} u_{f_\ve} \cdot u_{f_\ve}^{(\ve)}) \to \nu_f \rho_f |u_f|^2 \quad \text{a.e. in} \quad D_+ .
\end{align}
On the other hand, due to \eqref{eq: reg major} and Jensen's inequality, the following domination also holds:
\begin{align*}
    (\nu_{f_\ve} + \ve) \lt|\rho_{f_{\ve}} u_{f_\ve} \cdot u_{f_\ve}^{(\ve)}\rt| \le (\|\nu\|_{L^\infty} + 1) \rho_{f_\ve} |u_{f_\ve}|^2 \le C(\|\nu\|_{L^\infty}) \int_{\R^d} |v|^2 f_\ve \,\dv.
\end{align*}
By the averaging Lemma \ref{lem: subcubic},
\begin{align}\label{eq: v^2}
    \int_{\R^d} |v|^2 f_\ve \,\dv \to \intr |v|^2 f \,\dv \quad \text{in}\quad L^1((0,T)\times\Omega).
\end{align}
This and the pointwise convergence \eqref{eq: rho u^2 ptw} imply, through the Vitali convergence theorem, that 
\begin{align*}
    (\nu_{f_\ve} + \ve)(\rho_{f_\ve} u_{f_\ve} \cdot u_{f_\ve}^{(\ve)}) \to \nu_f \rho_f |u_f|^2 \quad \text{in}\quad L^1(D_+).
\end{align*}
On $D_0 := [(0,T)\times \Omega]\setminus D_+$, we have $f(t,x,v)=0$ a.e. $v\in\R^d$. Therefore, it follows that
\begin{align*}
    \int_{D_0} \lt| (\nu_{f_\ve} + \ve)(\rho_{f_\ve} u_{f_\ve} \cdot u_{f_\ve}^{(\ve)}) - \nu_f \rho_f |u_f|^2\rt| &= \int_{D_0} \lt|(\nu_{f_\ve} + \ve) (\rho_{f_\ve} u_{f_\ve} \cdot u_{f_\ve}^{(\ve)}) \rt| \,\ds\dx \\
    &\le (\|\nu\|_{L^\infty}+1) \int_{D_0} \rho_{f_\ve} |u_{f_\ve}|^2 \,\ds\dx\\
    &\le (\|\nu\|_{L^\infty}+1) \int_{D_0}\lt(\int_{\R^d} |v|^2 f_\ve \dv \rt) \,\ds\dx \\
    &\to (\|\nu\|_{L^\infty}+1) \int_{D_0}\int_{\R^d} |v|^2 f \,\dv\ds\dx \quad (\because \eqref{eq: v^2}) \\
    &= 0.
\end{align*}
We conclude together that
\begin{align*}
    (\nu_{f_\ve} + \ve)(\rho_{f_\ve} u_{f_\ve} \cdot u_{f_\ve}^{(\ve)}) \to \nu_f \rho_f |u_f|^2 \quad \text{in}\quad L^1((0,T)\times\Omega),
\end{align*}
and thus the integral in \eqref{eq: mass last int} converges.
\end{proof}

\subsubsection{Proof of Proposition \ref{prop: energy}}

It is enough to collect the results so far. Altogether, we have that
\begin{align*}
    &\int_{\calO} (1+|v|^2) f(t) \,\dx\dv - \int_{\calO} (1+|v|^2) f^0 \,\dx\dv + \int_{\Sigma_+^t} (1+|v|^2) \gamma f (n(x)\cdot v)_+ \,\ds\tnd\sigma\dv \\
    &\quad + \int_{(0,t)\times\calO} 2\nu_f|v|^2 f \,\ds\dx\dv \\
    &\le - \int_{\Sigma_-^t} (1+|v|^2) g (n(x)\cdot v)_-  \,\ds\tnd\sigma\dv   + 2\int_{(0,t)\times \calO} \nu_f d\calT_f f \,\ds\dx\dv + 2 \int_{(0,t)\times \Omega} \nu_f \rho_f |u_f|^2 \,\ds\dx\\
    &= - \int_{\Sigma_-^t} (1+|v|^2) g (n(x)\cdot v)_- \,\ds\tnd\sigma\dv + 2\int_{(0,t)\times\calO} 2 \nu_f |v|^2 f \,\ds\dx\dv.
\end{align*}
in other words
\begin{align*}
    &\int_{\calO} (1+|v|^2) f(t) \,\dx\dv - \int_{\calO} (1+|v|^2) f^0 \,\dx\dv + \int_{\Sigma_+^t} (1+|v|^2) \gamma f (n(x)\cdot v)_+ \,\ds\tnd\sigma\dv\\
    &\quad\le -\int_{\Sigma_-^t} (1+|v|^2) g (n(x)\cdot v)_- \,\ds\tnd\sigma\dv.
\end{align*}

%
%
%
%
%
%

\subsection{The entropy inequality} \label{sec: entropy}
In this section we prove the entropy inequality is satisfied by $f$.

\begin{proposition}\label{prop: entropy}
The limit $f$ satisfies the following entropy inequality
    \begin{align*}
    &\int_{\calO} f(t) \log f(t) \,\dx\dv - \int_{\calO} f^0 \log f^0 \,\dx\dv + \int_{\Sigma_+^t} \gamma f \log \gamma f (n(x)\cdot v)_+ \,\ds\tnd\sigma\dv\\
    &\quad + \int_{(0,t)\times \calO} \nu_f \frac{1}{f\calT_f} \lt|\calT_f \nabla_v f + (v-u_f) f\rt|^2 \,\ds\dx\dv\\
    &\le - \int_{\Sigma_-^t} g\log g (n(x)\cdot v)_- \,\ds\tnd\sigma\dv.
    \end{align*}
\end{proposition}

The proof method is similar as in the previous section. We recall from Lemma \ref{lem: entropy} (namely \eqref{eq: entropy identity}) that
\begin{equation} \label{eq: ENT ve}
\begin{split}
    &\int_{\calO} f_\ve(t) \log f_\ve(t) \,\dx\dv - \int_{\calO} f_\ve^0 \log f_\ve^0 \,\dx\dv + \int_{\Sigma_+^t} \gamma f_\ve \log \gamma f_\ve (n(x)\cdot v)_+ \,\ds\tnd\sigma\dv \\
        &\quad + 4\int_{(0,t)\times\calO} (\nu_{f_\ve} + \ve) \calT_{f_\ve}^{(\ve)} |\nabla_v \sqrt{f_\ve}|^2 \,\ds\dx\dv\\
        &= -\int_{\Sigma_-^t} g_\ve \log g_\ve (n(x)\cdot v)_- \,\ds\tnd\sigma \dv + \int_{(0,t)\times \calO} (\nu_{f_\ve} + \ve) f_\ve \lt(\nabla_v\cdot \renorm{v}{1}{\ve}\rt) \,\ds\dx\dv.
\end{split}
\end{equation}
Through subsequent integration by parts (validated through the pointwise decay \eqref{eq: ve decay}), we can easily check that the above is equivalent to
\begin{equation} \label{eq: ENT EQUIV}
\begin{split}
    &\int_{\calO} f_\ve(t) \log f_\ve(t) \,\dx\dv - \int_{\calO} f_\ve^0 \log f_\ve^0 \,\dx\dv + \int_{\Sigma_+^t} \gamma f_\ve \log \gamma f_\ve (n(x)\cdot v)_+ \,\ds\tnd\sigma\dv\\ 
    &\quad + \int_{(0,t)\times \calO} (\nu_{f_\ve} + \ve) \frac{1}{f_\ve \calT_{f_\ve}^{(\ve)}} \lt| \calT_{f_\ve}^{(\ve)} \nabla_v f_\ve + (v-u_{f_\ve}) f \rt|^2 \,\ds\dx\dv\\
    &= - \int_{\Sigma_-^t} g_\ve \log g_\ve (n(x)\cdot v)_- \,\ds\tnd\sigma\dv   + \int_{(0,t)\times \calO} (\nu_{f_\ve} + \ve) \lt( \frac{|v-u_{f_\ve}|^2 }{\calT_{f_\ve}^{(\ve)}} - d \rt) f_\ve \,\ds\dx\dv  \\
    &\quad + \int_{(0,t)\times \calO} (\nu_{f_\ve} + \ve) \lt( \nabla_v\cdot \renorm{v}{1}{\ve} - d \rt) f_\ve \,\ds\dx\dv \\
    &=: -R_1 + R_2 + R_3.
\end{split}
\end{equation}
The goal is to show that the left-hand side is lower semicontinuous as $\ve\downarrow 0$, whereas the right-hand side converges.

Let us check each term of the left-hand side first. Using the uniform energy estimates provided by Lemma \ref{lem: energy}, we can find by classical lower semicontinuity arguments \cite[p.9]{JKO98} that
\begin{align*}
    \int_{\calO} f(t) \log f(t) \,\dx\dv \le \liminf_{\ve\downarrow 0} \int_{\calO} f_\ve(t) \log f_\ve(t) \,\dx\dv .
\end{align*}
The assumption \eqref{ASSUMP f0 g} on the regularized initial data provides also that
\begin{align*}
    \lim_{\ve\downarrow 0} \int_{\calO} f_\ve^0 \log f_\ve^0 \,\dx\dv = \int_{\calO} f^0 \log f^0 \,\dx\dv.
\end{align*}
For the trace term $\gamma f_\ve \log \gamma f_\ve$, we handle it in the same way as in $f_\ve \log f_\ve$. Namely, we use the energy estimates of Lemma \ref{lem: energy}, the weak convergence
\begin{align*}
    \gamma f_\ve \weakto \gamma f \quad \text{in} \quad L^1(\Sigma_+^T, |n(x)\cdot v|\,\dt\tnd\sigma\dv),
\end{align*}
and the arguments of \cite{JKO98} to find
\begin{align*}
    \int_{\Sigma_+^t} \gamma f \log \gamma f \, (n(x)\cdot v)_+ \,\ds\tnd\sigma\dv \le \liminf_{\ve\downarrow 0} \int_{\Sigma_+^t} \gamma f_\ve \log \gamma f_\ve \, (n(x)\cdot v)_+ \,\ds\tnd\sigma\dv.
\end{align*}

Regarding the entropy dissipation, we save it for last. Let us move on to the right-hand side. The assumption \eqref{ASSUMP f0 g} on $\{g_\ve\}$ provides that
\begin{align*}
    \lim_{\ve\downarrow 0} R_1 := \lim_{\ve\downarrow 0} \int_{\Sigma_-^t} g_\ve \log g_\ve \, (n(x)\cdot v)_- \,\ds\tnd\sigma\dv = \int_{\Sigma_-^t} g \log g \, (n(x)\cdot v)_- \,\ds\tnd\sigma\dv.
\end{align*}
We now prove that $R_2\to 0$ in the limit. Let us begin by discussing the convergence of the following integral:
\begin{align*}
    \int_{(0,t)\times\calO} (\nu_{f_\ve} + \ve) \frac{|v-u_{f_\ve}|^2 f_\ve}{\calT_{f_\ve}^{(\ve)}} \,\ds\dx\dv.
\end{align*}
Integrating with respect to $v$ first, we find
\begin{align*}
    \int_{(0,t)\times\calO} (\nu_{f_\ve} + \ve) \frac{|v-u_{f_\ve}|^2 f_\ve}{\calT_{f_\ve}^{(\ve)}} \,\ds\dx\dv = \int_{(0,t)\times \Omega} (\nu_{f_\ve} + \ve) \frac{dV_{f_\ve}}{\calT_{f_\ve}^{(\ve)}} \,\ds\dx.
\end{align*}
On $D_+ := \{\rho_f > 0\}$, due to \eqref{eq: MAC convergence}, \eqref{eq: U T convergence}, and \eqref{eq: NU convergence}, it is evident that the integrand converges pointwise to
\begin{align}\label{eq: ENT LAST}
    (\nu_{f_\ve} + \ve) \frac{dV_{f_\ve}}{\calT_{f_\ve}^{(\ve)}} \to \frac{\nu_f}{\calT_f} V_f = \nu_f \rho_f.
\end{align}
Then, we observe that the relation
\begin{align*}
    \frac{1}{\calT_{f_\ve}^{(\ve)}} \le \frac{1}{\renorm{V_{f_\ve}}{\rho_{f_\ve}}{\ve}} := \frac{\rho_{f_\ve} + \ve(1 + V_{f_\ve})}{V_{f_\ve}} \le \frac{\rho_{f_\ve}}{V_{f_\ve}} + \frac{1}{V_{f_\ve}} + 1
\end{align*}
and the $L^1$ convergences in \eqref{eq: MAC convergence}, \eqref{eq: V convergence} show that for any measurable set $E\subset (0,T)\times \Omega$:
\begin{align*}
    &\sup_{\ve\in (0,1]} \int_E (\nu_{f_\ve} + \ve) \frac{dV_{f_\ve}}{\calT_{f_\ve}^{(\ve)}} \,\ds\dx  \le d(\|\nu\|_{L^\infty} + 1)   \sup_{\ve \in (0,1]} \int_E \lt( \rho_{f_\ve} + 1 + V_{f_\ve} \rt) \,\ds\dx  \to 0 \quad \text{as} \quad \scrL^{d+1}(E)\downarrow 0.
\end{align*}
The above, along with \eqref{eq: ENT LAST} and Egorov's theorem (applied in $(0,T)\times \Omega$), prove that
\begin{equation}\label{eq: ENT LAST 1}
\begin{split}
    \lim_{\ve\downarrow 0} \int_{(0,t)\times \calO} (\nu_{f_\ve} + \ve) \frac{|v-u_{f_\ve}|^2 f_\ve}{\calT_{f_\ve}^{(\ve)}} \,\ds\dx\dv  &= \lim_{\ve\downarrow 0} \int_{(0,t)\times \Omega} (\nu_{f_\ve} + \ve) \frac{d V_{f_\ve}}{\calT_{f_\ve}^{(\ve)}} \,\ds\dx \\
    &= \int_{(0,t)\times \Omega} \nu_f d\rho_f \,\ds\dx.
\end{split}
\end{equation}
To fully obtain the convergence of $R_2$, we must show that the following integral converges:
\begin{align*}
    \int_{(0,t)\times \calO} (\nu_{f_\ve} + \ve) df_\ve \,\ds\dx\dv = \int_{(0,t)\times \Omega} (\nu_{f_\ve} + \ve) d\rho_{f_\ve} \,\ds\dx \to \int_{(0,t)\times\Omega} \nu_f d \rho_f \,\ds\dx ,
\end{align*}
but this is easily checked using similar methods as before. Combining this with \eqref{eq: ENT LAST 1}, it follows that $R_2 \to 0$ as $\ve\downarrow 0$, as desired.

Next we similarly claim that $R_3\to 0$. A direct computation shows that
\begin{align*}
    \lt|\nabla_v\cdot \renorm{v}{1}{\ve} - d\rt| &= \ve \lt|\frac{d + (d+1)|v| + \ve d (1 + |v|)^2}{(1 + \ve (1+|v|))^2} \rt| \le \ve \lt| d + (d+1)|v| + \ve d (1 + |v|)^2 \rt|  \le \ve C_d (1+|v|^2),
\end{align*}
which, combined with the uniform bounds for the energy, gives
\begin{align*}
    |R_3| \le \ve \, C(\|\nu\|_{L^\infty},d) \int_{(0,t)\times \calO} (1+|v|^2) f_\ve \,\ds\dx\dv \xrightarrow[\ve\downarrow 0]{} 0.
\end{align*}

It only remains for us to discuss the entropy dissipation, the final term in the left-hand side of \eqref{eq: ENT EQUIV}. We need to show that it is lower semicontinuous in the limit. By our arguments so far, the right-hand side of \eqref{eq: ENT EQUIV} is convergent in the limit $\ve\downarrow 0$, and thus bounded independently of $\ve$. This provides first that the dissipation is bounded:
\begin{align*}
    \int_{(0,t)\times \calO} (\nu_{f_\ve} + \ve) \frac{1}{f_\ve \calT_{f_\ve}^{(\ve)}} \lt| \calT_{f_\ve}^{(\ve)} \nabla_v f_\ve + (v-u_{f_\ve}) f_\ve \rt|^2 \,\ds\dx\dv \le C_T,
\end{align*}
from which we deduce that the sequence
\begin{align*}
    \lt\{\sqrt{\nu_{f_\ve} + \ve} \frac{\calT_{f_\ve}^{(\ve)}\nabla_v f_\ve + (v-u_{f_\ve}) f_\ve}{\sqrt{f_\ve \calT_{f_\ve}^{(\ve)}}} \rt\}
\end{align*}
is uniformly bounded in $L^2((0,T)\times \calO)$ as $\ve\downarrow 0$. Owing to the pointwise convergence of all terms in the sequence, the weak $L^2$-limit can be identified in the same way as in the beginning of the proof of Lemma \ref{lem: prod sqrt STR0}. Since the $L^2$-norm is lower semicontinuous with respect to convergence in the weak topology,
\begin{align*}
    &\int_{(0,t)\times \calO} \nu_f \frac{1}{f \calT_f} \lt| \calT_f \nabla_v f + (v-u_f) f \rt|^2 \,\ds\dx\dv \\
    &\quad \le \liminf_{\ve\downarrow 0} \int_{(0,t)\times \calO} (\nu_{f_\ve} + \ve) \frac{1}{f_\ve \calT_{f_\ve}^{(\ve)}} \lt| \calT_{f_\ve}^{(\ve)} \nabla_v f_\ve + (v-u_{f_\ve}) f_\ve \rt|^2 \,\ds\dx\dv,
\end{align*}
as desired.

Collecting everything established so far, we have shown that the left-hand side of \eqref{eq: ENT EQUIV} is lower semicontinuous whereas the right-hand side is convergent in the limit $\ve\downarrow 0$. The assertion of Proposition \ref{prop: entropy} follows.

\begin{remark}
    One can choose to pass to the limit in \eqref{eq: ENT ve} instead. We then obtain using similar methods that
    \begin{equation*}
\begin{split}
    &\int_{\calO} f(t) \log f(t) \,\dx\dv - \int_{\calO} f^0 \log f^0 \,\dx\dv + \int_{\Sigma_+^t} \gamma f \log \gamma f (n(x)\cdot v)_+ \,\ds\tnd\sigma\dv \\
        &\quad + 4\int_{(0,t)\times\calO} \nu_f \calT_f |\nabla_v \sqrt{f}|^2 \,\ds\dx\dv \\
        &\le -\int_{\Sigma_-^t} g\log g (n(x)\cdot v)_- \,\ds\tnd\sigma\dv + d \int_{(0,t)\times \calO} \nu_f f \,\ds\dx\dv.
\end{split}
\end{equation*}
The passage to the limit of \eqref{eq: ENT ve} is simpler than passing to the limit in \eqref{eq: ENT EQUIV} and thus we omit the proof of the above identity.
\end{remark}

\noindent \textbf{Conclusion.} The results of Section \ref{sec: very weak}, \ref{sec: weak}, \ref{sec: energy}, and \ref{sec: entropy} demonstrate that all assertions of Theorem \ref{thm: inflow} hold, thereby completing its proof.

%
%
%
%
%
%

\section{The partial absorption-reflection problem} \label{sec: reflect}

Compared with the inflow case studied in Section \ref{sec: inflow}, the partially absorbing-reflecting boundary condition introduces a new mechanism at the boundary: when particles reach $\Sigma_-$, a fraction $\theta$ of the incoming flux is reflected specularly, while the remaining part is absorbed. This mixture of absorption and reflection modifies the energy and entropy balances by introducing boundary terms weighted by $(1-\theta)$. In particular, one expects the uniform bounds for moments and entropy to weaken as $\theta \uparrow 1$, reflecting the fact that pure reflection preserves the boundary flux. Nevertheless, as we shall see, the fundamental structure of the a priori estimates is still preserved, and the compactness arguments developed in Section \ref{sec: comp} continue to apply without essential changes. Our task in this section is therefore to adapt the approximation, renormalization, and compactness framework to this mixed boundary setting and establish the existence of weak solutions under the same physical assumptions on the initial data.

 Let us recall:
\begin{align*}
    \begin{cases}
        \p_t f + v\cdot \nabla_x f = \nu_f \nabla_v\cdot \lt( \calT_f \nabla_v f + (v - u_f) f \rt), \\
        f|_{t=0} = f^0, \\
        \gamma f(t,x,v) = \theta \gamma f(t,x,L_xv), \quad \forall (t,x,v)\in \Sigma_-^T, \quad \theta\in [0,1),\\
        L_xv := v - 2(n(x)\cdot v) n(x).
    \end{cases}
\end{align*}
We emphasize once more that the only assumptions on the initial data are in \eqref{init phys}.

In analogy with the previous section, we fix $\ve>0$ and consider the following regularized problem
\begin{align} \label{spec: eq: reg}
    \begin{cases}
        \p_t f_\ve + v\cdot \nabla_x f_\ve = (\nu_{f_\ve} + \ve) \nabla_v \cdot \lt(\calT_{f_\ve}^{(\ve)} \nabla_v f_\ve + ( \renorm{v}{1}{\ve} - u_{f_\ve}^{(\ve)}) f_\ve \rt), \\
        f|_{t=0} = f_\ve^0, \\
        \gamma f_\ve(t,x,v) = \theta \gamma f_\ve(t,x,L_x v) \quad \forall (t,x,v)\in \Sigma_-^T,
    \end{cases}
\end{align}
in $(0,T)\times \Omega \times \R^d$. The regularized initial data is assumed to be of class
\begin{align*}
    f_\ve^0 \in C_c(\overline{\Omega}\times \R^d),
\end{align*}
and to satisfy
\begin{align*}
\begin{cases}
    \displaystyle\sup_{\ve\in (0,1]} \int_{\calO} (1+|v|^2 + |\log f_\ve^0|) f_\ve^0 \, \dx\dv < + \infty, \\
    \displaystyle f_\ve^0 \to f^0 \quad \text{in} \quad L^1(\calO),\\
    \displaystyle \lim_{\ve\downarrow 0} \int_{\calO} |v|^2 f_\ve^0 \,\dx\dv = \int_{\calO} |v|^2 f^0 \,\dx\dv, \\
    \displaystyle \lim_{\ve\downarrow 0} \int_{\calO} f_\ve^0 \log f_\ve^0 \,\dx\dv = \int_{\calO} f^0 \log f^0 \,\dx\dv.
\end{cases}
\end{align*}

Similarly, we also consider the corresponding linear counterpart of \eqref{spec: eq: reg} as
\begin{align} \label{spec: eq: lin}
    \begin{cases}
        \p_t F + v\cdot \nabla_x F = (\nu(\varrho, \bfj, V) + \ve) \lt(\renorm{V}{\varrho}{\ve} + \ve \rt) \Delta_v F + (\nu(\varrho, \bfj, V) + \ve) \nabla_v\cdot \lt( (\renorm{v}{1}{\ve} - \renorm{\bfj}{\varrho}{\ve}) F\rt) ,\\
        F|_{t=0} = f^0_\ve , \\
        \gamma F(t,x,v) = \theta \gamma F(t,x,L_xv) \quad \forall (t,x,v)\in \Sigma_-^T,
    \end{cases}
\end{align}
where the coefficients are extracted from (see Section \ref{sec: fixed})
\begin{align*}
    (\varrho, \bfj, V) \in \scrX.
\end{align*}
We refer again to \cite[Theorem 1.1]{Zhu24} for the existence, uniqueness, renormalization properties, and weighted estimates for the (bounded) solution $(F,\gamma F)$ to \eqref{spec: eq: lin}. Properly speaking, Zhu's main result is written only for the purely reflecting case of $\theta=1$, but following that work we note that the cases of $\theta\in [0,1)$ can be treated within their framework (see \cite[Sections 2, 5]{Zhu24}). 

\begin{remark}
It is also worth mentioning that the solution $(F,\gamma F)$ to \eqref{spec: eq: lin} is nonnegative (in analogue with Lemma \ref{lem: nonnegative}. Indeed, the solution $(F,\gamma F)$ is constructed in \cite[Section 2.2.2]{Zhu24} by approximating \eqref{spec: eq: lin} with a sequence of inflow problems (see also the more classical work \cite[Section 5]{Carrillo98}). Since the maximum principle Lemma \ref{lem: nonnegative} demonstrates that we can guarantee nonnegativity of the solutions to the inflow problem, this carries over to the solution to \eqref{spec: eq: lin} in passage to the limit.
\end{remark}

Then through the same fixed point argument, we obtain the existence of a weak solution to \eqref{spec: eq: reg}.

\begin{lemma}
    There exists a weak and renormalized solution $(f_\ve,\gamma f_\ve)$ to the regularized system \eqref{spec: eq: reg} which satisfies (not necessarily uniformly in $\ve$)
    \begin{align*}
        \begin{cases}
         f_\ve, \gamma f_\ve \ge 0 ; \\
        \lt<v\rt>^q f_\ve \in L^\infty(0,T;L^1 \cap L^\infty(\calO)) \quad \forall q \ge 0, \\
         \gamma f_\ve \in L^1\cap L^\infty(\Sigma^T,|n(x)\cdot v|\dt\tnd\sigma\dv).
            \end{cases}
    \end{align*}
\end{lemma}

\begin{proof}
Since the scheme is the same as that of Section \ref{sec: fixed} we only provide a sketch of proof. We note that the solution to \eqref{spec: eq: reg} is realized through a fixed point of the operator
\begin{align*}
    \scrR : (\varrho, \bfj, V) \mapsto (\rho_F, j_F, V_F).
\end{align*}
We can then check, in the same way as before, that $\scrR$ is continuous and compact as a map from $\scrX$ into itself. Schaefer's theorem then yields that $\scrR$ has a fixed point in $\scrX$, from which we deduce the existence of a solution $(f_\ve,\gamma f_\ve)$ to \eqref{spec: eq: reg}.
\end{proof}

Then, using the renormalization properties of \eqref{spec: eq: reg}, we can derive the following uniform estimates in similar manner as that of Lemmas \ref{lem: energy}, \ref{lem: entropy}.

\begin{lemma} \label{spec: lem: unif}
    Let $\{f_\ve\}_{\ve\in (0,1]}$ denote the family of solutions to \eqref{spec: eq: reg}. Then the following estimates hold uniformly in $\ve$:
    \begin{align*}
        \begin{cases}
        \displaystyle \sup_{t\in [0,T]} \int_{\calO} (1+|v|^2) f_\ve(t) \,\dx\dv  \le C_T,\\
        \displaystyle \int_{\Sigma_+^T} (1+|v|^2) \gamma f_\ve \,\dt\tnd\sigma\dv \le \frac{C_T}{1-\theta}, \\
        \displaystyle \sup_{t\in [0,T]} \int_{\calO} f_\ve(t) \log f_\ve(t) \,\dx\dv \le C_T, \\
        \displaystyle \int_{\Sigma_+^T} \gamma f_\ve \log \gamma f_\ve \,\dt\tnd\sigma\dv \le \frac{C_T}{1-\theta}, \\
        \displaystyle \int_{(0,T)\times \calO} (\nu_{f_\ve} + \ve) \calT_{f_\ve}^{(\ve)} |\nabla_v \sqrt{f_\ve}|^2 \,\dt\dx\dv \le C_T. 
        \end{cases}
    \end{align*}
\end{lemma}
\begin{proof}
    We only provide a formal proof, and refer to Lemmas \ref{lem: energy} and \ref{lem: entropy} for the rigorous derivation of these identities. A direct computation and integrating by parts allow us to find
    \begin{align*}
        &\ddt \int_{\calO} (1+|v|^2) f_\ve \,\dx\dv + \int_{\Gamma_+} (1+|v|^2) \gamma f_\ve (n(x)\cdot v)_+ \,\tnd\sigma\dv \\
        &\quad = -\int_{\Gamma_-} (1+|v|^2) \gamma f_\ve \,(n(x)\cdot v)_- \,\tnd\sigma\dv + \int_{\calO} (\nu_{f_\ve} + \ve) \lt(2 d \calT_{f_\ve}^{(\ve)} f_\ve - 2v\cdot (\renorm{v}{1}{\ve} - u_{f_\ve}^{(\ve)}) f_\ve \rt) \,\dx\dv.
    \end{align*}
    Using the prescribed boundary condition, and also that the reflection map $L_x$ has unit Jacobian, we observe 
    \begin{align*}
        -\int_{\Gamma_-} (1+|v|^2) \gamma f_\ve (n(x)\cdot v)_- \,\tnd\sigma\dv = \theta \int_{\Gamma_+} (1+|v|^2)  \gamma f_\ve (n(x)\cdot v)_+ \,\tnd\sigma\dv.
    \end{align*}
    Therefore
    \begin{align*}
        &\ddt \int_{\calO} (1+|v|^2) f_\ve \,\dx\dv + (1-\theta) \int_{\Gamma_+} (1+|v|^2) \gamma f_\ve (n(x)\cdot v)_+ \,\tnd\sigma\dv \\
        &\quad = \int_{\calO} (\nu_{f_\ve} + \ve) \lt(2d \calT_{f_\ve}^{(\ve)} f_\ve - 2v\cdot (\renorm{v}{1}{\ve} - u_{f_\ve}^{(\ve)}) f_\ve \rt) \,\dx\dv.
    \end{align*}
    The right-hand side of the above was already estimated in Lemma \ref{lem: energy}, and from there, applying Gr\"onwall's lemma yields
    \begin{equation}\label{spec: eq: energy}
        \sup_{t\in [0,T]} \int_{\calO} (1+|v|^2) f_\ve \,\dx\dv \le C_T.
    \end{equation}
    Using this bound, we can then in turn deduce
    \begin{align*}
        \int_{\Sigma_+^T} (1+|v|^2) \gamma f_\ve (n(x)\cdot v)_+ \,\dt\tnd\sigma\dv \le \frac{C_T}{1-\theta}.
    \end{align*}

    In similar nature, we can find the following formal identity for the entropy:
    \begin{equation} \label{spec: eq: entropy}
    \begin{split}
        &\ddt \int_{\calO} f_\ve \log f_\ve \,\dx\dv + (1 - \theta) \int_{\Gamma_+} \gamma f_\ve \log \gamma f_\ve \,(n(x)\cdot v)_+ \,\tnd\sigma\dv   + 4 \int_{\calO} (\nu_{f_\ve} + \ve) \calT_{f_\ve}^{(\ve)} |\nabla_v \sqrt{f_\ve}|^2 \,\dx\dv \\
        &\quad = \int_{\Gamma_+} \gamma f_\ve (\theta \log \theta) (n(x)\cdot v)_+ \,\tnd\sigma\dv +  d \int_{\calO} (\nu_{f_\ve} + \ve) f_\ve \,\dx\dv \\
        &\quad \le  d \int_{\calO} (\nu_{f_\ve} + \ve) f_\ve \,\dx\dv \quad (\because \theta \log \theta \le 0).
    \end{split}
    \end{equation}
    We note that since $\nu\in L^\infty$, \eqref{spec: eq: energy} allows us to estimate the right-hand side of \eqref{spec: eq: entropy} as
    \begin{align*}
        \int_{\calO} (\nu_{f_\ve} + \ve) f_\ve\,\dx\dv \le (\|\nu\|_{L^\infty} + 1) C_T \quad \forall \ve\in (0,1].
    \end{align*}
    Therefore, we deduce from \eqref{spec: eq: entropy} and Gr\"onwall's lemma
    \begin{align*}
        \sup_{t\in [0,T]} \int_{\calO} f_\ve \log f_\ve \,\dx\dv \le C_T,
    \end{align*}
    which in turn yields
    \begin{align*}
        \int_{\Sigma_+^T} \gamma f_\ve \log \gamma f_\ve \,\dt\tnd\sigma\dv \le \frac{C_T}{1-\theta}, \qquad \int_{(0,T)\times\calO} (\nu_{f_\ve} + \ve) \calT_{f_\ve}^{(\ve)}|\nabla_v \sqrt{f_\ve}|^2 \,\dt\dx\dv \le C_T.
    \end{align*}
\end{proof}

The estimates of Lemma \ref{spec: lem: unif} imply through the Dunford--Pettis criteria that we can pass to a subsequence such that there exists a limit $f$ satisfying
\begin{align*}
    &f_\ve \to f \quad \text{weakly in} \quad L^1((0,T)\times\calO), \\
    &\gamma f_\ve \to \gamma f \quad \text{weakly in}\quad L^1(\Sigma_+^T, (v\cdot n(x))\,\dt\tnd\sigma\dv).
\end{align*}
We now assume always that we are along this subsequence.

Next, analogously as with Lemma \ref{lem: ve higher moments}, we provide the gain of moments, which is key to obtaining compactness for the temperature.

\begin{lemma}
    Let $\{f_\ve\}_{\ve\in (0,1]}$ denote the family of solutions to \eqref{spec: eq: reg}. Then
\[
        \sup_{\ve\in (0,1]} \int_{(0,T)\times \calO} |v|^3 f_\ve(t) \,\dt\dx\dv \le C_T.
\]
\end{lemma}
\begin{proof}
    We just remark that it is clear from the boundary condition and Lemma \ref{spec: lem: unif} that $\lt\{\gamma f_\ve|_{\Sigma_-^T}\rt\}_{\ve\in(0,1]}$ also has uniformly bounded second moments. Indeed, changing variables shows
    \begin{align*}
        \int_{\Sigma_-^T} |v|^2 \gamma f_\ve |v\cdot n(x)|\,\dt\tnd\sigma\dv &= \theta \int_{\Sigma_+^T} |v|^2 \gamma f_\ve (v\cdot n(x))_+ \,\dt\tnd\sigma\dv \le \theta C_{\theta,T}.
    \end{align*}
    Therefore, it is enough to apply Lemma \ref{lem: higher moments} with $k=2$ and in the same setting as with Lemma \ref{lem: ve higher moments}. 
\end{proof}

We may then provide the analogue of Lemma \ref{lem: subcubic}.

\begin{lemma}\label{spec: lem: subcubic}
    For any $\varphi\in C^0(\R^d)$ with subcubic growth, it holds that the averages
    \begin{align*}
        \lt\{\intr f_\ve \varphi(v)\,\dv\rt\}_{\ve\in (0,1]}
    \end{align*}
    are compact in $L^1((0,T)\times \Omega)$.
\end{lemma}
\begin{proof}
    Using the renormalization formula, namely the same identity as in \eqref{eq: h1 h2}, we may apply the velocity averaging Lemma \ref{lem: L2 vel avg} to deduce the compactness of the averages
    \begin{align*}
        \lt\{\intr \frac{f_\ve}{1 + \delta f_\ve} \varphi(v) \,\dv\rt\}_{\ve\in (0,1]}
    \end{align*}
    for each $\delta>0$. Then, in the same way as in the proof of Lemma \ref{lem: subcubic}, we deduce using the equiintegrability of $\{f_\ve\}$ and taking $\delta\downarrow 0$ that the averages $\lt\{\intr f_\ve \varphi(v)\,\dv\rt\}$ are also compact in $L^1((0,T)\times\Omega)$.
\end{proof}

Finally, we may apply Proposition \ref{prop: weighted fisher} to deduce that
\begin{lemma}\label{spec: lem: ptw}
    Modulo a subsequence, 
    \begin{align*}
        f_\ve \to f \quad \text{a.e. and in} \quad L^1((0,T)\times\calO).
    \end{align*}
\end{lemma}
\begin{proof}
    The proof is the same as that with Lemma \ref{lem: ptw}, and it is a simple application of Proposition \ref{prop: weighted fisher} applied with $a_\ve = (\nu_{f_\ve} + \ve)\calT_{f_\ve}^{(\ve)}$.
\end{proof}

The main result now follows analogously as with the previous Section \ref{sec: inflow}.

\begin{proof}[Proof of Theorem \ref{thm: reflection}]
Collecting Lemmas \ref{spec: lem: subcubic} and \ref{spec: lem: ptw}, we have (up to some subsequence) that
\begin{align*}
    \begin{cases}
    f_\ve \to f \quad \text{a.e. and in}\quad L^1((0,T)\times \calO), \\
    \rho_{f_\ve} \to \rho_f \quad \text{a.e. and in}\quad L^1((0,T)\times \Omega),\\
    j_{f_\ve} \to j_f \quad \text{a.e. and in}\quad L^1((0,T)\times\Omega), \\
    V_{f_\ve} \to V_f \quad \text{a.e. and in}\quad L^1((0,T)\times \Omega), \\[4pt]
    u_{f_\ve}^{(\ve)}\to u_f \quad \text{a.e. in}\quad D_+ := \{(t,x)\in (0,T)\times\Omega: \rho_f(t,x) > 0 \}, \\[4pt]
    \calT_{f_\ve}^{(\ve)} \to \calT_f \quad \text{a.e. in} \quad D_+, \\
    \nu_{f_\ve} + \ve \to \nu_f \quad \text{a.e. in} \quad (0,T)\times\Omega.
    \end{cases}
\end{align*}
Finally we recall also that $\gamma f_\ve \weakto \gamma f$ weakly in $L^1(\Sigma_+^T,v\cdot n(x)\,\dt\tnd\sigma\dv)$. Then we may proceed in the same way as with Sections \ref{sec: very weak}, \ref{sec: weak}, \ref{sec: energy}, \ref{sec: entropy}, using the above convergences and Proposition \ref{prop: weighted fisher} to pass to the limit in each property satisfied by the $(f_\ve,\gamma f_\ve)$. The assertions of Theorem \ref{thm: reflection} are then consequently deduced.
\end{proof}

%
%
%
%
%
%

\section*{Acknowledgments}
This work is supported by NRF grant no. 2022R1A2C1002820 and RS-2024-00406821. The authors would like to thank Jos\'e Carrillo for his valuable comments on this work. 

%
%
%
%
%
%
\appendix
\section{Construction of regularized initial and boundary data} \label{app: REG INIT}

In this appendix, we present a construction of regularizations satisfying the assumption \eqref{ASSUMP f0 g}.

\begin{lemma} \label{lem: INIT}
    Let $(1+|v|^2 + |\log f^0|)f^0\in L^1_+(\calO)$. Then there exists a family $\{f^0_\ve\}_{\ve\in (0,1]} \subset C_c(\overline{\Omega}\times\R^d)$ such that
    \begin{align}
        &f_\ve^0 \to f^0 \quad \text{in} \quad L^1(\calO), \label{eq: INIT L1 conv}\\
        &\lim_{\ve\downarrow 0} \int_{\calO} f_\ve^0 \log f_\ve^0 \,\dx\dv  = \int_{\calO} f^0 \log f^0 \,\dx\dv , \label{eq: INIT ENT conv}\\
        &\lim_{\ve\downarrow 0} \int_{\calO} |v|^2 f_\ve^0 \,\dx\dv = \int_{\calO} |v|^2 f^0 \,\dx\dv. \label{eq: INIT EGY conv}
    \end{align}
\end{lemma}
\begin{proof}
    \textit{Step 1: $f^0\in L^\infty(\Omega;L^\infty_c(\R^d))$.} We assume first that $f^0\in L^\infty(\Omega;L^\infty_c(\R^d))$, where
    \begin{align*}
        L^\infty(\Omega;L^\infty_c(\R^d)) := \{h\in L^\infty(\calO) \, : \, \textnormal{supp}_v (h) \text{ is a compact subset of $\R^d$} \}.
    \end{align*}
    Fix $r>0$ such that $f^0(x,v) = 0$ outside $\Omega\times B_r$. Denote $M=\|f^0\|_{L^\infty(\calO)}$. Define $\{f_\ve^0\}_{\ve\in (0,1]}$ to be a sequence of continuous functions such that
    \begin{align*}
    \begin{cases}
        \textnormal{supp}(f_\ve^0) \subset \Omega\times B_r \quad \forall \ve\in (0,1],\\
        f_\ve^0 \to f^0 \quad \text{a.e. in}\quad \Omega\times B_r,\\
        0\le f_\ve^0 \le M.
    \end{cases}
    \end{align*}
    Then, it follows immediately from the compact $v$-support and the Lebesgue dominated convergence theorem that \eqref{eq: INIT L1 conv}, \eqref{eq: INIT ENT conv}, \eqref{eq: INIT EGY conv} all hold.

    \textit{Step 2: General case.} Next, we assume only that $f^0$ satisfies the conditions given in the lemma. Thanks to Step 1, it is enough for us to find a sequence $\{f_\ve^0\}$ such that
    \begin{align*}
        \begin{cases}
        f_\ve^0\in L^\infty(\Omega;L^\infty_c(\R^d)),\\
        \textnormal{$f_\ve^0$ satisfies \eqref{eq: INIT L1 conv}, \eqref{eq: INIT ENT conv}, \eqref{eq: INIT EGY conv}}.
        \end{cases}
    \end{align*}
    Let us adapt some of the arguments in \cite{JKO98}. For any $\zeta\in (0,1)$, we know there exists a constant $C>0$ for which
    \begin{align*}
        |\min\{z\log z, 0 \}| \le C z^\zeta.
    \end{align*}
    This gives us the following estimate:
    \begin{equation} \label{eq: FLOGF}
    \begin{split}
        &\int_{\Omega\times (\R^d\setminus B_R)} |\min\{f^0 \log f^0, 0\}| \,\dx\dv \\
        &\quad \le C \int_{\Omega\times (\R^d\setminus B_R)} (f^0)^\zeta \,\dx\dv \\
        &\quad  \le C\lt(\int_{\Omega\times(\R^d\setminus B_R)} (1+|v|^2) f^0 \,\dx\dv \rt)^{\zeta} \lt(\int_{\Omega\times (\R^d\setminus B_R)} \frac{\dx\dv}{(1+|v|^2)^{\frac{\zeta}{(1-\zeta)}}}\rt)^{1-\zeta} \\
        &\quad  \le C_d \lt(\frac{1}{1+R^2}\rt)^{\zeta - \frac{d}{2}(1-\zeta)}.
    \end{split}\end{equation}
    Now take
    \begin{align*}
        f_\ve^0(x,v) := \min\lt\{\mathbf{1}_{|v|\le 1/\ve} f^0(x,v), \frac{1}{\ve} \rt\} \in L^\infty(\Omega;L^\infty_c(\R^d)).
    \end{align*}
    Then $f_\ve^0$ increases to $f^0$ as $\ve\downarrow 0$. Hence, the monotone convergence theorem immediately implies \eqref{eq: INIT L1 conv} and \eqref{eq: INIT EGY conv}. Moreover,
    \begin{align*}
        \sup_{\ve} \int_{\calO} (1+|v|^2) f_\ve^0 \,\dx\dv \le C,
    \end{align*}
    and therefore the same estimate of \eqref{eq: FLOGF} gives
    \begin{align*}
        \sup_\ve \int_{\Omega\times (\R^d\setminus B_R)} |\min\{f_\ve^0 \log f_\ve^0 , 0 \}|\,\dx\dv  \le C_d  \lt(\frac{1}{1+R^2}\rt)^{\zeta - \frac{d}{2}(1-\zeta)}, \quad \forall \zeta \in (0,1).
    \end{align*}
    Consequently, for any $\tau>0$, we may fix $\zeta > \frac{d}{d+2}$ and $R$ so large that
    \begin{align*}
         \int_{\Omega\times (\R^d\setminus B_R)} |\min\{f^0 \log f^0, 0 \}| \,\dx\dv  \le \tau, \quad \sup_\ve \int_{\Omega\times (\R^d\setminus B_R)} |\min\{f_\ve^0 \log f_\ve^0 , 0 \}| \,\dx\dv \le \tau.
    \end{align*}
    Now write
    \begin{align*}
        \int_{\calO} f_\ve^0 \log f_\ve^0 \,\dx\dv &= \int_{\Omega\times B_R} \min\{f_\ve^0 \log f_\ve^0 , 0 \} \,\dx\dv + \int_{\Omega\times (\R^d\setminus B_R)} \min\{f_\ve^0 \log f_\ve^0 , 0 \} \,\dx\dv \\
        &\quad + \int_{\calO} \max\{f_\ve^0 \log f_\ve^0, 0 \} \,\dx\dv  \\
        &=: A_1 + A_2 + A_3.
    \end{align*}
    For $A_1$, we observe that $|\min\{f_\ve^0\log f_\ve^0, 0 \}| \le 1/e$, the right-hand side of which is integrable over $\Omega\times B_R$. By the Lebesgue dominated convergence theorem, we deduce that
    \begin{align*}
        A_1 \xrightarrow[\ve\downarrow 0]{} \int_{\Omega\times B_R} \min\{f^0 \log f^0, 0 \} \,\dx\dv .
    \end{align*}
    For $A_2$, our previous estimate gives
    \begin{align*}
        \sup_\ve A_2 \le \tau.
    \end{align*}
    Finally, since $f_\ve^0$ increases to $f^0$, we have that $\max\{f_\ve^0 \log f_\ve^0, 0\}$ increases to $\max\{f^0\log f^0, 0\}$. The monotone convergence theorem yields
    \begin{align*}
        A_3 \xrightarrow[\ve\downarrow 0]{} \int_{\calO} \max\{f^0 \log f^0, 0 \}.
    \end{align*}
    In conclusion,
    \begin{align*}
        &\limsup_{\ve\downarrow 0} \int_{\calO} (f_\ve^0 \log f_\ve^0 - f^0 \log f^0) \,\dx\dv \\
        &\quad \le \limsup_{\ve\downarrow 0} \int_{\Omega\times (\R^d\setminus B_R)} \lt( |\min\{f_\ve^0 \log f_\ve^0, 0 \}| + |\min\{f^0 \log f^0, 0 \} | \rt) \,\dx\dv \\
        &\quad \le 2\tau.
    \end{align*}
    Since our choice of $\tau>0$ was arbitrary, we deduce \eqref{eq: INIT ENT conv}.
\end{proof}

\begin{lemma}
    Assume that $g$ is a measurable function defined on $\Sigma^T$, and that $(1+|v|^2 + |\log g|)g\in L^1_+(\Sigma_-^T, |n(x)\cdot v|\dt\tnd\sigma\dv)$. Then there exists a sequence $\{g_\ve\} \in C_c([0,T]\times \p\Omega \times\R^d)$ such that
    \begin{align}\label{eq: app: condition}
        \begin{cases}
        g_\ve \to g \quad \text{in} \quad L^1(\Sigma_-^T, |n(x)\cdot v|\dt\tnd\sigma\dv), \\
        \displaystyle\lim_{\ve\downarrow 0} \int_{\Sigma_-^T} (1+|v|^2) g_\ve |n(x)\cdot v| \,\dt\tnd\sigma\dv = \int_{\Sigma_-^T} (1+|v|^2) g |n(x)\cdot v| \,\dt\tnd\sigma\dv, \\
        \displaystyle\lim_{\ve\downarrow 0} \int_{\Sigma_-^T} g_\ve \log g_\ve |n(x)\cdot v| \,\dt\tnd\sigma\dv = \int_{\Sigma_-^T} g \log g |n(x)\cdot v| \,\dt\tnd\sigma\dv.
        \end{cases}
    \end{align}
\end{lemma}
\begin{proof}
The proof follows the same lines as Lemma \ref{lem: INIT}, so we only indicate the main modifications. By proceeding as in Step 1 of the previous lemma, it is enough to find a sequence $\{g_\ve\}$ in $L^\infty([0,T]\times \p\Omega; L^\infty_c(\R^d))$ which satisfies the conditions in \eqref{eq: app: condition}. We set
    \begin{align*}
        g_\ve(t,x,v) := \min\lt\{ \frac{1}{\ve} \; , \; \mathbf{1}_{|v|\le 1/\ve} g(t,x,v)  \rt\}.
    \end{align*}
    In this way, $g_\ve$ increases pointwise to $g$, and thus the monotone convergence theorem shows
    \begin{align*}
        &g_\ve \to g \quad \text{in} \quad L^1(\Sigma_-^T, |n(x)\cdot v| \dt\tnd\sigma\dv),\\
        &\int_{\Sigma_-^T} (1 + |v|^2) g_\ve |n(x)\cdot v| \,\dt\tnd\sigma\dv \le \int_{\Sigma_-^T} (1 + |v|^2) g |n(x)\cdot v| \,\dt\tnd\sigma\dv \le  C,\\
        &\lim_{\ve\downarrow 0} \int_{\Sigma_-^T} (1+|v|^2) g_\ve |n(x)\cdot v| \,\dt\tnd\sigma\dv = \int_{\Sigma_-^T} (1+|v|^2) g |n(x)\cdot v| \,\dt\tnd\sigma\dv.
    \end{align*}
    To complete the proof of \eqref{eq: app: condition} it only remains to obtain the convergence of the entropy. We can obtain as in \eqref{eq: FLOGF}, that for any $R>0$ and $\zeta\in (0,1)$
    \begin{align*}
        &\int_{\Sigma_-^T , |v|\ge R} |\min\{g_\ve \log g_\ve, 0\}| |n(x)\cdot v| \,\dt\tnd\sigma\dv \\
        &\quad \le C \lt(\int_{\Sigma_-^T, |v|\ge R} (1+|v|^2) g_\ve |n(x)\cdot v| \,\dt\tnd\sigma\dv \rt)^\zeta \lt(\int_{\Sigma_-^T, |v|\ge R} \frac{|n(x)\cdot v|}{(1+|v|^2)^{\frac{\zeta}{1-\zeta}}} \dt\tnd\sigma\dv \rt)^{1-\zeta} \\
        &\quad  \le C T \sigma(\p\Omega) \lt(\frac{1}{1+R^2}\rt)^{\zeta - \frac{(d+1)(1-\zeta)}{2}}.
    \end{align*}
    Thus, for any $\tau>0$, we can then choose $\zeta > \frac{d+1}{d+3}$ and set $R$ arbitrarily large so that
    \begin{align*}
        &\sup_\ve \int_{\Sigma_-^T, |v|\ge R} |\min\{g_\ve\log g_\ve, 0\}| |n(x)\cdot v| \,\dt\tnd\sigma\dv \le \tau,\\
        &\int_{\Sigma_-^T, |v|\ge R} |\min\{g\log g, 0\}| |n(x)\cdot v| \,\dt\tnd\sigma\dv \le \tau.
    \end{align*}
    The remaining parts of the proof are then the same as with Lemma \ref{lem: INIT}.
\end{proof}

%
%
%
%
%
%


\begin{thebibliography}{GHJO20}

\bibitem[1]{kirkwood1946}
{\sc J.~G. Kirkwood.}
\newblock The statistical mechanical theory of transport processes i. general theory.
\newblock {\em The Journal of Chemical Physics}, 14(3):180--201, 03 1946.

\bibitem[2]{Green1952}
{\sc H.~S. Green.}
\newblock {\em The molecular theory of fluids}.
\newblock North-Holland Publishing Co., Amsterdam; Interscience Publishers, Inc., New York, 1952.
\newblock Deformation and Flow: Monographs on the Rheological Behaviour of Natural and Synthetic Products.


\bibitem[3]{BDS99}
{\sc J.~J. Brey, J.~W. Dufty, \& A.~Santos.}
\newblock Kinetic models for granular flow.
\newblock {\em J. Statist. Phys.}, 97(1-2):281--322, 1999.

\bibitem[4]{Bou93}
{\sc  F.~c. Bouchut.}
\newblock Existence and uniqueness of a global smooth solution for the
  {V}lasov-{P}oisson-{F}okker-{P}lanck system in three dimensions.
\newblock {\em J. Funct. Anal.}, 111(1):239--258, 1993.

\bibitem[5]{Deg86}
{\sc P.~Degond.}
\newblock Global existence of smooth solutions for the
  {V}lasov-{F}okker-{P}lanck equation in {$1$} and {$2$} space dimensions.
\newblock {\em Ann. Sci. \'Ecole Norm. Sup. (4)}, 19(4):519--542, 1986.

\bibitem[6]{HJ13}
{\sc H.~J. Hwang \& J.~Jang.}
\newblock On the {V}lasov-{P}oisson-{F}okker-{P}lanck equation near
  {M}axwellian.
\newblock {\em Discrete Contin. Dyn. Syst. Ser. B}, 18(3):681--691, 2013.

\bibitem[7]{RW92}
{\sc G.~Rein \& J.~Weckler.}
\newblock Generic global classical solutions of the
  {V}lasov-{F}okker-{P}lanck-{P}oisson system in three dimensions.
\newblock {\em J. Differential Equations}, 99(1):59--77, 1992.

\bibitem[8]{VO90}
{\sc H.~D. Victory, Jr. \& B.~P. O'Dwyer.}
\newblock On classical solutions of {V}lasov-{P}oisson {F}okker-{P}lanck
  systems.
\newblock {\em Indiana Univ. Math. J.}, 39(1):105--156, 1990.

\bibitem[9]{CS95}
{\sc J.~A. Carrillo \& J.~Soler.}
\newblock On the initial value problem for the
  {V}lasov-{P}oisson-{F}okker-{P}lanck system with initial data in {$L^p$}
  spaces.
\newblock {\em Math. Methods Appl. Sci.}, 18(10):825--839, 1995.

\bibitem[10]{CS97}
{\sc J.~A. Carrillo \& J.~Soler.}
\newblock On the {V}lasov-{P}oisson-{F}okker-{P}lanck equations with measures
  in {M}orrey spaces as initial data.
\newblock {\em J. Math. Anal. Appl.}, 207(2):475--495, 1997.

\bibitem[11]{Cas98}
{\sc F.~Castella.}
\newblock The {V}lasov-{P}oisson-{F}okker-{P}lanck system with infinite kinetic
  energy.
\newblock {\em Indiana Univ. Math. J.}, 47(3):939--964, 1998.



\bibitem[12]{Vic91}
{\sc H.~D. Victory, Jr.}
\newblock On the existence of global weak solutions for
  {V}lasov-{P}oisson-{F}okker-{P}lanck systems.
\newblock {\em J. Math. Anal. Appl.}, 160(2):525--555, 1991.





\bibitem[13]{Bou95}
{\sc  F.~c. Bouchut.}
\newblock Smoothing effect for the non-linear
  {V}lasov-{P}oisson-{F}okker-{P}lanck system.
\newblock {\em J. Differential Equations}, 122(2):225--238, 1995.

\bibitem[14]{BCS97}
{\sc L.~L. Bonilla, J.~A. Carrillo, \& J.~Soler.}
\newblock Asymptotic behavior of an initial-boundary value problem for the
  {V}lasov-{P}oisson-{F}okker-{P}lanck system.
\newblock {\em SIAM J. Appl. Math.}, 57(5):1343--1372, 1997.

\bibitem[15]{Carrillo98}
{\sc J.~A. Carrillo.}
\newblock Global weak solutions for the initial-boundary-value problems to the {V}lasov-{P}oisson-{F}okker-{P}lanck system.
\newblock {\em Math. Methods Appl. Sci.}, 21(10):907--938, 1998.

\bibitem[16]{HK19}
{\sc H.~J. Hwang \& J.~Kim.}
\newblock The {V}lasov-{P}oisson-{F}okker-{P}lanck equation in an interval with
  kinetic absorbing boundary conditions.
\newblock {\em Stochastic Process. Appl.}, 129(1):240--282, 2019.

\bibitem[17]{MH22}
{\sc X.~Ma \& F.~He.}
\newblock The initial boundary value problem for the
  {V}lasov-{P}oisson-{F}okker-{P}lanck system.
\newblock {\em J. Math. Phys.}, 63(9):Paper No. 091506, 15, 2022.

\bibitem[18]{Mischler2010}
{\sc S.~Mischler.}
\newblock Kinetic equations with {M}axwell boundary conditions.
\newblock {\em Ann. Sci. \'{E}c. Norm. Sup\'{e}r. (4)}, 43(5):719--760, 2010.

\bibitem[19]{Bed17}
{\sc J.~Bedrossian.}
\newblock Suppression of plasma echoes and {L}andau damping in {S}obolev spaces by weak collisions in a {V}lasov-{F}okker-{P}lanck equation.
\newblock {\em Ann. PDE}, 3(2):Paper No. 19, 66, 2017.

\bibitem[20]{LY21}
{\sc J.~Liao \& X.~Yang.}
\newblock Stability of global {M}axwellian for fully nonlinear {F}okker-{P}lanck equations.
\newblock {\em J. Stat. Phys.}, 185(3):Paper No. 23, 27, 2021.

\bibitem[21]{MM17}
{\sc J.~Mathiaud \& L.~Mieussens.}
\newblock A {F}okker-{P}lanck model of the {B}oltzmann equation with correct {P}randtl number for polyatomic gases.
\newblock {\em J. Stat. Phys.}, 168(5):1031--1055, 2017.

\bibitem[22]{SJ19}
{\sc J.~Sun \& L.~Jing.}
\newblock Global existence and long time behavior of the ellipsoidal-{F}okker-{P}lanck equation.
\newblock {\em Appl. Anal.}, 98(9):1605--1625, 2019.

\bibitem[23]{AC23}
{\sc I.~Almuslimani \& N.~Crouseilles}.
\newblock Conservative stabilized {R}unge-{K}utta methods for the {V}lasov-{F}okker-{P}lanck equation.
\newblock {\em J. Comput. Phys.}, 488:Paper No. 112241, 21, 2023.

\bibitem[24]{FN22}
{\sc F.~Filbet \& C.~Negulescu.}
\newblock Fokker-{P}lanck multi-species equations in the adiabatic asymptotics.
\newblock {\em J. Comput. Phys.}, 471:Paper No. 111642, 28, 2022.

\bibitem[25]{DFT10}
{\sc R.~Duan, M.~Fornasier, \& G.~Toscani.}
\newblock A kinetic flocking model with diffusion. {\em Comm. Math. Phys.}, 300 (2010), no. 1, 95--145


\bibitem[26]{CG04}
{\sc E.~A. Carlen \& W.~Gangbo.}
\newblock Solution of a model {B}oltzmann equation via steepest descent in the 2-{W}asserstein metric.
\newblock {\em Arch. Ration. Mech. Anal.}, 172(1):21--64, 2004.


\bibitem[27]{CG03}
{\sc E.~A. Carlen \& W.~Gangbo.}
\newblock Constrained steepest descent in the 2-{W}asserstein metric.
\newblock {\em Ann. of Math. (2)}, 157(3):807--846, 2003.

\bibitem[28]{CHY25}
{\sc Y.-P. Choi, B.-H. Hwang, \& Y.~Yoo.}
\newblock Global existence of weak solutions to the nonlinear {V}lasov-{F}okker-{P}lanck equation.
\newblock {\em J. Differential Equations}, 444:Paper No. 113573, 53, 2025.

\bibitem[29]{H67}
{\sc L.~H\"{o}rmander.}
\newblock Hypoelliptic second order differential equations.
\newblock {\em Acta Math.}, 119:147--171, 1967.

\bibitem[30]{Bou02}
{\sc F.~Bouchut.}
\newblock Hypoelliptic regularity in kinetic equations.
\newblock {\em J. Math. Pures Appl.} (9) 81 (2002), no. 11, 1135--1159.

\bibitem[31]{HN04}
{\sc F.~H\'erau \& F.~Nier.}
\newblock Isotropic hypoellipticity and trend to equilibrium for the
  {F}okker-{P}lanck equation with a high-degree potential.
\newblock {\em Arch. Ration. Mech. Anal.}, 171(2):151--218, 2004.

\bibitem[32]{LL08}
{\sc C.~Le~Bris \& P.-L. Lions.}
\newblock Existence and uniqueness of solutions to {F}okker-{P}lanck type equations with irregular coefficients.
\newblock {\em Comm. Partial Differential Equations}, 33(7-9):1272--1317, 2008.

\bibitem[33]{DL88FKB}
{\sc R.~J. DiPerna \& P.-L. Lions.}
\newblock On the {F}okker-{P}lanck-{B}oltzmann equation.
\newblock {\em Comm. Math. Phys.}, 120(1):1--23, 1988.

\bibitem[34]{Zhu24}
{\sc Y.~Zhu.}
\newblock Regularity of kinetic {F}okker-{P}lanck equations in bounded domains.
\newblock {\em Ann. H. Lebesgue}, 7:1323--1366, 2024.


\bibitem[35]{DipernaLionsODE1989}
{\sc R.~J. DiPerna \& P.-L. Lions.}
\newblock Ordinary differential equations, transport theory and {S}obolev spaces.
\newblock {\em Invent. Math.}, 98(3):511--547, 1989.

\bibitem[36]{shvydkoy2025}
{\sc R.~Shvydkoy.}
\newblock Global well-posedness and relaxation for solutions of the {F}okker-{P}lanck-{A}lignment equations.
\newblock {\em Comm. Pure Appl. Math.}, e70049.

\bibitem[37]{KMT2015}
{\sc T.~K. Karper, A.~Mellet, \& K.~Trivisa.}
\newblock Existence of weak solutions to kinetic flocking models.
\newblock {\em SIAM J. Math. Anal.}, 45(1):215--243, 2013.

\bibitem[38]{Choi16}
{\sc Y.-P. Choi.}
\newblock Global classical solutions of the {V}lasov-{F}okker-{P}lanck equation with local alignment forces.
\newblock {\em Nonlinearity}, 29(7):1887--1916, 2016.



\bibitem[39]{CJ24}
{\sc Y.-P. Choi \& J.~Jung.}
\newblock Incompressible {N}avier-{S}tokes limit from nonlinear {V}lasov-{F}okker-{P}lanck equation.
\newblock {\em Appl. Math. Lett.}, 158:Paper No. 109214, 7, 2024.

\bibitem[40]{CJ26}
{\sc Y.-P. Choi \& J.~Jung.}
\newblock Incompressible {E}uler limits from a nonlinear {V}lasov-{F}okker-{P}lanck equation with constant temperature.
\newblock {\em Appl. Math. Lett.}, 172:Paper No. 109721, 6, 2026.

\bibitem[41]{ImbertMouhot21}
{\sc C.~Imbert \& C.~Mouhot.}
\newblock The {S}chauder estimate in kinetic theory with application to a toy nonlinear model.
\newblock {\em Ann. H. Lebesgue}, 4:369--405, 2021.



\bibitem[42]{AnceschiZhu24}
{\sc F.~Anceschi \& Y.~Zhu.}
\newblock On a spatially inhomogeneous nonlinear {F}okker-{P}lanck equation: {C}auchy problem and diffusion asymptotics.
\newblock {\em Anal. PDE}, 17(2):379--420, 2024.

\bibitem[43]{GHJO20}
{\sc Y.~Guo, H.~J. Hwang, J.~W. Jang, \& Z.~Ouyang.}
\newblock The {L}andau equation with the specular reflection boundary condition.
\newblock {\em Arch. Ration. Mech. Anal.}, 236(3):1389--1454, 2020.

\bibitem[44]{Hamdache92}
{\sc K.~Hamdache.}
\newblock Initial-boundary value problems for the {B}oltzmann equation: global existence of weak solutions.
\newblock {\em Arch. Rational Mech. Anal.}, 119(4):309--353, 1992.

\bibitem[45]{Abdallah1994}
{\sc N.~B. Abdallah}.
\newblock Weak solutions of the initial-boundary value problem for the {V}lasov-{P}oisson system.
\newblock {\em Math. Methods Appl. Sci.}, 17(6):451--476, 1994.


\bibitem[46]{GLPS88}
{\sc F.~Golse, P.-L. Lions, B.~Perthame, \& R.~Sentis.}
\newblock Regularity of the moments of the solution of a transport equation.
\newblock {\em J. Funct. Anal.}, 76(1):110--125, 1988.

\bibitem[47]{GS02}
{\sc F.~Golse \& L.~Saint-Raymond.}
\newblock Velocity averaging in {$L^1$} for the transport equation.
\newblock {\em C. R. Math. Acad. Sci. Paris}, 334(7):557--562, 2002.


\bibitem[48]{dipernalionsMaxwell1989}
{\sc R.~J. DiPerna \& P.-L. Lions.}
\newblock Global weak solutions of {V}lasov-{M}axwell systems.
\newblock {\em Comm. Pure Appl. Math.}, 42(6):729--757, 1989.



\bibitem[49]{GolseSaintRaymond05}
{\sc F.~Golse \& L.~Saint-Raymond.}
\newblock Hydrodynamic limits for the {B}oltzmann equation.
\newblock {\em Riv. Mat. Univ. Parma (7)}, 4**:1--144, 2005.

\bibitem[50]{Perthame1992}
{\sc P.-L. Lions \& B.~Perthame.}
\newblock Lemmes de moments, de moyenne et de dispersion.
\newblock {\em C. R. Acad. Sci. Paris S\'{e}r. I Math.}, 314(11):801--806, 1992.


\bibitem[51]{sampaio2024}
{\sc P.~Sampaio.}
\newblock Global solutions to the {L}andau-{F}ermi-{D}irac equation, preprint, arXiv:2410.12681.

\bibitem[52]{GPS85}
{\sc F.~Golse, B.~Perthame, \& R.~Sentis.}
\newblock Un r\'{e}sultat de compacit\'{e} pour les \'{e}quations de transport et application au calcul de la limite de la valeur propre principale d'un op\'{e}rateur de transport.
\newblock {\em C. R. Acad. Sci. Paris S\'{e}r. I Math.}, 301(7):341--344, 1985.

\bibitem[53]{Glassey96}
{\sc R.~T. Glassey.}
\newblock {\em The {C}auchy problem in kinetic theory}.
\newblock Society for Industrial and Applied Mathematics (SIAM), Philadelphia, PA, 1996.



\bibitem[54]{DunfordSchwartz}
{\sc N.~Dunford \& J.~T. Schwartz.}
\newblock {\em Linear operators. {P}art {I}}.
\newblock Wiley Classics Library. John Wiley \& Sons, Inc., New York, 1988.
\newblock General theory, With the assistance of William G. Bade and Robert G. Bartle, Reprint of the 1958 original, A Wiley-Interscience Publication.

\bibitem[55]{Mischler2000}
{\sc S.~Mischler.}
\newblock On the trace problem for solutions of the {V}lasov equation.
\newblock {\em Comm. Partial Differential Equations}, 25(7-8):1415--1443, 2000.

\bibitem[56]{BlouzaLeDret2001}
{\sc A.~Blouza \& H.~Le~Dret.}
\newblock An up-to-the-boundary version of {F}riedrichs's lemma and applications to the linear {K}oiter shell model.
\newblock {\em SIAM J. Math. Anal.}, 33(4):877--895, 2001.

\bibitem[57]{DipernaLionsMeyer}
{\sc R.~J. DiPerna, P.-L. Lions, \& Y.~Meyer.}
\newblock {$L^p$} regularity of velocity averages.
\newblock {\em Ann. Inst. H. Poincar\'{e} C Anal. Non Lin\'{e}aire}, 8(3-4):271--287, 1991.

\bibitem[58]{perthamesouganidis}
{\sc B.~Perthame \& P.~E. Souganidis.}
\newblock A limiting case for velocity averaging.
\newblock {\em Ann. Sci. \'{E}cole Norm. Sup. (4)}, 31(4):591--598, 1998.

\bibitem[59]{TT07}
{\sc E.~Tadmor \& T.~Tao.}
\newblock Velocity averaging, kinetic formulations, and regularizing effects in quasi-linear PDEs.
\newblock {\em Comm. Pure Appl. Math. 60} (2007), no. 10, 1488--1521.

\bibitem[60]{perthame1996}
{\sc B.~Perthame.}
\newblock Time decay, propagation of low moments and dispersive effects for kinetic equations.
\newblock {\em Comm. Partial Differential Equations}, 21(3-4):659--686, 1996.

\bibitem[61]{Perthame2004}
{\sc B.~Perthame.}
\newblock Mathematical tools for kinetic equations.
\newblock {\em Bull. Amer. Math. Soc. (N.S.)}, 41(2):205--244, 2004.


\bibitem[62]{ChenZhang2016}
{\sc Z.~Chen \& X.~Zhang.}
\newblock Global existence and uniqueness to the {C}auchy problem of the {BGK} equation with infinite energy.
\newblock {\em Math. Methods Appl. Sci.}, 39(11):3116--3135, 2016.



\bibitem[63]{ChoiYun20}
{\sc Y.-P. Choi \& S.-B. Yun.}
\newblock Global existence of weak solutions for {N}avier-{S}tokes-{BGK} system.
\newblock {\em Nonlinearity}, 33(4):1925--1955, 2020.


\bibitem[64]{JKO98}
{\sc R.~Jordan, D.~Kinderlehrer, \& F.~Otto.}
\newblock The variational formulation of the {F}okker-{P}lanck equation.
\newblock {\em SIAM J. Math. Anal.}, 29(1):1--17, 1998.




\bibitem[65]{DipernaLions89Annals}
{\sc R.~J. DiPerna \& P.-L. Lions.}
\newblock On the {C}auchy problem for {B}oltzmann equations: global existence and weak stability.
\newblock {\em Ann. of Math. (2)}, 130(2):321--366, 1989.



\bibitem[66]{Masmoudi10}
{\sc N.~El~Ghani \& N.~Masmoudi.}
\newblock Diffusion limit of the {V}lasov-{P}oisson-{F}okker-{P}lanck system.
\newblock {\em Commun. Math. Sci.}, 8(2):463--479, 2010.


\end{thebibliography}
\end{document}